\documentclass[11pt]{article}  
\usepackage{amsthm}

\makeatletter
\newtheorem*{rep@theorem}{\rep@title}
\newcommand{\newreptheorem}[2]{%
\newenvironment{rep#1}[1]{%
 \def\rep@title{#2 \ref{##1}}%
 \begin{rep@theorem}}%
 {\end{rep@theorem}}}
\makeatother

\usepackage{amssymb, amsmath}
\usepackage{verbatim}
\newtheorem{theorem}{Theorem}
\newtheorem{definition}{Definition}
\newtheorem{lemma}{Lemma}
 \newtheorem{prop}{Proposition}                                                                                                                                                                                                                                                     
 \newtheorem{corollary}{Corollary}
 \newtheorem{question}{Question}
 \newreptheorem{theorem}{Theorem}

\usepackage[margin=.5in,footskip=0.25in]{geometry}
\newgeometry{left=0.5in,right=0.5in,top=.6in,bottom=1in}
 \begin{document}
 \title{ Unboundedness Theorems for Symbols Adapted to Large Subspaces}
 \author{Robert M. Kesler}
 \maketitle

 \abstract{For every integer $n \geq 3$, we prove that the n-sublinear generalization of the Bi-Carleson operator of Muscalu,  Tao, and Thiele given by
 
 \begin{eqnarray*}
nC^{\vec{\alpha}} :(f_1,..., f_n) \mapsto  \sup_{M} \left| \int_{\vec{\xi} \cdot \vec{\alpha} >0, \xi_n < M} \left[\prod_{j=1}^n  \hat{f}_j(\xi_j) e^{2 \pi i x \xi_j }\right]d\vec{\xi} ~\right|
 \end{eqnarray*}
satisfies no $L^p$ estimates provided $\vec{\alpha} \in \mathbb{Q}^n$ with distinct, non-zero entries.  Furthermore, if $n \geq 5$ and $\vec{\alpha} \in \mathbb{Q}^n$ has distinct, non-zero entries, it is shown that there is a  symbol $m:\mathbb{R}^n \rightarrow \mathbb{C}$ adapted to the hyperplane $\Gamma^{\vec{a}}=\left\{ \vec{\xi} \in \mathbb{R}^n: \sum_{j=1}^n \xi_j \cdot a_j =0 \right\} $ and supported in $\left\{ \vec{\xi} : dist(\vec{\xi}, \Gamma^{\vec{\alpha}}) \lesssim 1 \right\}$ 
for which the associated $n$-linear multiplier  also satisfies no $L^p$ estimates. Next, we construct a H\"{o}rmander-Marcinkiewicz symbol $\Pi: \mathbb{R}^2 \rightarrow \mathbb{C}$, which is a paraproduct of $(\phi, \psi)$ type, such that the trilinear operator $T_m$ whose symbol $m$ is $ sgn(\xi_1 + \xi_2) \Pi(\xi_2, \xi_3)$ satisfies no $L^p$ estimates.   Finally, we state a converse to a theorem of Muscalu, Tao, and Thiele using Riesz kernels in the spirit of Muscalu's recent work: for every pair of integers $(\mathfrak{d},n) $ s.t. $ \frac{n}{2}+\frac{3}{2} \leq \mathfrak{d}<n$ there is an explicit collection $\mathfrak{C}$ of uncountably many $\mathfrak{d}$-dimensional non-degenerate subspaces of $\mathbb{R}^n$ such that for each $\Gamma \in \mathcal{C}$  there is an associated symbol $m_\Gamma$ adapted to $\Gamma$ in the Mikhlin-H\"{o}rmander sense and supported in $\left\{ \vec{\xi} : dist(\vec{\xi}, \Gamma) \lesssim 1 \right\}$ for which the associated multilinear multiplier $T_{m_\Gamma}$ is unbounded. }

\section{Introduction}
C. Muscalu, T. Tao, and C. Thiele prove in \cite{MR2221256} that a maximal variant of the $BHT$ called the Bi-Carleson operator  defined \emph{a priori}  for any pair of Schwartz functions $f_1, f_2$ by 
 
 \begin{eqnarray*}
BiC(f_1, f_2)(x)= \sup_{N \in \mathbb{R}} \left| \int_{\xi_1 < \xi_2< N} \hat{f}_1(\xi_1) \hat{f}_2(\xi_2) e^{2 \pi i x(\xi_1 + \xi_2)} d\xi_1 d\xi_2 \right|
 \end{eqnarray*}
 extends to a continuous map $L^{p_1}(\mathbb{R}) \times L^{p_2}(\mathbb{R}) \rightarrow L^{\frac{ p_1 p_2}{p_1 + p_2}} (\mathbb{R})$ for all $(p_1, p_2)$ such that $1<p_1,p_2 <\infty$, and $\frac{1}{p_1}+\frac{1}{p_2} <3/2$. While the Bi-Carleson is non-degenerate, it is well known that multi-(sub)linear Fourier multipliers with degenerate singularities may satisfy no $L^p$ estimates. To be precise, we recall the following two definitions: 
 
 \begin{definition}\cite{MR1887641}
A $\mathfrak{d}-$dimensional subspace $\Gamma \subset \mathbb{R}^n$ is said to be non-degenerate provided

\begin{eqnarray*}
\tilde{\Gamma} := \left\{ (\xi_1,..., \xi_{n+1}) : (\xi_1, ..., \xi_n) \in \Gamma , \sum_{j=1}^{n+1} \xi_j =0\right\} \subset \mathbb{R}^{n+1}
\end{eqnarray*}
is a graph over the variables $(\xi_{i_1}, ..., \xi_{i_\mathfrak{d}})$ for every chain $1 \leq i_1 < ...<i_\mathfrak{d} \leq n+1$.  A subspace $\Gamma \subset \mathbb{R}^n$ is said to be degenerate if it is not non-degenerate. 
\end{definition}

\begin{definition}
An $n$-(sub)linear operator $T$ defined \emph{a priori} on $\mathcal{S}(\mathbb{R})^n$ satisfies no $L^p$ estimates provided there does not exist any $n-$tuple $ (p_1, ..., p_n) \in [1, \infty]^n$ for which threre is a constant $C_{T,\vec{p}}$ such that 
\begin{eqnarray*}
\left|\left| T\left(\vec{f}\right)\right|\right|_{\frac{1}{\sum_{j=1}^n \frac{1}{p_j}}} \leq C_{T, \vec{p}} \prod_{j=1}^n || f_j||_{p_j}
\end{eqnarray*}
for all $n$-tuples of functions $(f_1, ..., f_n) \in \mathcal{S}(\mathbb{R})^n$.
\end{definition}

For example, \cite{MR1981900} shows that the degenerate operator
 \begin{eqnarray*}
\widetilde{BiC}(f_1, f_2)(x)= \sup_{M \in \mathbb{R}} \left| \int_{\xi_1 + \xi_2 <0, \xi_2 <M} \hat{f}_1(\xi_1) \hat{f}_2(\xi_2) e^{2 \pi i x (\xi_1 + \xi_2)} d\xi_1 d\xi_2 \right|
 \end{eqnarray*}
 satisfies no $L^p$ estimates even though the Hilbert transform of a product of functions $H_2: (f,g) \mapsto H(f \cdot g)$ is a multiplier with degenerate singularity $\{\xi_1 + \xi_2 =0\}$ that still maps into $L^p$ for $1 < p <\infty$. Degenerate simplex multipliers for which no $L^p$ estimates hold do nonetheless have weaker mixed estimates, see \cite{MR3286565}.  Our first result in this paper is to exhibit $L^p$ unboundedness results for the most natural generalization of $BiC$ to higher dimensions. Specifically, we prove

 \begin{theorem}\label{MT**}  Fix a dimension $n \geq 3$ along with $\vec{\alpha} \in \mathbb{R}^n$ satisfying $\alpha_j^{-1} = \alpha + q_j$ for some $\vec{q} \in \mathbb{Q}^n$ and $\alpha \in \mathbb{R}$ such that $q_i \not = q_j$ whenever $ i \not = j$ and $ q_j \not =-\alpha$ for all $1 \leq j \leq n$. Then the operator defined a priori on $\mathcal{S}(\mathbb{R})^n$ by the formula 
 
 \begin{eqnarray*}
 nC^{\vec{\alpha}} (\vec{f})(x) = \sup_{M \in \mathbb{R}} \left| \int_{\vec{\xi} \cdot \vec{\alpha} >0, \xi_n < M} \left[ \prod_{j=1}^n \hat{f}_j(\xi_j) e^{2 \pi i x \xi_j }\right] \xi_1 ... d\xi_n \right|.
 \end{eqnarray*}
 satisfies no $L^p$ estimates. 
\end{theorem}

Because of \cite{MR3294566}, this result is new only in $3$ dimensions. Moreover, the most natural generalization of $BHT$ to higher dimensions is the $n-$linear Hilbert transform, for which negative results have already been obtained when $n=3$ by C. Demeter in \cite{MR2449537} for target  exponent below $\frac{1}{3}\left( 1+ \frac{\log_6 2}{1+\log_62}\right)$. Our next result is 

\begin{theorem}\label{IT}
There exists a H\"{o}rmander-Marcinkiewicz symbol $a: \mathbb{R}^2 \rightarrow \mathbb{R}$ of $(\phi, \psi)$ type, i.e. $a$ is sum of tensor products that are $\phi$ type in the first index and $\psi$ type in the second index satisfying $|\partial^{\vec{\alpha}} a(\xi)| \leq \frac{C_{\vec{\alpha}}}{dist(\vec{\xi}, \vec{0})^{|\vec{\alpha}|}}$, such that the trilinear operator $T_m$ whose symbol $m$ is given by $m(\xi_1, \xi_2, \xi_3) = sgn(\xi_1 + \xi_2) a(\xi_2, \xi_3)$ satisfies no $L^p$ estimates. 
\end{theorem}

To make sense of the remaining statements, we shall need
 
 \begin{definition}\cite{MR3294566}
 For each subspace $\Gamma \subset \mathbb{R}^d$ let 
 \begin{eqnarray*}
\mathcal{M}_\Gamma(\mathbb{R}^d) = \left\{ m : \mathbb{R}^d \rightarrow \mathbb{R} : \forall~\vec{\alpha} \in (\mathbb{N} \cup\{0\})^d~\exists C_{\vec{\alpha}}~such~that~ \left| \partial ^{\vec{\alpha}} m (\vec{\xi})\right| \leq\frac{C_{\vec{\alpha}}}{dist(\vec{\xi}, \Gamma)^{|\vec{\alpha}|}} \right\}.
\end{eqnarray*}
 \end{definition}
 
 \begin{definition}
 For $m \in L^\infty(\mathbb{R}^d)$ let $T_m$ be the operator defined \emph{a priori} on d-tuples of Schwartz functions $\vec{f} =(f_1, \cdots, f_d)$ by the formula
 
 \begin{eqnarray*}
  T_{m}(\vec{f})(x):= \int_{\mathbb{R}^d} m(\vec{\xi})\left[ \prod_{j=1}^d \hat{f}_j(\xi_j) e^{2 \pi i x  \xi_j} \right] d\vec{\xi}.
 \end{eqnarray*}
 
 \end{definition}
The following is already known:
 
 \begin{theorem}[\cite{MR3294566}]\label{AT1}
For any two generic non-degenerate subspaces $\Gamma_1, \Gamma_2 \subseteq  \mathbb{R}^n$ of maximal dimension $n-1$, there are symbol $m_1 \in \mathcal{M}_{\Gamma_1}(\mathbb{R}^n)$ and  $m_2 \in \mathcal{M}_{\Gamma_2}(\mathbb{R}^n)$ for which at least one of the associated n-linear operators $T_{m_1}, T_{m_2}$ does not  satisfy any $L^p$ estimates.  
 \end{theorem}

Our next result gives an uncountable collection of subspaces with adapted symbols satisfying no $L^p$ estimates:
 
 \begin{theorem}\label{MT*} 
Let $n \geq 5$ and $\vec{\alpha} \in \mathbb{R}^n$ satisfy $\alpha_j^{-1} = q_j + \alpha q_j^2$ for some $\vec{q} \in \mathbb{Q}^n$ with distinct, non-zero entires such that $q_j \alpha \not = -1$ for all $1 \leq j \leq n$. Then there exists $m \in \mathcal{M}_{\Gamma^{\vec{\alpha}}}(\mathbb{R}^n)$ such that $T_m$ satisfies no $L^p$ estimates. 
\end{theorem}
This statement should be viewed as a microlocal version of one of the main results in \cite{MR3294566} that the intersection of two generic half-spaces does not produce a bounded operator on any $L^p$ space. Indeed, unboundedness of multipliers given by intersections of generic half-spaces implies that for a given $\vec{p}$ there is a symbol adapted in the Mikhlin-H\"{o}rmander sense either to one hyperplane or a symbol adapted to the other hyperplane for which at least one of the two multipliers satisfies no $L^p$ estimates, i.e. 
\begin{eqnarray*}
T_{1_{\Gamma^{\vec{\alpha}} \cap \Gamma^{\vec{\beta}}}}= T_{m_1} + T_{m_2},
\end{eqnarray*}
 where  $m_1 \in \mathcal{M}_{\Gamma^{\vec{\alpha}}} (\mathbb{R}^d)$ and $m_2 \in \mathcal{M}_{\Gamma^{\vec{\beta}}}(\mathbb{R}^d)$. Of course, if $T$ does not satisfy any $L^p$ estimates than either $T_{m_1}$ or $T_{m_2}$ does not satisfy estimates for a given $\vec{p} = (p_1, ..., p_d)$.  In fact, C. Muscalu's arguments show that at least one of $T_{m_1}$ and $T_{m_2}$ must satisfy no $L^p$ estimates. Hence, the main benefit of our calculations is to exhibit unbounded hyperplane-adapted multipliers for an explicit collection of uncountably many hyperplanes.

C. Muscalu strengthens Theorem \ref{AT1} in \cite{MR3294566} by showing that generic subspaces of codimension smaller than around $\sqrt{n}$ also satisfy no $L^p$ estimates:
 \begin{theorem}[\cite{MR3294566}]
 Let $n \geq 5$. For any $1 < p_1,..., p_n \leq \infty$ and $0 < p <\infty$ with $1/p_1 + ... + 1/p_n = 1/p$, for any integer $\mathfrak{d}$ satisfying 
 
 \begin{eqnarray*}
n - \left( \frac{n-1}{2} \right)^{\frac{1}{2}}< \mathfrak{d} \leq n-1 ,
 \end{eqnarray*}
and for any two non-degenerate subspaces $\Gamma_1, \Gamma_2 \subset \mathbb{R}^n$ with $dim(\Gamma_1)=dim(\Gamma_2)=\mathfrak{d}$, there are symbols $m_1 \in \mathcal{M}_{\Gamma_1} (\mathbb{R}^n)$ and $m_2 \in \mathcal{M}_{\Gamma_2}(\mathbb{R}^n)$ for which at least one of the associated $n-$linear operators $T_{m_1}, T_{m_2}$ do not map $L^{p_1} \times \cdots \times L^{p_n}$ into $L^p$.  
 \end{theorem}

In the positive direction, C. Muscalu et al.  have established $L^p$ estimates for Mikhlin-H\"{o}rmander multipliers adapted to non-degenerate subspace singularities of dimension at most roughly half of the total dimension; among other results, they have
\begin{theorem}[\cite{MR1887641}]\label{OT}
Let $\Gamma \subset \mathbb{R}^n$ be a non-degenerate subspace of dimension $\mathfrak{d}$ where 

\begin{eqnarray*}
0 \leq \mathfrak{d} < \frac{n+1}{2}
\end{eqnarray*}
and furthermore suppose $m \in \mathcal{M}_\Gamma(\mathbb{R}^n)$.  Then $T_m$ maps $L^{p_1}(\mathbb{R}) \times \cdots \times L^{p_n}(\mathbb{R}) \rightarrow L^{p_{n+1}}(\mathbb{R})$ for all $\sum_{j=1}^{n+1} \frac{1}{p_j} =1$ and $1 \leq p_j <\infty$ for all $j \in \{1, ..., n+1\}$. 

\end{theorem}
Our last theorem should be viewed as a converse to Theorem \ref{OT}. Adopting the strategy of \cite{MR3294566}, we use Riesz kernels to prove the existence of uncountably many non-degenerate subspaces of dimension roughly half of the total spatial dimension such that for each subspace there is a Mikhlin-H\"{o}rmander multiplier adapted to that subspace that satisfies no $L^p$ estimates. We have the following formal statement:
\begin{theorem}\label{MT}
Let  $n, \mathfrak{d} \in \mathbb{N}$ satisfy $\frac{n+3}{2} \leq \mathfrak{d} <n$ and $n \geq 5$.  Then there is an uncountable collection $\mathfrak{C}$ of $\mathfrak{d}-$ dimensional non-degenerate subspace $\Gamma \subset \mathbb{R}^n$ such that for each $\Gamma \in \mathfrak{C}$ there is an associated symbol $m_{\Gamma}$ adapted to $\Gamma$ in the Mikhlin-Hormander sense for which the associated multilinear multiplier $T_{m_{\Gamma}}$ is unbounded. 
\end{theorem} 

\subsection{Open Questions}
These theorems leave open the question of $L^p$ estimates for Mikhlin-H\"{o}rmander multipliers with non-degenerate singularities of dimension $\mathfrak{d} \in [\frac{n+1}{2}, \frac{n+3}{2})$. Hence, for each $n \geq 3$ there is a unique dimension $\mathfrak{d}(n)= \lceil \frac{n+1}{2} \rceil$ for which $n-$linear Mikhlin-H\"{o}rmander multipliers adapted to singularities of dimension $\mathfrak{d}(n)$ do not have guaranteed $L^p$ estimates via Theorem \ref{OT} and for which Theorem \ref{MT} provides no multiplier counterexamples. It is therefore likely that new ideas are required to understand the behavior of subspace-adapted multipliers in this range.

Furthermore, our arguments only show the existence of particular subspaces of small dimension for which there exist unbounded multipliers. It is easy to see that for a \emph{generic} choice of non-degenerate subspace no counterexamples can be constructed using our methods with codimension larger than roughly the square root of the total space dimension.

\begin{question}
Can one prove Theorems \ref{MT**} and \ref{MT*} only assuming $\vec{\alpha} \in \mathbb{R}^n$ with distinct non-zero entries?
\end{question}

\begin{question}\label{Q1}
For a given $n \geq 3$ is there a non-degenerate subspace $\Gamma \subset \mathbb{R}^n$ of dimension $\lceil \frac{n+1}{2} \rceil$ and Mikhlin-H\"{o}rmander multiplier $m: \mathbb{R}^n \rightarrow \mathbb{R}$ adapted to $\Gamma$ for which no $L^p$ estimates are satisfied?
\end{question}

\begin{question}
If the answer to Question \ref{Q1} is yes, can one produce generic counterexamples having singularities of the smallest possible size? That is, for a given $n \geq 3$ and \emph{generic} choice of non-degenerate subspace $\Gamma$ of dimension $\lceil \frac{n+1}{2} \rceil$ can one construct a Mikhlin-H\"{o}rmander multiplier $m: \mathbb{R}^n \rightarrow \mathbb{R}$ adapted to $\Gamma$ for which no $L^p$ estimates are satisfied?
\end{question}

\subsection{Method of Proof and Organization}
The Gaussian chirps appearing throughout this paper were originally devised to show unboundedness for particular multilinear multipliers related to AKNS expansions in \cite{MR1981900} and were later applied in \cite{MR3294566} to construct unbounded multipliers given by the intersection of two generic hyperplanes and to generate unbounded multipliers with singularities of small codimension using Riesz transforms. In addition, a lower bound on the size of Bohr sets will enable us to extend counterexamples featuring hyperplanes with normal vectors in $ \mathbb{R}(\mathbb{Q}^n)$ to the non-trivial irrational setting. 

The organization of the paper is as follows: 

\S{2} proves Theorem \ref{MT**}.

\S{3} proves Theorem \ref{IT}.

\S{4} proves Theorem \ref{MT*}.

\S{5} proves Theorem \ref{MT}. 

 \subsection{Notation}
In an abuse of notation, the author has not used the principal value symbol $p.v.$ for the many singular integrals appearing throughout this paper. Where necessary, the reader should always take such integrals in the principal value sense.  Moreover, one should interpret $\lesssim$ to mean $\leq C$ where the constant $C$ depends on inessential parameters, which should be clear from context. The symbol $\mathcal{S}(\mathbb{R})$ will always denote the collection of Schwartz functions on the real line.

 \section{Higher-Dimensional Generalizations of the Bi-Carleson Operator}
 
 \begin{reptheorem}{MT**} Fix a dimension $n \geq 3$ along with $\vec{\alpha} \in \mathbb{R}^n$ satisfying $\alpha_j^{-1} = \alpha + q_j$ for some $\vec{q} \in \mathbb{Q}^n$ and $\alpha \in \mathbb{R}$ such that $q_i \not = q_j$ whenever $ i \not = j$ and $ q_j \not =-\alpha$ for all $1 \leq j \leq n$. Then the operator defined a priori on $\mathcal{S}(\mathbb{R})^n$ by the formula 
 
 \begin{eqnarray*}
 dC^{\vec{\alpha}} (\vec{f})(x) = \sup_{M \in \mathbb{R}} \left| \int_{\vec{\xi} \cdot \vec{\alpha} >0, \xi_n < M} \left[ \prod_{j=1}^n \hat{f}_j(\xi_j) e^{2 \pi i x \xi_j }\right] \xi_1 ... d\xi_n \right|.
 \end{eqnarray*}
 satisfies no $L^p$ estimates. 
\end{reptheorem}
Setting $\alpha=0$, we obtain

\begin{corollary}
Let $n \geq 3$ and $\vec{q} \in \mathbb{Q}^n$ with distinct non-zero entires. Then the operator defined a priori on $\mathcal{S}(\mathbb{R})^n$ by the formula 
 
 \begin{eqnarray*}
 dC^{\vec{\alpha}} (\vec{f})(x) = \sup_{M \in \mathbb{R}} \left| \int_{\vec{\xi} \cdot \vec{\alpha} >0, \xi_n < M} \left[ \prod_{j=1}^n \hat{f}_j(\xi_j) e^{2 \pi i x \xi_j }\right] \xi_1 ... d\xi_n \right|.
 \end{eqnarray*}
 satisfies no $L^p$ estimates. 

\end{corollary}

To prove Theorem \ref{MT**}, we shall use an elementary fact concerning the existence of orthogonal rational vectors:

\begin{lemma}\label{IRL}
Fix $n \geq 3$. Suppose  $\vec{\alpha} \in \mathbb{R}^n$ satisfies for all $j \in \{1, 2, 3\}$

\begin{eqnarray*}
\alpha_j^{-1} =\alpha + q_j 
\end{eqnarray*}
for some $\alpha \in \mathbb{R} $ and  $\vec{q} \in \mathbb{Q}^n$ with distinct entries such that $q_j \not = -\alpha$ for all $1 \leq j \leq n$. Then there exists a non-trivial solution $\vec{\#} \in \mathbb{R}^n$ to the system 

\begin{eqnarray*}
\sum_{j=1}^n \#_j \alpha _j = \sum_{j=1}^n \#_j \alpha_j^2 = 0
\end{eqnarray*}
with the additional property that $\alpha^2_j \#_j \in \mathbb{Q}$ for all $1 \leq j \leq n$.

\end{lemma}
\begin{proof}
Introduce $\tilde{\#}_j=\#_j  \alpha^2_j $. Then we want to show that there exists $\vec{\tilde{\#}} \in \mathbb{Q}^n$ such that 

\begin{eqnarray*}
\sum_{j=1}^n \tilde{\#}_j =\sum_{j=1}^n \tilde{\#}_j \alpha_j^{-1} &=& 0.
\end{eqnarray*}
Moreover, $\vec{\tilde{\#}} \in \mathbb{Q}^n$ is a solution iff it is orthogonal to $\vec{1}$ and $\vec{\alpha}^{-1}$. Hence, if  $\#_j =0$ for $3 < j \leq n$, any 3-tuple  $(\#_1, \#_2, \#_3)$ parallel to the symbolic determinant

 \[\det 
  \begin{bmatrix}
    \vec{e}_1 & \vec{e}_2& \vec{e}_3\\ 
   1 &1 & 1 \\
   \alpha_1^{-1}& \alpha_2^{-1}& \alpha_3^{-1} 
  \end{bmatrix}
= \det  \begin{bmatrix}
    \vec{e}_1 & \vec{e}_2& \vec{e}_3\\ 
   1 &1 & 1 \\
q_1 & q_2&q_3 
  \end{bmatrix} \in \mathbb{Q}^3  \] 
  is a solution.

\end{proof}
Note that if $\alpha \in \mathbb{R} \cap \mathbb{Q}^c$, then $\vec{\alpha} \not \in \mathbb{R}(\mathbb{Q}^n)$. We shall also need to include the following brief detour: 
\subsection{Bohr Sets}

Fix $S \subset \mathbb{R}, 0<\rho \leq \frac{1}{2}$, and $N \in \mathbb{N}$. Define the Bohr set $ Bohr_N(S, \rho)$ by
 \begin{eqnarray*}
 Bohr_N(S, \rho) := \left\{ n \in \mathbb{Z} \cap [1, N]: \sup_{\xi \in S} || \xi \cdot n ||_{\mathbb{R} / \mathbb{Z}} < \rho \right\}.
 \end{eqnarray*}

  \begin{lemma}\label{ML***}
Let $S \subset \mathbb{R}$ with $|S| < \infty, \rho \in (0, \frac{1}{2}],$ and $N \in \mathbb{N}$.
 Then $\left| Bohr_N(S,  \rho)) \right| \geq N \rho^{|S|} -1$. 
 \end{lemma}

 \begin{proof}
 The proof is a straightforward adaptation of Lemma 4.22 from \emph{Additive Combinatorics} by Terry Tao and Van Vu \cite{MR2573797}. Letting $\mathcal{L}^{|S|}$ denote $|S|-$dimensional Lebesgue measure on $\mathbb{T}^{|S|}$ and $\{ \xi_1, ..., \xi_{|S|}\}$ be an enumeration of the elements in $S$, we have for all $n \in \{ 1, ..., N\}$
 
 \begin{eqnarray*}
 \mathcal{L}^{|S|}  \left\{ \vec{\theta} \in \mathbb{T}^{|S|} : || \xi_i \cdot n -\theta_i ||_{\mathbb{R} / \mathbb{Z}} < \rho~\forall~i\in \{1, ..., |S|\} \right\}  = 2^{|S|}\rho^{|S|}.
 \end{eqnarray*}
It follows that
 
 \begin{eqnarray*}
 N2^{|S|} \rho^{|S|} &=&\sum_{n =1}^N  \mathcal{L}^{|S|} \left| \left\{ \vec{\theta} : || \xi_i \cdot n -\theta_i ||_{\mathbb{R} / \mathbb{Z}} < \rho~\forall~i\in \{1, ..., n\} \right\} \right|\\ &=& 
 \int_{\mathbb{T}^{|S|}} \sum_{n=1}^N 1_{\left\{ \vec{\theta} \in \mathbb{T}^{|S|} : || \xi_i \cdot n -\theta_i ||_{\mathbb{R} / \mathbb{Z}} < \rho~\forall~i\in \{1, ..., |S|\} \right\} }(\vec{\theta}) d\vec{\theta}.
 \end{eqnarray*}
 Therefore, there exists $\vec{\theta}_* \in \mathbb{T}^n$ for which 
 
 \begin{eqnarray*}
 \sum_{n=1}^N 1_{\left\{ \vec{\theta} \in \mathbb{T}^{|S|} : || \xi_i \cdot n -\theta_i ||_{\mathbb{R} / \mathbb{Z}} < \rho~\forall~i\in \{1, ..., |S|\} \right\} }(\vec{\theta}_*) \geq N 2^{|S|} \rho^{|S|}.
 \end{eqnarray*}
Then $S_N:=\left\{ n \in \{1, ..., N\}:  || \xi_i \cdot n -\theta_{*, i} ||_{\mathbb{R} / \mathbb{Z}} < \rho~\forall~i\in \{1, ..., |S|\} \right\}$ satisfies $|S_N| \geq N2^{|S|} \rho^{|S|}$.  However, for every $(n_1, n_2) \in S_N\times S_N$, the triangle inequality yields

\begin{eqnarray*}
|| \xi _i \cdot (n_1 - n_2) ||_{\mathbb{R} / \mathbb{Z}} < 2 \rho~\forall~ i \in \{1, ..., |S|\}.
\end{eqnarray*}
Hence,  $\# \left\{ n \in [1,N] \cap \mathbb{Z}: \sup_{\xi \in S} || \xi \cdot n ||_{\mathbb{R}/ \mathbb{Z}} <2 \rho~\right\} \geq  N 2^{|S|}\rho^{|S|} -1$. Let $\rho \mapsto \rho/2$.  
 \end{proof}

Fix $S \subset \mathbb{R}, 0<\rho \leq \frac{1}{2}$, and $N \in \mathbb{N}$. Define the Bohr set $ Bohr_N(S, \rho)$ by
 \begin{eqnarray*}
 Bohr_N(S, \rho) := \left\{ n \in \mathbb{Z} \cap [1, N]: \sup_{\xi \in S} || \xi \cdot n ||_{\mathbb{R} / \mathbb{Z}} < \rho \right\}.
 \end{eqnarray*}

\begin{proof}

\subsection{PART I: The Rational Case}

Assume $\vec{\alpha} \in \mathbb{Q}^n$ so that $\alpha=0$. Then dilate $\vec{\alpha}$ by some suitably large integer to ensure $\vec{\alpha} \in \mathbb{Z}^n$ and assume WLOG that $\alpha_n >0$. Indeed, the proof will easily be seen to hold with minor adjustments for the case $\alpha_d <0$.   Let $\phi \in \mathcal{S}(\mathbb{R})$ satisfy $\phi \geq 0$, $\phi(0)\not =0$, and have compact Fourier support in $[-\frac{1}{2},\frac{1}{2}]$. Fix $N \in \mathbb{N}$ and let $A \in \mathbb{Z}$ be chosen independent of N and sufficiently large. What sufficiently large means will be determined later.  Construct for each $1 \leq j \leq n$ the function

\begin{eqnarray*}
f^{N, A, \#}(x) := \sum_{-N \leq m \leq N} \phi(x-Am) e^{2 \pi i A \# m x} =: \sum_{-N \leq m \leq N} f_{m}^{N,A, \#}(x),
\end{eqnarray*}
where we choose $\vec{\#} \in \mathbb{Z}^n$ such that $\vec{\#} \cdot \vec{\alpha} = \vec{\#} \cdot \vec{\alpha^2}=0$ and $\#_n >0$. [Here, $\vec{\alpha^2}  =\vec{\alpha} \wedge \vec{\alpha}:= (\alpha_1^2, ..., \alpha_n^2)$.]  One may always choose $\vec{\#}$ satisfying the above conditions because of our assumptions on $\vec{\alpha}$. So for each $n_0 \in [-N/3, N/3]$ fix $x \in \left[An_0, An_0+\frac{c_{\vec{\alpha}}}{A}\right]$  for some $c_{\vec{\alpha}}<<1$ to be determined and set $M(n_0) = A \#_nn_0+5\#_n$. Inserting $\vec{f}^{N, A, \vec{\#}}$ yields 

\begin{eqnarray*}
nC^{\vec{\alpha}}(\vec{f}^{N, A, \vec{\#}})(x) &\geq&\left| \sum_{-N \leq m_1, ..., m_n \leq N} \int_{\vec{\xi} \cdot \vec{\alpha} >0, \xi_n < M(n_0)}\left[ \prod_{j=1}^n  \mathcal{F} \left[  \phi(\cdot-Am_j) e^{2 \pi i A \#_j m_j \cdot} \right] (\xi_j)\right]e^{2 \pi i x(\sum_{j=1}^n \xi_j )}d\vec{\xi} \right| \\ &=&\left| \sum_{-N\leq m_1, ..., m_n \leq N}  T_{\vec{\alpha}, n_0} \left( \left\{ f_{m_j}^{N, A, \$_j} \right\}_{j=1}^n\right)(x) \right|.
\end{eqnarray*}
 The frequency restriction $\xi_n < M(n)$ combined with the compact Fourier support of $\phi$ ensures that all and only those terms $\vec{m}$ corresponding to $ m_n \leq n_0$ contribute non-zero summands.  Discretizing the kernel representation of our operator therefore yields

\begin{eqnarray*}
\sum_{k \in \mathbb{Z}}~~ \sum_{-N \leq m_1, ..., m_n \leq N: m_n \leq n_0}  \int_{A(k-\frac{1}{2})} ^{A(k+\frac{1}{2})}  \prod_{j=1}^n \left[  \phi(x-Am_j -\alpha_j t) e^{2 \pi i A \#_j m_j (x -\alpha_j t)} \right] \frac{dt}{t}.
\end{eqnarray*}
The main contribution for fixed $k$ derives from the terms $\vec{m}$ for which $n_0-m_j(n_0, k) -\alpha_j k =0$ for each $j \in \{1, ..., n\}$. 
Whenever this is the case, however, use the conditions $\vec{\#} \cdot \vec{\alpha} = \vec{\#}\cdot \vec{\alpha^2}=0$ to arrive at a lower bound 
\begin{eqnarray*}
 \left.  Re \left[ e^{-2 \pi i (\sum_{j=1}^n \#_j) n_0 x} \int_{A(k-\frac{1}{2})}^{A(k+\frac{1}{2})}  \prod_{j=1}^n \left[ \phi(x-Am_j -\alpha_j t)  e^{2 \pi iA \#_j m_j(x-\alpha_j t)} \right]\frac{dt}{t} \right] \right|_{m_j = n_0-\alpha_jk } \gtrsim \frac{1}{Ak}. 
\end{eqnarray*}
 For fixed $k \in \mathbb{Z}$, of course it may be the case that no such vector $\vec{m}$ satisfies the desired inequality. There are two possible reasons for this: either $k < 0$ so the additional constant $m_n \leq n_0$ must be violated, or there exists some coordinate $j: 1 \leq j \leq n$ satisfying $n_0 - \alpha_j k\not \in [1, N]$. However, if $k  \geq 0$ and for each $j: 1 \leq j \leq n$ the condition $n_0 - \alpha _j k \in [1,N]$ is satisfied, then $\exists \vec{m} \in \left[  [1, N] \cap \mathbb{N} \right]^n $ such that $n_0 - m_j - \alpha _j k=0$ for all $j \in \{1, ..., n\}$ and the desired bound is found. In particular, suppose we further impose the condition $n_0 \in [N/2, 2N/3]$. Then for all $k  \in\left [1, \frac{ N/3}{\max_{1 \leq j \leq n} \{ | \alpha_j |\}}\right] \cap \mathbb{N}$ such a vector $\vec{m}$ exists.

\begin{lemma}\label{ML}
To prove Theorem \ref{MT}, it suffices to show $\exists c_{\vec{\alpha}}>0$ such that for every $x \in \bigcup_{n_0 \in \mathbb{Z}\cap [ N/2, N] }[An_0, An_0 + \frac{c_{\vec{\alpha}}}{ A}]$ and $k\in \left[1, \frac{N/3}{\max_{1 \leq j \leq n} \{ | \alpha_j |\}}\right]$
 
 \begin{eqnarray*}
 T^k _{\vec{\alpha},n_0} \left( \left\{ f^{N, A, \#_j} \right\}_{j=1}^n \right) (x) := \sum_{-N \leq m_1, ..., m_n \leq N: m_n \leq n_0}  \int_{A(k-\frac{1}{2})} ^{A(k+\frac{1}{2})}  \prod_{j=1}^n \phi(x-Am_j -\alpha_j t) e^{2 \pi iA \#_j m_j (x -\alpha_j t)} \frac{dt}{t} 
 \end{eqnarray*}
  satisfies 
  \begin{eqnarray*}
  \left| Re\left[  e^{ - 2 \pi i A (\sum_{j=1}^n \#_j) n_0 x} T^k \left( \left\{ f^{N, A, \#_j} \right\}_{j=1}^n \right) (x)\right] \right| \gtrsim \frac{1}{Ak}.
  \end{eqnarray*} 
 \end{lemma}
 
 \begin{proof}
 The first claim is that $T^k \left( \left\{ f^{N, A, \#_j} \right\}_{j=1}^n \right)$ decays rapidly in $|k|$ when $k \not \in [1,N]$. For $k \in [1,N] \cap \left[1, \frac{N/3}{\max_{1 \leq j \leq n} \{ | \alpha_j |\}}\right]^c$, we may content ourselves with the cheapest possible upper bound
 
 \begin{eqnarray*}
 \left| T^k(f_1^N, ..., f_n^N)(x) \right| \lesssim \frac{1}{Ak}.
 \end{eqnarray*}
 Summing over all $k \simeq N$ yields a $O(1)$ bound. 
 For $k \geq N$, it is enough to observe 
 
 \begin{eqnarray*}
 \sum_{k \geq N} \left| T^k\left( \left\{ f^{N, A, \#_j} \right\}_{j=1}^n \right)(x)\right| &\leq& \sum_{k \geq N} \sum_{1 \leq m_1, ..., m_n \leq N: m_n \leq n_0}  \left| \int_{A(k-\frac{1}{2})} ^{A(k+\frac{1}{2})}  \prod_{j=1}^n\left[  \phi(x-Am_j -\alpha_j t) e^{2 \pi i \#_j m_j (x -\alpha_j t)} \right]\frac{dt}{t} \right| \\ &\lesssim_A& \sum_{k \geq N} \sum_{1 \leq m_1, ..., m_n \leq N} \prod_{j=1}^n  \frac{1}{1+|n_0 - m_j - \alpha _j k|^N} \\ &\lesssim_A& \sum_{k \geq N} \sum_{1\leq m_n \leq N}   \frac{1}{1+|n_0 - m_n - \alpha _nk|^N} \\ &\lesssim_A&   \sum_{1\leq m_n \leq N}   \frac{1}{1+|n_0 - m_n  |^{N-1}}  \\ &\lesssim_A& 1. 
 \end{eqnarray*}
 Moreover, replacing the sum over $k\geq N$ with the sum $k \leq 0$ yields the same estimate. Therefore,
 
 \begin{eqnarray*}
&& Re \left[ e^{ - 2 \pi i A (\sum_{j=1}^n \#_j) n_0 x} T\left( \left\{ f^{N, A, \#_j} \right\}_{j=1}^n \right) (x) \right]\\&\geq& \sum_{1 \leq k \lesssim_{A, \vec{\alpha}} N} \frac{1}{Ak}  - \left| \sum_{k \geq N} T_k(f_1^N, ..., f_n^N)(x) + \sum_{ k \leq 0} T_k(f_1^N, ..., f_n^N)(x)  + \sum_{k \simeq_{A, \vec{\alpha}} N} T_k(f_1^N, .., f_n^N)(x)\right| \\&\gtrsim& \frac{\log(N)}{A}. 
 \end{eqnarray*}
Using the point-wise bound $\left| T\left( \left\{ f^{N, A, \#_j} \right\}_{j=1}^n \right)(x) \right| \gtrsim \frac{\log(N)}{A} 1_{S_N}(x)$, where 
 
 \begin{eqnarray*}
 S_N= \bigcup_{N/2 \leq n_0 \leq N} \left[ An_0, An_0 + \frac{c_{\vec{\alpha}}}{A} \right]
 \end{eqnarray*}
together with $|S_N| \simeq_{\vec{\alpha}, A} N$, we may conclude 
 
 \begin{eqnarray*}
\left|\left| T(f_1^N, ..., f_n^N) \right|\right| _{\frac{1}{\sum_{i=1}^N} \frac{1}{p_i}} \gtrsim_A\log(N) N^{\sum_{i=1}^n \frac{1}{p_i}} >> N^{\sum_{i=1}^n \frac{1}{p_i}} \simeq \prod_{i=1}^n || f_i||_{p_i}. 
 \end{eqnarray*}
 Taking N arbitrarily large finishes the proof. 
 \end{proof}
 The remaining portion of \S{3} is dedicated to proving Lemma \ref{ML}. We must take a little care in understanding those terms for which the $\phi$ arguments are relatively small yet oscillation may be introduced  with respect to $x$ or $t$.   For fixed $n_0 \in [-N/3, N/3]$ and $k \in  \left[1, \frac{N/3}{\max_{1 \leq j \leq n} \{ | \alpha_j \}}\right]$, let $m_j(n_0, k):=n_0 -\alpha_j k$. We now proceed to organize the sum over $\vec{m} \in [1,N]^n$ around this ``core'' vector into two sets: small perturbations and large perturbations. Essentially, integration by parts together with the integrality of $\vec{m}$ and  $\vec{\alpha}$ will allow us to handle those terms arising from the small perturbations successfully. 
 
\subsubsection{Small Perturbations: $\sum_{j=1}^n \#_j\Delta_j\alpha_j= 0$}

This contribution cannot be subsumed as error. As before, let $\vec{m}( n_0,k)= n_0 - \vec{\alpha}k$ be the unperturbed initial state. Let $\Delta_j$ satisfy $\tilde{m}_j = m_j(n_0, k) + \Delta_j$ for each $j \in \{1, ..., n\}$ and $|\Delta_j| \leq 2 \cdot max_{1 \leq j \leq n} \{|\alpha_j|\}$. Then the contribution of the perturbed summand is

\begin{eqnarray*}
 Re\left[e^{-2 \pi i A ( \sum_{j=1}^n \#_j) n_0 x} \int_{A(k-\frac{1}{2})} ^{A(k+\frac{1}{2})}  \prod_{j=1}^n \left[ \phi(x-Am_j -\alpha_j t) e^{2 \pi iA \#_j m_j (x-\alpha_j t)} \right]\frac{dt}{t} \right] .
\end{eqnarray*}
Because $x \in [An_0, An_0 +\frac{c_{\vec{\alpha}}}{A}]$, we may rewrite $x = An_0 + \theta_x$, where $|\theta_x| \leq  \frac{c_{\vec{\alpha}}}{A}$, and use the integrality condition: 

\begin{eqnarray*}
 && Re\left[ e^{ - 2\pi i A (\sum_{j=1}^n \#_j) n_0 x} \int_{A(k-\frac{1}{2})} ^{A(k+\frac{1}{2})}  \prod_{j=1}^n \phi(x-Am_j -\alpha_j t) e^{2 \pi iA(\sum_{j=1}^n \#_j m_j) (x-\alpha_jt)}\frac{dt}{t} \right] \\ &=&  Re\left[ e^{2 \pi iA(\sum_{j=1}^n \#_j \Delta_j) x}\int_{A(k-\frac{1}{2})} ^{A(k+\frac{1}{2})}  \prod_{j=1}^n \phi(x-Am_j -\alpha_j t) \frac{dt}{t} \right]   \\ &\geq&Re\left[e^{2 \pi i(\sum_{j=1}^n \#_j \Delta_j)c_{\vec{\alpha}}} \int_{A(k-\frac{1}{2})} ^{A(k+\frac{1}{2})}  \prod_{j=1}^n \phi(x-Am_j -\alpha_j t) \frac{dt}{t} \right] \\ &\gtrsim& \frac{1}{Ak}. 
\end{eqnarray*}
This term is acceptable as an additional piece of the main contribution provided we take $c_{\vec{\alpha}}$ sufficiently small. 

 \subsubsection{Small Perturbations: $\sum_{j=1}^n \#_j \Delta_j\alpha_j \not = 0$}
There may be many small perturbations of $\vec{m}$, say $\tilde{\vec{m}}$, for which $\sum_{j=1}^n \#_j \alpha_j \tilde{m}_j \not =0$ and yet $0 \in [n_0 - \tilde{m}_j - \alpha_j (l + \frac{1}{2}) , n_0 - \tilde{m}_j - \alpha_j ( l - \frac{1}{2})] $ for all $j$ where $\tilde{m}_j = m_j + \Delta_j$ for all $j \in \{1,..., n\}$ and

\begin{eqnarray*}
\max_{1 \leq j \leq n} |\Delta_j| \leq \max_{1 \leq j \leq n}|\alpha_j|.
\end{eqnarray*}
Putting absolute values inside the integral for these terms would therefore yield unacceptable error terms on the same order as the main contribution. Instead, we need to observe

\begin{eqnarray*}
\left| \int_{A(k-\frac{1}{2})} ^{A(k+\frac{1}{2})}  \prod_{j=1}^n \phi(x-Am_j -\alpha_j t)e^{-2 \pi i A(\sum_{j=1}^n \#_j \tilde{m}_j \alpha_j) t} \frac{dt}{t} \right| \lesssim \frac{1}{A^2 k}. 
\end{eqnarray*}
To show this, it suffices to prove 

\begin{eqnarray*}
\left| \int_{A(k-\frac{1}{2})} ^{A(k+\frac{1}{2})}  \prod_{j=1}^n \phi(x-Am_j -\alpha_j t) e^{-2 \pi i \tilde{A}t} \frac{dt}{t} \right| \lesssim_{\vec{\alpha}}  
\frac{1}{|\tilde{A} A| k}. 
\end{eqnarray*}
Indeed, using integration by parts, we have

\begin{eqnarray*}
&& \left| \int_{A(k-\frac{1}{2})} ^{A(k+\frac{1}{2})}  \prod_{j=1}^n \phi(x-Am_j -\alpha_j t) e^{-2 \pi i \tilde{A}t} \frac{dt}{t} \right|  \\&=& \frac{1}{2 \pi\tilde{A} } \left| \int_{A(k-\frac{1}{2})} ^{A(k+\frac{1}{2})}  \prod_{j=1}^n \phi(x-Am_j -\alpha_j t) \frac{d}{dt} \left[ e^{-2 \pi i \tilde{A}t} \right] \frac{dt}{t} \right| \\ &\lesssim & \frac{1}{|\tilde{A}| }\left[  \frac{c}{Ak} + \left| \int_{A(k-\frac{1}{2})} ^{A(k+\frac{1}{2})}  \frac{d}{dt} \left[ \prod_{j=1}^n \phi(x-Am_j -\alpha_j t)  \right]e^{-2 \pi i \tilde{A}t}  \frac{dt}{t} \right|  +\left| \int_{A(k-\frac{1}{2})} ^{A(k+\frac{1}{2})}  \prod_{j=1}^n \phi(x-Am_j -\alpha_j t) e^{-2 \pi i \tilde{A}t}  \frac{dt}{t^2} \right| \right] \\&\lesssim& \frac{1}{|\tilde{A} |A k}.
\end{eqnarray*}
As there are $O_{\vec{\alpha}}(1)$ many small perturbations of $\vec{m}(n_0, k)$, those perturbations satisfying the additional property that $\sum_{j=1}^n \#_j \Delta_j \alpha_j \not = 0$ may be subsumed as error upon taking $A$ sufficiently large.

\subsubsection{Large Perturbations}
Fix $n_0, k$. Restrict attention to all vectors  $\tilde{\vec{m}}$ s.t. $\exists$ index $1 \leq j_* \leq n$ for which $\tilde{m}_{j_*} = m_{j_*}(n_0, k) + \Delta_{j_*}$ and $|\Delta_{j_*}| \geq 2 \cdot max_{1 \leq j \leq n}\{ |\alpha_j|\}$. Then 

\begin{eqnarray*}
|1_{[A(k-\frac{1}{2}), A(k+\frac{1}{2})]} (t) \phi(x - Am_{j_*}-A\Delta_{j_*}  -\alpha_{j_*} t)| \lesssim_{\phi} \frac{1}{A^N} \frac{1}{1+|\Delta_{j_*} |^N}.
\end{eqnarray*}
Therefore, the total contribution of large perturbations can be majorized by 

\begin{eqnarray*}
\frac{C}{A^N} \sum_{\vec{\Delta} \in \mathbb{Z}^n} \prod_{j=1}^n \frac{1}{1+|\Delta_j|^N} \frac{1}{k} \lesssim_{N , \vec{\alpha}} \frac{1}{A^N k}.  
\end{eqnarray*}
Again, by taking $A$ sufficiently large independent of N, this contribution becomes an error term. 

We have now proven Lemma \ref{ML} and therefore Theorem \ref{MT} when $\vec{\alpha} \in \mathbb{Q}^d$.

\subsection{PART II: The Irrational Case}
 Fix $N, A \in \mathbb{N}$. Construct for each $1 \leq j \leq n$ the function 

\begin{eqnarray*}
f^{N,A ,\#}(x)= \sum_{-N \leq m \leq N} \phi(x-A\alpha_j m) e^{2 \pi i A \# \alpha_jm x} = \sum_{-N \leq m \leq N} f_{m}^{N, A, \#}(x).
\end{eqnarray*}
Set $S = \left\{ 1/ \alpha_1, ..., 1/\alpha_n \right\}$, $\rho =1/A^2$, and $Bohr_{c(\vec{\alpha}) N }(S, \rho)=\left\{ N_1, ..., N_{|Bohr_{c(\vec{\alpha}) N}(S, \rho)|} \right\}$ for some constant $0<c_{\vec{\alpha}}<<1$.  Then $|Bohr_{c(\vec{\alpha}) N}(S, \rho)| \simeq_{\vec{\alpha}, A} N$ by Lemma  \ref{ML***}.  Moreover, setting $\mathfrak{N}^{n_0}_j$ equal to the closest integer to $\alpha^{-1}_j n_0$ for each $j \in \{1, ..., n\}$ and $n_0 \in Bohr_{c(\vec{\alpha})N}(S, \rho)$  ensures $m_j = n_0 -\alpha _j k $ can be approximately solved in the sense that setting $m_j = -k + \mathfrak{N}^{n_0}_j$ ensures $| n_0 - \alpha_jm_j  -\alpha _j k | \lesssim_{\vec{\alpha}} \frac{1}{A^2}$. Moreover, by choosing $c(\vec{\alpha})$ small enough we may assume $m_j = -k + \mathcal{N}^{n_0}_j \in [-N,N]$ for all $n_0 \in Bohr_{c(\vec{\alpha}) N}(S, \rho)$ and $|k| \lesssim N$.  Restrict our attention to $x \in \Omega := \bigcup_{n_0 \in Bohr_{c(\vec{\alpha})N}(S, \rho)} [An_0, An_0+ c_{\vec{\alpha}}/A]$. Therefore, $|\Omega| \gtrsim_{A, \vec{\alpha}} N$. Now we are ready to investigate the main contribution. To this end, assume $t = \tilde{t} + Ak$ where $|\tilde{t}| \leq \frac{A}{2}$ and write down

\begin{eqnarray*}
\sum_{j=1}^n \#_j (m_j (n_0, k)) (x-\alpha_j t) &=& \sum_{j=1}^n \#_j (n_0 - \alpha_j k  + \delta(n_0,j) )(x-\alpha_j t) \\&=& \left[   \sum_{j=1}^n \#_j \right] n_0 x + \sum_{j=1}^n \#_j ( \delta(n_0,j)) (x -\alpha_j t) \\ &=& \left[   \sum_{j=1}^n \#_j \right] n_0 x + \sum_{j=1}^nA  \#_j  \delta(n_0,j) n_0 + \sum_{j=1}^n  \#_j \delta(n_0, j)(\theta_x -\alpha_j \tilde{t} - A \alpha_j k) \\ &=& C( x) - \sum_{j=1}^n \#_j \delta(n_0, j) \alpha_j \tilde{t} - \sum_{j=1}^n A \#_j \delta(n_0, j) \alpha_j k
\end{eqnarray*}
where $Re \left[ C(x) - \left[ \sum_{j=1}^n \#_j \right] n_0 x - \sum_{j=1} A \#_j \delta(n_0, j) n_0 \right]  \gtrsim 1$ for all $|x-An_0| \lesssim_{\vec{\alpha}, A} 1$. 
Moreover,

\begin{eqnarray*}
\sum_{j=1}^n  \#_j \delta(n_0, j) \alpha_j k = \sum_{j=1}^n  \#_j (n_0-\alpha_j \mathfrak{N}_j^{n_0})\alpha_j k=- \sum_{j=1}^n \#_j \alpha_j^2 \mathfrak{N}^{n_0}_j k \in \mathbb{Z}. 
\end{eqnarray*}
Hence, provided $A \in \mathbb{Z}$,

\begin{eqnarray*}
e^{ 2 \pi i A\sum_{j=1}^n \#_j (m_j (n_0, k)) (x-\alpha_j t)} = e^{2 \pi i A (C(n_0, x) -\sum_{j=1}^n \#_j \delta(n_0, j) \alpha_j \tilde{t})}.
\end{eqnarray*}
Using $|\delta(n_0, j)| \lesssim_{\vec{\alpha}} \frac{1}{A^2}$ gives an acceptable main contribution, i.e.

\begin{eqnarray*}
 && Re\left[ e^{ - 2\pi i A \left[ \sum_{j=1}^n \#_j) n_0 x+A \sum_{j=1}^n \#_j \delta(n_0, j) n_0 \right]} \int_{A(k-\frac{1}{2})} ^{A(k+\frac{1}{2})}  \prod_{j=1}^n \phi(x-Am_j -\alpha_j t) e^{2 \pi iA(\sum_{j=1}^n \#_j m_j) (x-\alpha_jt)}\frac{dt}{t} \right] \\ &=&  Re\left[ \int_{-\frac{A}{2}} ^{\frac{A}{2}}  \prod_{j=1}^n \phi(\theta_x -A \delta(n_0, j)+\alpha_j t) e^{-2 \pi i A \left[ \sum_{j=1}^n \#_j\delta(n_0, j) \alpha_j t \right]}\frac{dt}{t+Ak} \right]  \\&\gtrsim& \frac{1}{Ak}. 
\end{eqnarray*}

\subsubsection{Small Perturbations: $\sum_{j=1}^n \#_j \Delta_j \alpha^2_j = 0$}
As in the rational case, small perturbation with the above cancellation property cannot be subsumed as error. So, let $\tilde{m}_j = m_j (n_0, k) + \Delta_j$ with $\max_{1 \leq j \leq n} |\Delta_j|  < 2 \max_{1 \leq j \leq n} |\alpha_j| $, then observe  

\begin{eqnarray*}
\sum_{j=1}^n \#_j \alpha_j (m_j (n_0, k) + \Delta_j) (x-\alpha_j t) &=& \sum_{j=1}^n \#_j (n_0 - \alpha_j k  + \delta(n_0,j) + \alpha_j \Delta_j) (x-\alpha_j t) \\&=& \left[   \sum_{j=1}^n \#_j \right] n_0 x + \sum_{j=1}^n \#_j ( \delta(n_0,j) + \alpha_j \Delta_j) (x -\alpha_j t) \\ &=& \left[   \sum_{j=1}^n \#_j \right] n_0 x + \sum_{j=1}^n \#_j ( \delta(n_0,j) + \alpha_j \Delta_j) (x -\alpha_j \tilde{t} - A \alpha_j k).
\end{eqnarray*}
Because $x \in [An_0, An_0 +\frac{c_{\vec{\alpha}}}{A}]$, we may rewrite $x = An_0 + \theta_x$, where $|\theta_x| \leq  \frac{c_{\vec{\alpha}}}{A}$, and again use integrality observe

\begin{eqnarray*}
e^{ 2 \pi i A\sum_{j=1}^n \#_j (m_j (n_0, k)) (x-\alpha_j t)} = e^{2 \pi i A (C( x) -\sum_{j=1}^n \#_j \delta(n_0, j) \alpha_j \tilde{t}+ \sum_{j=1}^n  \#_j \Delta_j \alpha_j x)} .
\end{eqnarray*}
We are not quite content with the additional factor $e^{2 \pi i\left[ \sum_{j=1}^n \#_j  \alpha_j \Delta_j \right]  x}$, as we do not at first glance have good control on the sign of its real part, say. However, because $|x-An_0| \lesssim_{\vec{\alpha}} \frac{1}{A}$,  it clearly suffices to obtain good control on the sign of the real part of 
\begin{eqnarray*}
e^{2 \pi i A^2 \left[ \sum_{j=1}^n \#_j \alpha_j \Delta_j \right] n_0} =  e^{2 \pi i A^2 \left[ \sum_{j=1}^n \#_j \alpha_j \Delta_j \right] (\alpha_j \mathcal{N}_j^{n_0} + \delta(n_0, j))} = e^{2 \pi i A^2 \left[ \sum_{j=1}^n \#_j \alpha_j \Delta_j \right]  \delta(n_0, j)}. 
\end{eqnarray*}
Since $|\delta(n_0, j)| \lesssim_{\vec{\alpha}} \frac{1}{A^2}$, we do in fact have good control. Hence, small perturbations satisfying $\sum_{j=1}^n \#_j \Delta_j \alpha_j^2=0$ always reinforce the main contribution.

\subsubsection{Small Perturbations: $\sum_{j=1}^n \#_j \Delta_j \alpha^2_j \not = 0$} 
Use $|\delta(n_0, j)| \leq \frac{1}{A^2}$ combined with the integration by parts from before to produce acceptable error terms.

\subsection{Large Perturbations}
This case is handled using the same argument as before, so the details are omitted.

\end{proof}
\section{Symbols of Type $sgn(\xi_1+ \xi_2) \Pi(\xi_2, \xi_3)$ and Maximal Variants}
Recall 
\begin{reptheorem}{IT}
There exists a H\"{o}rmander-Marcinkiewicz symbol $a: \mathbb{R}^2 \rightarrow \mathbb{R}$ of $(\phi, \psi)$ type, i.e. $a$ is sum of tensor products that are $\phi$ type in the first index and $\psi$ type in the second index satisfying $|\partial^{\vec{\alpha}} a(\xi)| \leq \frac{C_{\vec{\alpha}}}{dist(\vec{\xi}, \vec{0})^{|\vec{\alpha}|}}$, such that the trilinear operator $T_m$ whose symbol $m$ is given by $m(\xi_1, \xi_2, \xi_3) = sgn(\xi_1 + \xi_2) a(\xi_2, \xi_3)$ satisfies no $L^p$ estimates. 
\end{reptheorem}
Remark: This negative result is a strengthening of Muscalu, Tao, and Thiele's observation in \cite{MR1981900}, where the symbol $m$ is taken to be $m(\xi_1, \xi_2, \xi_3) = sgn(\xi_1 + \xi_2) sgn(\xi_2 + \xi_3)$. Morally speaking, the $sgn$ multiplier cannot be combined with even nice symbols involving indices outside the $sgn$ to yield a bounded operator.

\begin{corollary}
There exists two families of Schwartz functions $\{\phi_k\}_{k \in \mathbb{Z}}, \{\tilde{\phi}_k\}_{k \in \mathbb{Z}}$ with uniformly bounded $L^1$ norms satisfying $supp~\hat{\phi}_k , \hat{\tilde{\phi}}_k \subset [-2^k,2^k]$ such that the maximal bi-sublinear operator given by

\begin{eqnarray*}
S_1\left( \{\tilde{\phi}_k \}, \{ \phi_k\} \right) :(f_1, f_2) \mapsto \sup_{k \in \mathbb{Z}} \left| H(f_1*\tilde{\phi}_k, f_2*\phi_k) \right|
\end{eqnarray*}
satisfies no $L^p$ estimates. 
\end{corollary}

\begin{proof}
For a contradiction, assume every pair of families $\{\tilde{\phi}_k\}$ and $\{\phi_k\}$ satisfying the conditions of the corollary are bounded on some Lebesgue tuple $(p_1, p_2)$ satisfying $1 < p_1, p_2, \frac{p_1 p_2}{p_1 + p_2} <\infty$ depending on $\{\tilde{\phi}_k\}$ and $\{\phi_k\}$. We proceed to show multipliers of the form $m(\xi_1, \xi_2, \xi_3) = sgn(\xi_1 + \xi_2) a(\xi_2, \xi_3)$ where $a:\mathbb{R}^2 \rightarrow \mathbb{R} $ is a Mikhlin-H\"{o}rmander symbol in $\mathbb{R}^2$  would satisfy some estimates. WLOG, decompose for some bounded sequences $\{c_k\}_{k \in \mathbb{Z}}, \{d_k\}_{k \in \mathbb{Z}}, \{e_k\}_{k \in \mathbb{Z}}$

\begin{eqnarray*}
T_m(f_1, f_2, f_3) &=& \sum_{k \in \mathbb{Z}}\left[ c_k H(f_1 \cdot f_2 * \psi_k) \cdot f_3*\psi_k +  d_k H(f_1 \cdot f_2 * \phi_k)\cdot f_3*\psi_k + e_k H(f_1 \cdot f_2*\psi_k) \cdot f_3*\phi_k\right] \\ &:=& I + II + III,
\end{eqnarray*}
where $\{\psi_k\}$ is another family of uniformly $L^1$ bounded Schwartz functions with $supp~\hat{\psi}_k \subset [-2^{k+1}, -2^{k-1}] \cup [2^{k-1}, 2^{k+1}]$.
It suffices to prove estimates for $I, II, III$ separately. Handling the contribution from $I$  is immediate from 
from Cauchy-Schwarz, basic vector-valued inequalities, and the standard square functions estimates. By dualizing, we see that estimating $II$ is essentially the same as estimating $III$. To this end, we write down

\begin{eqnarray*}
&& \int_\mathbb{R} \sum_{k \in \mathbb{Z}} \left[ d_k H\left[f_1 \cdot f_2*\phi_k \right] f_3*\psi_k\right] f_4 dx\\ &=& \int_\mathbb{R} \sum_{k \in \mathbb{Z}}  \left[ d_k H\left[f_1 *\tilde{\phi}_k \cdot f_2*\phi_k \right] f_3*\psi_k\right] f_4 dx +  \int_\mathbb{R} \sum_{k \in \mathbb{Z}} \sum_{l \gtrsim k} d_k H\left[f_1 *\tilde{\psi}_l \cdot f_2*\phi_k \right] f_3*\psi_k f_4 dx \\ &=& II_a + II_b. 
\end{eqnarray*} 
Then $II_a =  \int_\mathbb{R} \sum_{k \in \mathbb{Z}}  d_k H\left[f_1 *\tilde{\phi}_k \cdot f_2*\phi_k \right]  f_3*\psi_kf_4 * \tilde{\psi}_k dx $
which satisfies estimates by our assumption. For $II_b$, we may again break the sum into two subsums: 

\begin{eqnarray*}
II_b &=&\int_\mathbb{R} \sum_{k \in \mathbb{Z}}  \sum_{ l \simeq k} d_k H\left[f_1 *\tilde{\psi}_l \cdot f_2*\phi_k \right] f_3*\psi_k f_4 dx + \int_\mathbb{R} \sum_{k \in \mathbb{Z}}  \sum_{ l >> k} d_k H\left[f_1 *\tilde{\psi}_l \cdot f_2*\phi_k \right] f_3*\psi_k f_4 dx \\ &:=& II_{b,1} + II_{b,2}. 
\end{eqnarray*}
Estimates for $II_{b,1}$ follows immeadiately by Cauchy-Schwarz, vector-valued inequalities, and routine estimates for square and maximal functions. Moreover,

\begin{eqnarray*}
II_{b,2} &=&  \int_\mathbb{R} \sum_{l \in \mathbb{Z}}  \sum_{ k<<l} d_k H\left[f_1 *\tilde{\psi}_l \cdot f_2*\phi_k \right] f_3*\psi_k f_4 *\tilde{ \tilde{\psi}}_l dx \\ &=& \int_\mathbb{R} \sum_{l \in \mathbb{Z}}  \sum_{ k << l} d_k f_1 *\tilde{\psi}_l \cdot f_2*\phi_k f_3*\psi_k f_4 *\tilde{ \tilde{\psi}}_l dx  \\ &=& \int_{\mathbb{R}} \sum_{l \in \mathbb{Z}} \sum_{k \in \mathbb{Z}} d_k f_1 *\tilde{\psi}_l \cdot f_2*\phi_k f_3*\psi_k f_4 * \tilde{\tilde{\psi}}_l dx  \\&-&  \int_{\mathbb{R}} \sum_{l \in \mathbb{Z}} \sum_{k\simeq l} d_k f_1 *\tilde{\psi}_l \cdot f_2*\phi_k f_3*\psi_k f_4 *\tilde{ \tilde{\psi}}_l dx\\&-&  \int_{\mathbb{R}} \sum_{l \in \mathbb{Z}} \sum_{k>>l} d_k f_1 *\tilde{\psi}_l \cdot f_2*\phi_k f_3*\psi_k f_4 *\tilde{ \tilde{\psi}}_l dx \\ &:=& II_{b,2,1} - II_{b,2,2} - II_{b,2,3}.
\end{eqnarray*}
Estimates for $II_{b,2,1}$ are immediate from the paraproduct theory. Bounds for $II_{b,2,2}$ follows from routine estimates for square and maximal functions. Lastly, $II_{b,2,3}=0$ on account of the various frequency supports and the assumption $k >> l$.

\end{proof}

\begin{corollary}
For each tuple $(p_1, p_2)$ such that $1 < p_1 , p_2 , \frac{p_1 p_2}{p_1 + p_2} <\infty$ there exists two families of Schwartz functions $\{\psi_k\}, \{ \tilde{\psi}_k \}$ with uniformly bounded $L^1$ norms satisfying $supp ~\hat{\psi}_k \subset[ -2^{k+1}, -2^{k-2}] \cup [2^{k-2}, 2^{k+1}]$  such that the maximal bi-sublinear operator given by 

\begin{eqnarray*}
S_2\left(\{\phi_k\}, \{\tilde{\phi}_k\}\right):(f_1, f_2) \mapsto \sup_{l \in \mathbb{Z}}\left| \sum_{k \leq l} H(f_1*\tilde{\psi}_k f_2 *\psi_k)\right|
\end{eqnarray*}
does not continuously map $L^{p_1}(\mathbb{R}) \times L^{p_2}(\mathbb{R}) \rightarrow L^{\frac{p_1 p_2}{p_1 + p_2}}(\mathbb{R})$. 
\end{corollary}
\begin{proof}
For a contradiction, suppose $S_2$ were bounded for some Lebesgue type $(p_1, p_2)$ satisfying $1 < p_1, p_2, \frac{p_1 p_2}{p_1 + p_2} <\infty$ for all admissible families $\{\tilde{\psi}_k\}, \{ \psi_k\}$.  It suffices to observe for any decomposition $\mathbb{R} = \bigcup_{i \in \mathcal{I}} C_i$

\begin{eqnarray*}
\int_{\mathbb{R}} \sum_{i \in \mathcal{I}} 1_{C_i} H(f_1 * \tilde{\phi}_i, f_2*\phi_i) f_3~dx  &\simeq& \int_{\mathbb{R}}  \sum_{i \in \mathcal{I}}\sum_{k_1 , k_2 \leq i}  1_{C_i} H(f_1 * \tilde{\psi}_{k_1}, f_2*\psi_{k_2}) f_3~dx,
\end{eqnarray*}
where the two families $\{\tilde{\psi}\}$ and $\{ \psi\}$ satisfy the uniform $L^1$ property in addition to $supp~\hat{\psi}_k ,supp~\hat{\tilde{\psi}}_k \subset[-2^{k+1}, -2^{k-1}] \cup  [2^{k-1}, 2^{k+1}]$. By assumption, the diagonal terms $k_1 = k_2$ will satisfy estimates, so we are left handling 

\begin{eqnarray*}
 \int_{\mathbb{R}}  \sum_{i \in \mathcal{I}}\sum_{k_1 \not = k_2 \leq i}  1_{C_i} H(f_1 * \tilde{\psi}_{k_1}, f_2*\psi_{k_2}) f_3~dx := I. 
\end{eqnarray*}
By the frequency support assumptions, if $|k-l| \geq 2$, then $H(f _1* \tilde{\psi}_k \cdot f_2 * \psi_l) = f_1*\tilde{\psi}_k \cdot f_2 * \psi_l$. Therefore, we may further decompose the above sum into the following parts:

\begin{eqnarray*}
I &=& \int_\mathbb{R}  \sum_{i \in \mathcal{I}} \sum_{ (k_1, k_2): |k_1 - k_2| = 1, k_1, k_2 \leq i }  1_{C_i}\cdot H(f_1 * \tilde{\psi}_{k_1}\cdot f_2*\psi_{k_2})\cdot  f_3~dx  \\&+& \int_\mathbb{R} \sum_{i \in \mathcal{I}}  \sum_{ (k_1, k_2): |k_1 - k_2| \geq 2, k_1, k_2 \leq i}  1_{C_i}\cdot  f_1 * \tilde{\psi}_{k_1} \cdot f_2*\psi_{k_2} \cdot f_3~dx \\ &:=& I_a + I_b. 
\end{eqnarray*}
Estimating $I_b$ is straightforward. Indeed, it is easy to see for $f_3 \in L^{p_3}(\mathbb{R})$ such that $||f_3||_{p_3}=1$ and $\sum_{j=1}^3 \frac{1}{p_i} =1$ that $|I_b| \lesssim || \mathfrak{M} (f_1)||_{p_1} || \mathfrak{M}(f_2)||_{p_2} + ||\mathcal{S}(f_1)||_{p_1} || \mathcal{S}(f_2)||_{p_2}$
where $\mathfrak{M}(f)(x) := \sup_{k_1,k_2} \left| \sum_{k_1 \leq k \leq k_2} f*\psi_k(x) \right|$ and $\mathcal{S}$ is the Littlewood-Paley square function. Moreover, as $\mathfrak{M}$ maps $L^p$ into $L^p$ for all $1 <p<\infty$, it suffices to prove estimates for $I_a$. We may rewrite

\begin{eqnarray*}
& &I_a\\ &=& \int_\mathbb{R}  \sum_{i \in \mathcal{I}}  \sum_{ k_1\leq i }  1_{C_i} \left( \left[ H(f_1 * \tilde{\psi}_{k_1} \cdot f_2*\psi_{k_1-1}) \right]+ H \left[ f_1 * \tilde{\psi}_{k_1-1} \cdot f_2 * \psi_{k_1} \right] \right)f_3~dx \\ &=&  \int_\mathbb{R}  \sum_{i \in \mathcal{I}} \sum_{ k_1\leq i }  1_{C_i} H \left[ f_1 * (\tilde{\psi}_{k_1}+ \tilde{\psi}_{k_1 -1})  f_2*(\psi_{k_1} + \psi_{k_1-1}) \right] dx\\ &-& \int_\mathbb{R} \sum_{i \in \mathcal{I}}  \sum_{ k_1\leq i }  1_{C_i}\left( H \left[ f_1 * \tilde{\psi}_{k_1-1}f_2 * \psi_{k_1} \right] +H \left[ f_1 * \tilde{\psi}_{k_1-1} f_2 * \psi_{k_1-1} \right]  \right) f_3~dx .
\end{eqnarray*}
However, by our hypothesis, each of the three main terms can be bounded, and hence $ \sup_{k \in \mathbb{Z}} \left| H(f_1 * \tilde{\phi}_k \cdot f_2 * \phi_k )\right|$ would satisfy estimates.   
\end{proof}
In fact, it is not hard to prove directly that for specific choices of $\{\phi_k\}$ and $\{\psi_k\}$ obeying the uniform $L^1$ conditions and support restrictions $supp~\hat{\phi}_k \subset [-2^k, 2^k]$ and $supp~\hat{\psi}_k \subset [2^{k-1}, 2^{k+1}] $  the maps 

\begin{eqnarray*}
(f_1, f_2)  \mapsto \sup_{k \in \mathbb{Z}} \left| H( f_1 * \phi_k \cdot f_2 * \psi_k ) \right|
\end{eqnarray*}
as well as
\begin{eqnarray*}
(f_1, f_2) \mapsto \sup_{k \in \mathbb{Z}} \left| \sum_{l \leq k} H \left( f_1 * \psi_k, f_2*\psi_k \right) \right|
\end{eqnarray*}
satisfy no $L^p$ estimates. 
We now prove Theorem \ref{IT}. 
\begin{proof}
Fix $N, M \in \mathbb{N}$. Let $A=2^M$ and choose $\phi \in \mathcal{S}(\mathbb{R})$ with compact Fourier support inside $[-1/4, 1/4]$ such that $\phi (0) \not = 0$. Then construct the functions 

\begin{eqnarray*}
f_1^{N,A}(x) &=& \sum_{1 \leq m \leq N} \phi(x-Am) e^{-2 \pi i 2^m x} := \sum_{1 \leq m \leq N} f^{N,A}_{1,m}(x) \\ 
f_2^{N,A}(x) &=& \sum_{1 \leq m \leq N} \phi(x-Am) e^{2 \pi i 2^m x}:= \sum_{1 \leq m \leq N} f^{N,A}_{2,m}(x) \\ 
f_3^{N,A}(x) &=& \sum_{1 \leq m \leq N} \phi(x-Am) e^{2 \pi i 2^m x} := \sum_{1 \leq m \leq N} f^{N,A}_{3,m}(x) .
\end{eqnarray*}
Let $S_N := \bigcup_{n \in \mathbb{Z} \cap [N/2, 2N/3]} [ An, An+1]$. The claim is that $|T_m(f_1^N, f_2^N, f_3^N)(x)1_{S_N}(x)|  \gtrsim_A \log(N) 1_{S_N}(x)$ for sufficiently large choice of $A$ independent of $N$,  from which we immediately deduce that $||T_m(f_1^N, f_2^N, f_3^N)||_p \gtrsim \log(N) N^{\frac{1}{p_1} + \frac{1}{p_2}}$ whereas $\prod_{i=1}^2 || f_i^N||_{p_i} \simeq N^{\frac{1}{p_1} + \frac{1}{p_2}}$. Taking $N$ arbitrarily large yields the theorem. For each $k \in \mathbb{N}$ choose $\phi_k , \psi_k \in \mathcal{S}(\mathbb{R})$ satisfying $1_{[-2^{k-1}, 2^{k-1}]} \leq \hat{\phi}_k \leq 1_{[-2^{k}, 2^{k}]}$ and $1_{[2^{k-1/4}, 2^{k+1/4}]} \leq \hat{\psi}_k \leq 1_{ [2^{k-\frac{1}{2}}, 2^{k+\frac{1}{2}}]}$ in addition to the properties 

\begin{eqnarray*}
\left| \left[\frac{d}{d\xi}\right]^\alpha \hat{\psi}_k (\xi) \right| , \left| \left[ \frac{d}{d\xi}\right]^\alpha \hat{\phi}_k (\xi) \right|\leq \frac{C_\alpha}{2^{k \alpha}} ~~\forall~\vec{\alpha} \in 0 \cup \mathbb{N}.
\end{eqnarray*}
 Take $a(\xi_2, \xi_3) = \sum_{k \in \mathbb{N}} \hat{\phi}_k(\xi_2) \hat{\psi}_k(\xi_3)$ and observe that for $1 \leq n_2, n_3 \leq N$ whenever $n_3 << n_2$

\begin{eqnarray*}
&&T_m(f_1^A, f_2^A, f_3^A )(x) \\&:=& \sum_{k \in \mathbb{N}}~\sum_{1 \leq m_1, m_2, m_3 \leq N}  \int_{\xi_1 + \xi_2>0} \hat{f}^A_{1,m_1}(\xi_1) \hat{f}^A_{2,m_2}(\xi_2)  \hat{f}^A_{3,m_3}(\xi_3)\hat{\phi}_k(\xi_2) \hat{\psi}_k(\xi_3) e^{2 \pi i x (\xi_1 + \xi_2 + \xi_3)} d\xi_1 d \xi_2 d \xi_3 \\ &:=& \sum_{k \in \mathbb{Z}}~ \sum_{1 \leq m_1, m_2, m_3 \leq N} T_m^k(f^{N,A}_{1,m_1}, f^A_{2, m_2}, f^A_{3, m_3})(x).
\end{eqnarray*}
By construction, $supp~\hat{f}^A_{2,n}, supp~\hat{f}^A_{3,n} \subset [2^{M +n}-1/4, 2^{M+n}+1/4]$. Hence, the only tuples $(m_1, m_2, m_3)$ for which $T_m^k (f^A_{1, m_1}, f^A_{2, m_2}, f^A_{3, m_3}) \not \equiv 0$ must satisfy $m_2 \leq  m_3$. Moreover, $m_2 \leq m_3-2$ together with $T_m^k(f^{N,A}_{1,m_1}, f^A_{2, m_2}, f^A_{3, m_3})(x) \not \equiv 0$ ensures 

\begin{eqnarray*}
T_m^k(f^{N,A}_{1,m_1}, f^{A}_{2, m_2}, f^{A}_{3, m_3})(x) = \left( \int_\mathbb{R} \prod_{j=1}^2 \phi(x-Am_j -t)e^{2\pi i(-1)^j 2^{m_j}t}  \frac{dt}{t} \right) \left( \phi(x-Am_3) e^{2 \pi i  (2^{m_2}+ 2^{m_3}-2^{m_1})x} \right).
\end{eqnarray*}
To save space, say $A << B$ provided $A \leq B-2$. Moreover, it is easy to see that those terms corresponding to $m_2\in \{m_3-1, m_3\}$ and $m_2 \not <\not< m_3$ can be satisfactorily estimated. Indeed, this portion is writable as $\sum_{k \in \mathbb{N}} H(f_1 \cdot f_2*\tilde{\psi}_k) \cdot f_3*\psi_k$, which can again be handled by Cauchy-Schwarz, the Fefferman-Stein inequality, and routine square function estimates. Hence, it suffices to produce a $\log$-type pointwise blow-up for the remaining terms: 
\begin{eqnarray*}
&& \tilde{T}_m(f_1^A, f_2^A, f_3^A)i (x)\\&:=& \sum_{1 \leq m_2 \leq N } \sum_{ m_1 \leq  m_2} \sum_{m_3 >> m_2} \left(  \int_\mathbb{R} \prod_{j=1}^2 \phi(x-Am_j -t)e^{2\pi i (-1)^j2^{m_j}t}  \frac{dt}{t} \right) \left( \phi(x-Am_3) e^{2 \pi i  (2^{m_2}+ 2^{m_3}-2^{m_1})x} \right).
\end{eqnarray*}
\subsection{Main Contribution}
Fix $x \in \left[An_0,An_0+1\right]$. Due to the Schwartz decay of $\phi$, one expects the main term to arise from the cases where $m_3 = n_0$, in which case the corresponding sum over the $(m_1, m_2)$ will be $\sum_{ m_2 << n_0} \sum_{ m_1 \leq m_2}$. Moreover, of these terms, one expects the largest component to arise from the terms where $m_1 \simeq m_2$. Thus, at least heuristically, we are able to produce a quantity $O(\log(N))$ on a set of size $O(N)$ as claimed. 

Now let $x = An_0 + \theta_x$ where $0 \leq \theta_x <1$. The main contribution to $T(f_1^{N,A}, f_2^{N,A}, f_3^{N,A})$ when $m_3 = n_0, m_2=m_1$ is given  by the formula

\begin{eqnarray*}
&&  \sum_{m_1 <<  n_0}   T_m\left(f^{N,A}_{1,m_1}, f^{N,A}_{2,m_1}, f^{N,A}_{3,n_0}\right)(x) \\ &=&\left( \sum_{m_1 <<  n_0}  \int_\mathbb{R}  \phi^2(An_0 + \theta_x-Am_1 -t)  \frac{dt}{t} \right) \left( \phi(\theta_x) e^{2 \pi i 2^{n_0}\theta_x} \right) \\ &=& \left( \sum_{m_1 <<  n_0} \sum_{l \in \mathbb{Z}}   \int_{A(l-\frac{1}{2})}^{A(l+\frac{1}{2})}  \phi^2(An_0 + \theta_x-Am_1 -t)  \frac{dt}{t} \right) \left( \phi(\theta_x) e^{2 \pi i 2^{n_0}\theta_x} \right) \\ &:=&   \sum_{m_1 <<  n_0}  \sum_{l \in \mathbb{Z}} T_m^l\left(f^{N,A}_{1,m_1}, f^{N,A}_{2,m_1}, f^{ A}_{3,n_0}\right)(x).
\end{eqnarray*}
\begin{lemma}
 To prove Theorem \ref{IT}, it  suffices to show that for sufficiently large $M \in \mathbb{Z}$ and each $l \in \mathbb{Z} \cap [2, N/4]$ 

\begin{eqnarray*}
 Re \left[ e^{-2 \pi iA 2^{n_0} x} \sum_{m_1 <<  n_0} T^l_m \left(f^{N,A}_{1,m_1}, f^{N,A}_{2,m_1}, f^{N, A}_{3,n_0}\right)(x) \right]1_{S_N}(x) \gtrsim \frac{1}{Al}.
 \end{eqnarray*}
 \end{lemma}
Indeed, assuming the claim, it follows that $\forall x \in S_N$
\begin{eqnarray*}
&& Re \left[ e^{-2 \pi i A 2^{n_0}x}  \sum_{m_1 <<  n_0}   T_m\left(f_{1,m_1}^{N, A}, f^{N,A}_{2,m_1}, f^{N,A}_{3,n_0}\right)(x)   \right]  \\&=&  Re \left[ e^{-2 \pi iA 2^{n_0}x}  \sum_{m_1 <<  n_0}  \sum_{l \in \mathbb{Z}}  T^l_m \left(f^{N,A}_{1,m_1}, f^{N,A}_{2,m_1}, f^{N, A}_{3,n_0}\right)(x)   \right] \\ &=&Re \left[ e^{-2 \pi i A2^{n_0}x}  \sum_{m_1 <<  n_0}~  \sum_{l \in \mathbb{Z}\cap [2, N/4] }  T^l_m \left(f^{N,A}_{1,m_1}, f^{N,A}_{2,m_1}, f^{N,A}_{3,n_0}\right)(x)   \right] \\&+&Re \left[ e^{-2 \pi i A2^{n_0}x}  \sum_{m_1 <<  n_0} ~ \sum_{l \in \mathbb{Z} \cap [2, N/4]^c}  T^l_m \left(f^{N,A}_{1,m_1}, f^{N,A}_{2,m_1}, f^{N,A}_{3,n_0}\right)(x)   \right] \\ &\geq &  C_{A, \alpha} \log(N)- \sum_{m_1 <<  n_0} \left|  \sum_{l \in \mathbb{Z} \cap [2, N/4]^c}  T^l_m \left(f^{N,A}_{1,m_1}, f^{N,A}_{2,m_1}, f^{N,A}_{3,n_0}\right)(x)  \right|.
 \end{eqnarray*}
We further break apart the last sum as follows: 

\begin{eqnarray*}
\sum_{m_1 <<  n_0} \left|  \sum_{l \in \mathbb{Z} \cap [\alpha, N/4]^c}  T^l_m \left(f^{N,A}_{1,m_1}, f^{N,A}_{2,m_1}, f^{ A}_{3,n_0}\right)(x)  \right| &=& \sum_{m_1 <<  n_0} \left|  \sum_{l \in \mathbb{Z} \cap (N/4, 100N] }  T^l_m \left(f^{N,A}_{1,m_1}, f^{N,A}_{2,m_1}, f^{N,A}_{3,n_0}\right)(x)  \right| \\&+& \sum_{m_1 <<  n_0} \left|  \sum_{l \in \mathbb{Z} \cap (100N, \infty) }  T^l_m \left(f^{N,A}_{1,m_1}, f^{N,A}_{2,m_1}, f^{N,A}_{3,n_0}\right)(x)  \right| \\ &+& \sum_{m_1 <<  n_0} \left|  \sum_{l \in \mathbb{Z} \cap [0, 2)}  T^l_m \left(f^{N,A}_{1,m_1}, f^{N,A}_{2,m_1}, f^{ A}_{3,n_0}\right)(x)  \right| \\&+& \sum_{m_1 <<  n_0} \left|  \sum_{l \in \mathbb{Z} \cap (-\infty , 0)}  T^l_m \left(f^{N,A}_{1,m_1}, f^{N,A}_{2,m_1}, f^{ A}_{3,n_0}\right)(x)  \right|  \\ &:=& I + II + III + IV. 
\end{eqnarray*}
To bound $I(x)$, note $|T^l_m \left(f^{N,A}_{1,m_1}, f^{N,A}_{2,m_1}, f^{N, A}_{3,n_0}\right)(x)| \lesssim \frac{1}{(1+A^2 |n_0 - m_1-l|^2) l}$  which ensures 

\begin{eqnarray*}
 |~I(x)~| \lesssim \sum_{m_1 <<  n_0} \sum_{l \in \mathbb{Z} \cap (N/4, 100N] }    \frac{1}{(1+A^2 |n_0 - m_1-l|^2) \cdot l} \lesssim 1. 
\end{eqnarray*}
For $II(x)$, note that by the restrictions placed on $n_0$ and $m_1$ 

\begin{eqnarray*}
|~II(x)~| \lesssim \sum_{m_1 <<  n_0} ~\sum_{l \in \mathbb{Z} \cap (100N, \infty) }   \frac{1}{1+A^3 |n_0 - m_1-l|^3 } \lesssim \sum_{1 \leq m_1 \leq N} \frac{1}{1+A^3 |n_0-m_1|^2} \lesssim 1.
\end{eqnarray*}
Term $III(x)$ is estimated by $T^l_m \left(f^{N,A}_{1,m_1}, f^{N,A}_{2,m_1}, f^{N, A}_{3,n_0}\right)(x)| \lesssim 1$. Indeed, it is trivial for $l \not = 0$. If $l=0$, then one only needs
\begin{eqnarray*}
|T_m \left(f^{N,A}_{1,m_1}, f^{N,A}_{2,m_1}, f^{ A}_{3,n_0}\right)(x) | \lesssim \int_{-\frac{A}{2}}^{\frac{A}{2}} \left| \phi^2(An_0 + \theta_x - Am_1 +t) - \phi^2(An_0 + \theta_x - Am_1 -t) \right|  \frac{dt}{t} \lesssim_A \frac{1}{1+A^2|n_0 -m_1|^2}
\end{eqnarray*}
since $\phi^2 \in \mathcal{S}(\mathbb{R})$ hence its derivative is uniformly bounded and $\phi^2$ is Lipschitz. 
Lastly, it is clear that

\begin{eqnarray*}
|~IV(x)~| \lesssim \sum_{m_1 <<  n_0}~ \sum_{l \in \mathbb{Z} \cap (-\infty , 0)}   \frac{1}{A^3 |n_0 - m_1-l|^3 } \lesssim \sum_{m_1 <<  n_0} \frac{1}{A^3 |n_0 - m_1|^2 } \lesssim 1.
\end{eqnarray*}
Hence, to prove Theorem \ref{IT} assuming the claim, it  is enough to handle two remaining error terms $E_I(x)$ and $E_{II}(x)$ where

\begin{eqnarray*}
E_I(x)&:=&  \sum_{1 \leq m_3 \leq N} ~  \sum_{m_2 << m_3}~ \sum_{ m_1 <m_2 }T_m(f^{N,A}_{1,m_1}, f_{2,m_2}^{N,A}, f_{3,m_3}^{N,A})(x)\\ E_{II}(x)&:=& \sum_{m_3 \not = m_0}~ \sum_{ m_2 << m_3} ~\sum_{m_1 = m_2} T_m(f^{N,A}_{1,m_1}, f_{2,m_2}^{N,A}, f_{3,m_3}^{N,A})(x).
\end{eqnarray*}
We estimate $E_I(x), E_{II}(x)$ separately:

\begin{eqnarray*}
&& |~E_I(x)~| \\ &\leq&\sum_{1  \leq m_3 \leq N} ~\sum_{ m_2 <<m_3}   ~\sum_{m_1 < m_2 }  \left| \left(  \int_\mathbb{R} \phi(x-Am_1 -t) \phi(x-Am_2 -t) e^{2 \pi i (2^{m_1}-2^{m_2}) t} \frac{dt}{t} \right) (\phi(x-Am_3)) \right| \\&\leq& \sum_{1 \leq m_3 \leq N} \sum_{m_2 << m_3} \sum_{ m_1 < m_2} \phi(x-Am_1) \phi(x-Am_2) \phi(x-Am_3) \\ &\lesssim& \sum_{1 \leq m_3 \leq N} \sum_{m_1< m_2} \sum_{m_2 << m_3}  \frac{1}{1+A^2 |n_0-m_1|^2} \frac{1}{1+A^2 |n_0-m_2|^2} \frac{1}{1+A^2 |n_0-m_3|^2} \\ &\lesssim 1. 
\end{eqnarray*}
The remaining error will require A to be sufficiently large: 

\begin{eqnarray*}
&&|~E_{II} (x)~| \\ &=& \left| \sum_{m_3 \not = n_0} \sum_{m_1 = m_2} \sum_{m_2 << m_3} \left(  \int_\mathbb{R} \left[ \prod_{j=1}^2 \phi(x-Am_j -t) \right]e^{-2 \pi i  2^{m_1} (x -t)} e^{2 \pi i  2^{m_2} (x-t)} \frac{dt}{t} \right) \left( \phi(x-Am_3) e^{2 \pi i  2^{m_3} x}\right) \right| \\ &=& \left| \sum_{m_3 \not = n_0} \sum_{m_1 << m_3 }H \left[ \phi^2 \right] (x-Am_1) \cdot \phi(x-Am_3)e^{2 \pi i 2^{m_3}x} \right| \\ &\lesssim& \sum_{m_3 \not = n_0} \sum_{1 \leq m_1 \leq N} \frac{1}{1+A|n_0-m_1|} \frac{1}{1+A^2 |n_0-m_3|^2} \\ &\lesssim& \frac{\log(N)}{A^{3}}. 
\end{eqnarray*}
We can therefore choose $M \in \mathbb{N}$ large enough to achieve the desired point-wise lower bound for the main contribution.

Lastly, to prove the claim, write down $\forall~x \in S_N$ and $\forall ~l \in \mathbb{Z} \cap [2, N/4]$

\begin{eqnarray*}
  Re \left[ e^{-2 \pi i 2^{n_0} x} \sum_{m_1 <<  n_0} T^l_m \left(f^{N,A}_{1,m_1}, f^{N,A}_{2,m_1}, f^{N,A}_{3,n_0}\right)(x) \right] \gtrsim Re \left[ \left( \sum_{m_1 << n_0}  \int_{A(l-\frac{1}{2})}^{A(l+\frac{1}{2})}  \phi^2(An_0 + \theta_x-Am_1 -t)  \frac{dt}{t} \right) \right] .
\end{eqnarray*}
Assuming $n_0 - N/4 \leq m_1 \leq n_0-2$ and $l = n_0 -m_1$ note that  $\int_{A(l-\frac{1}{2})}^{A(l+\frac{1}{2})}  \phi^2(An_0 + \theta_x-Am_1 -t)  \frac{dt}{t} \gtrsim  \frac{1}{Al}. $ As $n_0 \in [N/2, 2N/3]$,  $1 \leq m_1 \leq N$, $m_1 << n_0$, $l \in [2, N/4]$ satisfy our constraints, and the remaining terms in the sum are all positive, the claim is true and, hence, Theorem \ref{IT} has been shown. 


\end{proof}

\section{Unboundedness for Hyperplane Symbols in Dimension $n \geq 5$}

\begin{definition}
Fix $\Phi, \Psi \in \mathcal{S}(\mathbb{R})$ such that $1_{[-1/2,1/2]} \leq \Phi \leq 1_{[-1,1]}$ and $1_{[2, \infty)} \leq \Psi \leq 1_{[1,\infty)}$.  For every pair of distinct hyperplanes $(\Gamma^{\vec{\alpha}}, \Gamma^{\vec{\beta}}) \subset \mathbb{R}^n$ and symbol $m : \mathbb{R}^n \rightarrow \mathbb{C}$, the $(\Phi, \Psi)$-localization of $m$ near $\Gamma^{\vec{\alpha}}$ away from $\Gamma^{\vec{\beta}}$ is the symbol defined by 

\begin{eqnarray*}
m[\vec{\alpha}, \vec{\beta}, \Phi, \Psi] (\vec{\xi}) = \Phi(\vec{\alpha} \cdot \vec{\xi}) \Psi(\vec{\beta} \cdot \vec{\xi})m (\vec{\xi})~~\forall\vec{\xi} \in \mathbb{R}^n.
\end{eqnarray*}
By construction, $m[\vec{\alpha}, \vec{\beta}, \Phi, \Psi] (\vec{\xi}) \in \mathcal{M}_{\Gamma^{\vec{\alpha}}}(\mathbb{R}^n) \cap \mathcal{M}_{\Gamma^{\vec{\beta}}}(\mathbb{R}^n)  $ is supported inside $\{dist(\vec{\xi}, \Gamma^{\vec{\alpha}}) \lesssim 1\}$
 and if $\vec{\xi} \in \{ dist(\vec{\xi}, \Gamma^{\vec{\alpha}}) \lesssim 1\} \cap \{dist(\vec{\xi}, \Gamma^{\vec{\beta}})  \gtrsim dist(\vec{\xi}, \Gamma^{\vec{\alpha}})\}$, then $m_{\Gamma^{\vec{\alpha}}, \Gamma^{\vec{\beta}}}(\vec{\xi})= m(\vec{\xi})$.
\end{definition}


Our main result in this section is
\begin{theorem}\label{LocThm}
Fix $n \geq 5$. Let $\vec{\alpha} , \vec{\beta} \in \mathbb{R}^n$ satisfy $\alpha_j = \beta_j z_j$ for some $\vec{z} \in \mathbb{Z}^n$. 
Assume there exists $\vec{\#} \in \mathbb{R}^n$ s.t.

\begin{eqnarray*}
\sum_{j=1}^n \#_j \alpha_j = \sum_{j=1}^n \#_j \alpha_j^2 = \sum_{j=1}^n \#_j \beta_j = \sum_{j=1}^n \#_j \beta_j^2 &=&0 \\
\sum_{j=1}^n \#_j \alpha_j \beta_j  &\not =& 0  
\end{eqnarray*}
with the additional property that $ \#_j \alpha^2_j \in \mathbb{Q}$ for all $1 \leq j \leq n$. 
Moreover, suppose $K_1, K_2$ are two real-valued kernels for which $\hat{K}_1, \hat{K}_2 \in \mathcal{M}_{\{0\}}(\mathbb{R})$, $K_1, K_2$ are odd, there exist $C_1, C_2 >0$ so that

\begin{eqnarray*}
 K_1(\xi) \geq 0~~~~~~~~\forall \xi >0 \\ 
 K_1(\xi) \geq C_1/ \xi~~~\forall \xi \geq 1\\
\liminf_{s \rightarrow \infty} Im [\check{K}_2 ](s) = C_2.
\end{eqnarray*}
Then every $(\Phi, \Psi)$-localization of $ \hat{K}_1(\vec{\alpha} \cdot \vec{\xi}) \hat{K}_2(\vec{\beta} \cdot \vec{\xi}) : \mathbb{R}^2 \rightarrow \mathbb{C}$ gives rise to a multiplier $T_{ \hat{K}_1(\vec{\alpha} \cdot \vec{\xi}) \hat{K}_2(\vec{\beta} \cdot \vec{\xi} )[\vec{\alpha}, \vec{\beta}, \Phi, \Psi]}$ satisfying no $L^p$ estimates.
\end{theorem}

\begin{lemma}\label{IRL2}
Let $n \geq 5$ and $\vec{\alpha} \in \mathbb{R}^n$ satisfy  $\alpha_j^{-1} =q_j \alpha + q_j^2$ for some $\vec{q} \in \mathbb{Q}^n$ with distinct, non-zero entries such that $q_j \alpha \not = -1$ for all $1 \leq j \leq n$.  Moreover, let $\beta_j = \alpha_j q^{-1}_j $. Then there exists a non-trivial solution $\vec{\#} \in \mathbb{R}^n$ to the system 
\end{lemma}

\begin{eqnarray*}
\sum_{j=1}^n \#_j \alpha_j =  \sum_{j=1}^n \#_j \alpha_j^2 = \sum_{j=1}^n \#_j \beta_j =  \sum_{j=1}^n \#_j \beta_j^2 &=& 0 \\ 
\sum_{j=1}^n \#_j \alpha_j \beta_j &\not =& 0 
\end{eqnarray*}
with the property that $\#_j \alpha_j^2 \in \mathbb{Q}$ for all $1 \leq j \leq n$. 

\begin{proof}
Set $\tilde{\#}_j = \#_j \alpha_j^2$ for all $1 \leq j \leq n$. We seek non-trivial  $\vec{\tilde{\#}}\in \mathbb{Q}^n$ satisfying

\begin{eqnarray*}
\sum_{j=1}^n \tilde{\#}_j \alpha^{-1}_j = \sum_{j=1}^n \tilde{\#}_j (q_j \alpha + q^2_j)= 0 ; ~~\sum_{j=1}^n \tilde{\#}_j  &=& 0 \\ 
\sum_{j=1}^n\tilde{ \#}_j \alpha_j^{-2} \beta_j = \sum_{j=1}^n \tilde{\#}_j (q_j \alpha + q_j^2)q^{-1}_j = 0 ;~~ \sum_{j=1}^n \tilde{\#}_j \alpha_j^{-2} \beta_j^2 = \sum_{j=1}^n \tilde{\#}_j q_j^{-2}&=& 0 \\
\sum_{j=1}^n\tilde{ \#}_j \alpha_j^{-1} \beta_j  = \sum_{j=1}^n \tilde{\#_j}q^{-1}_j & \not =& 0.  
\end{eqnarray*}
Requiring $\sum_{j=1}^n \tilde{\#}_j q_j^m =0$ for all $m \in \{ -2, 0, 1, 2\}$, we may for $n \geq 5$ find $\vec{\tilde{\#}}_j \in \mathbb{Q}^n$ by the Gram-Schmidt process for which $\sum_{j=1}^n \tilde{\#}_j q_j^{-1} \not = 0$.

 \end{proof}

Combining Theorem \ref{LocThm} and Lemma \ref{IRL2}, we obtain
\begin{reptheorem}{MT*}
Let $n \geq 5$ and $\vec{\alpha} \in \mathbb{R}^n$ satisfy $\alpha_j^{-1} = q_j  + \alpha q_j^2$ for some $\vec{q} \in \mathbb{Q}^n$ with distinct, non-zero entires such that $q_j \alpha \not = -1$ for all $1 \leq j \leq n$. Then there exists $m \in \mathcal{M}_{\Gamma^{\vec{\alpha}}}(\mathbb{R}^n)$ supported in $\left\{ \vec{\xi}: dist(\vec{\xi}, \Gamma^{\vec{\alpha}}) \lesssim 1\right\}$ such that $T_m$ satisfies no $L^p$ estimates. 
\end{reptheorem}

Note that if $\alpha \in \mathbb{R} \cap \mathbb{Q}^c$, then $\vec{\alpha} \not \in \mathbb{R}(\mathbb{Q}^n)$. Setting $\alpha =0$ yields that for any $\vec{q} \in \mathbb{Q}^n$ with distinct, non-zero entries, there exists $m \in \mathcal{M}_{\Gamma^{\vec{q}}}(\mathbb{R}^n)$ such that $T_m$ satisfies no $L^p$ estimates. Hence, we have

\begin{corollary}[Generic $n-$linear Hilbert transform satisfies no $L^p$ estimates]
Let $n \geq 5$ and $\alpha_j = j$ for all $1 \leq j \leq n$. Then there exists $m \in \mathcal{M}_{\Gamma^{\vec{\alpha}}}(\mathbb{R}^n)$ supported in $\left\{ \vec{\xi}: dist(\vec{\xi}, \Gamma^{\vec{\alpha}}) \lesssim 1\right\}$ such that $T_m$ satisfies no $L^p$ estimates. 
\end{corollary}

\begin{proof}\textit{(Theorem \ref{LocThm})}. 

\subsection{PART 1: The Rational Case} Assuming $\vec{\alpha}, \vec{\beta} \in \mathbb{Q}^n$, we may always dilate $\vec{\alpha}, \vec{\beta}$  by sufficiently large integers to guarantee $\vec{\alpha} , \vec{\beta} \in \mathbb{Z}^n$. Then our assumption is that there exists $\vec{\#} \in \mathbb{Z}^n$ solving the system 

\begin{eqnarray*}
\sum_{j=1}^n \#_j \alpha_j = \sum_{j=1}^n \#_j \alpha_j^2 = \sum_{j=1}^n \#_j \beta_j = \sum_{j=1}^n \#_j \beta_j^2 &=&0 \\
\sum_{j=1}^n \#_j \alpha_j \beta_j  &\not =& 0.   
\end{eqnarray*}
For $A \in \mathbb{Z}^+$ and $j \in \{1, ..., d\}$ , construct for each $1 \leq j \leq n$ the function

\begin{eqnarray*}
f^{N, A, \#_j}(x) = \sum_{-N \leq m \leq N} \phi(x-A m) e^{2 \pi i A \#_j  m x},
\end{eqnarray*} 
where we now wish to choose $\phi \in \mathcal{S}(\mathbb{R})$ satisfying $supp~ \hat{\phi} \subset [-1,1]$, $\phi \geq 0$, $\phi(0) \not = 0$, and $\phi$ is symmetric about the origin. This is easily done by taking one's favorite non-trivial real-valued non-negative smooth function, symmetric about the origin with compact support in $[-\frac{1}{2},\frac{1}{2}]$, say $\Phi$. Then one need only set $\hat{\phi} = \Phi*\Phi $ to observe 

\begin{eqnarray*}
\phi := \mathcal{F}^{-1} (\Phi* \Phi)= \check{\Phi}^2 \geq 0
\end{eqnarray*}
with $\phi(0) = || \Phi||_{L^1}^2 >0$ in addition to the desired Fourier support and symmetry properties. Moreover, we want to choose $m \in \mathcal{M}_{\vec{\alpha}}(\mathbb{R}^d)$ to satisfy the property that is identically equal to $1_{ \vec{\xi} \cdot \vec{\alpha} \geq 0}1_{\vec{\xi} \cdot \vec{\beta} \leq 0}$ in some neighborhood of the hyperplane $\Gamma^{\vec{\alpha}}$ away from the singularity $\Gamma^{\vec{\beta}}$. That is, $m$ is supported inside $\{dist(\vec{\xi}, \Gamma^{\vec{\alpha}}) \lesssim 1\}$ and if $\vec{\xi} \in \{ dist(\vec{\xi}, \Gamma^{\vec{\alpha}}) \lesssim 1\} \bigcap \{dist(\vec{\xi}, \Gamma^{\vec{\beta}})  \gtrsim dist(\vec{\xi}, \Gamma^{\vec{\alpha}})\}$, then $m(\vec{\xi})= 1_{\vec{\xi} \cdot \vec{\alpha} \geq 0} (\vec{\xi})1_{\vec{\xi} \cdot \vec{\beta} \leq  0}(\vec{\xi})$. The parameter A is to be taken sufficiently large to give us sparseness in frequency and time, which enables to assume that the only $\vec{m}$ which may appears in the integrand of the kernel representation of T are \emph{precisely} those for which both $\sum_{j=1}^n \#_j m_j \alpha_j =0$ and $\sum_{j=1}^n \#_j m_j \beta_j < 0$. Moreover, assuming $\sum_{j=1}^n \#_j m_j \beta_j < 0$ for each such $\vec{m}$, the multiplier $m(\vec{\xi})$  is identically $1_{\vec{\xi} \cdot \vec{\alpha} \geq 0} 1_{\vec{\xi} \cdot \beta \leq 0}$ when restricted to the domain $\prod_{j=1}^n [A\#_j m_j-1, A\#_j m_j +1]$. 
Recall that the n-linear operator given by
\begin{eqnarray*}
T(f_1, , ..., f_n)(x) :=p.v.  \int_{\mathbb{R}^2} \prod_{j=1}^n f_j(x-\alpha_j t - \beta_j s) K_1(t) K_2(s) ds ~dt
\end{eqnarray*}
 has the Fourier representation (defined on Schwartz functions, say) given by

\begin{eqnarray*}
 \int_{\mathbb{R}^n}\hat{K}_1(\vec{\alpha} \cdot \vec{\xi}) \hat{K}_2(\vec{\beta} \cdot \vec{\xi}) \left[ \prod_{j=1}^n \hat{f}_j(\xi_j)  e^{2 \pi i \xi_j x}  \right] d\vec{\xi}.
\end{eqnarray*}
Thus, we may rewrite $T_{ \hat{K}_1(\vec{\alpha} \cdot \vec{\xi}) \hat{K}_2(\vec{\beta} \cdot \vec{\xi})[\vec{\alpha}, \vec{\beta}, \Phi, \Psi]}$ acting on $\vec{f}^{N,A , \vec{\#}}$ as 

\begin{eqnarray*}
&& T_{ \hat{K}_1(\vec{\alpha} \cdot \vec{\xi}) \hat{K}_2(\vec{\beta} \cdot \vec{\xi})[\vec{\alpha}, \vec{\beta}, \Phi, \Psi]}\left(\vec{f}^{N,A , \vec{\#}}\right) (x)\\&=& \sum_{(k,l) \in \mathbb{Z}^2} \sum_{\vec{m} \in \mathbb{M}} \int_{A(k-\frac{1}{2})} ^{A(k+\frac{1}{2})} \int_{A(l-\frac{1}{2})} ^{A(l + \frac{1}{2})}\prod_{j=1}^n \left[\phi(x-Am_j - \alpha_j t - \beta_j s)  e^{2 \pi i A  \#_j m_j (x-\alpha_j t - \beta_j s)} \right]K_1(s) K_2(t) ds dt,
\end{eqnarray*}
where
\begin{eqnarray*}
\mathbb{M} &:=& \left\{ \vec{m}  \in \mathbb{Z}^n:-N \leq m_j \leq N~\forall~j \in \{1, ...,n\}, ~\sum_{j=1}^n m_{j} \#_j \alpha_j  = 0,~ \sum_{j=1}^n m_j \#_j \beta_j < 0\right\} .
\end{eqnarray*}

\begin{lemma}\label{RdL2}
Let 
\begin{eqnarray*}
T^{k,0}_{ \hat{K}_1(\vec{\alpha} \cdot \vec{\xi}) \hat{K}_2(\vec{\beta} \cdot \vec{\xi})[\vec{\alpha}, \vec{\beta}, \Phi, \Psi]}  \left(\left\{ f_{m_j}^{N,A , \#_j} \right\}_{j=1}^n\right) (x) := \int_{A(k-\frac{1}{2})}^{A(k+\frac{1}{2})}  \int_{-\frac{A}{2}}^{\frac{A}{2}}\left[  \prod_{j=1}^n f^{N, A, \#_j}_{ m_j}(x-\alpha_j t -\beta_j s) \right] K_1(s) K_2(t)~ ds dt.
\end{eqnarray*}
To prove Theorem \ref{MT*}, it suffices to show $\exists c_{\vec{\alpha}, \vec{\beta}} >0$ such that for every $x \in \left[ An_0, An_0 + \frac{c_{\vec{\alpha}, \vec{\beta}}}{A} \right]$, $k \in \left[ k_0, \frac{N}{3 \max_{1 \leq j \leq n} \{ |\alpha_j| \} } \right]$, and $n_0 \in [-N/3,N/3]$, the following estimate holds:

\begin{eqnarray*}\label{Rd2}
\sum_{1 \leq k \leq \frac{N}{C_{\vec{\alpha}, \vec{\beta}}}} Im\left[ e^{-2 \pi i A\left(\sum_{j=1}^n \#_j\right) n_0 x} T^{(k,0)}_{ \hat{K}_1(\vec{\alpha} \cdot \vec{\xi}) \hat{K}_2(\vec{\beta} \cdot \vec{\xi})[\vec{\alpha}, \vec{\beta}, \Phi, \Psi]}  \left( \left\{ f_{m_j}^{N,A , \#_j} \right\}_{j=1}^n\right) (x)\right]  \gtrsim \frac{\log(N)}{A}.
\end{eqnarray*}

\end{lemma}
Before proving the lemma, we verify the lower bound in the above display.

\subsection{Main Contribution}

The condition for the term corresponding to $\vec{m}$ to be in our sum is $0 \leq A\sum_{j=1}^n \#_j m_j \alpha_j \lesssim 1$, which gives that $\sum_{j=1}^n \#_j m_j \alpha_j =0$ by the integrality of $\vec{\#}, \vec{\alpha}$ and $\vec{m}$. Now choose any $c_{\vec{\alpha}, \vec{\beta}} \in \mathbb{N}$ such that $c_{\vec{\alpha}, \vec{\beta}} > 10 ( max_{j \in \{1, ..., n\}} \{|\alpha_j|\}+ max_{j \in \{1, ..., n\}} \{|\beta_j|\} )$ and fix $x \in \left[An_0, An_0+\frac{c_{\vec{\alpha}, \vec{\beta}}}{A}\right]$ for some $n_0 \in [-N/3, N/3]$. We expect the largest contribution to come from the terms where $n_0 - m_j - \alpha_j k - \beta_j l =0$ for each $j \in \{1, ..., n\}$. Of course, by the frequency restrictions imposed by $0 \leq \vec{\xi} \cdot \vec{\alpha} \lesssim 1$ and $\vec{\xi} \cdot \vec{\beta}  \leq0$, we have that $l=0$ and $k >1$ for A sufficiently large.  To this end, construct $\vec{m}_{k,l} := n_0 - \vec{\alpha} k - \vec{\beta} l$ for each $1 \leq k \leq \frac{N}{C_{\vec{\alpha}, \vec{\beta}}}$. Then, observe

\begin{eqnarray*}
\sum_{j=1}^n \#_j m_j \alpha_j = \sum_{j=1}^n \#_j (n_0-\alpha_j k - \beta_j l  )\alpha_j =0
\end{eqnarray*}
must be satisfied, which, provided $\sum_{j=1}^n \alpha_j \beta_j \#_j \not = 0$, requires $l = 0$. For the condition $\sum_{j=1}^n \#_j m_j \beta_j <0$, we require $k \geq  1$. (Take A sufficiently large to guarantee that all and only those $k \geq 1$ arise in the in sum over $k,l$ corresponding to $m_j (k,l)= n_0 - \alpha_j k - \beta_j l$.) The main contribution can be expressed in this new notation as 

\begin{eqnarray*}
\sum_{1 \leq k \leq \frac{N}{C_{\vec{\alpha}, \vec{\beta}}}} T^{k,0}_{ \hat{K}_1(\vec{\alpha} \cdot \vec{\xi}) \hat{K}_2(\vec{\beta} \cdot \vec{\xi})[\vec{\alpha}, \vec{\beta}, \Phi, \Psi]}  \left( \left\{ f_{m_j(n_0, k,0)}^{N,A ,\#_j} \right\}_{j=1}^n\right) (x)
\end{eqnarray*}
 Let $\sum_{j=1}^n \#_j \alpha_j \beta_j=: \mathfrak{C}>0$. Fix a single term in the above sum, set $C( x)=e^{-2 \pi i A(\sum_{i=1}^n \#_i)n_0  x}$ and compute using Lemma \ref{ML7}

\begin{eqnarray*}
&& Im\left[C(x) \cdot T^{k,0}_{ \hat{K}_1(\vec{\alpha} \cdot \vec{\xi}) \hat{K}_2(\vec{\beta} \cdot \vec{\xi})[\vec{\alpha}, \vec{\beta}, \Phi, \Psi]}\left( \left\{ f_{ m_j(n_0, k,0)}^{N,A , \#_j} \right\}_{j=1}^n\right) (x)\right] \\&=& \left. Im \left[ C( x) \cdot \int_{A(k-\frac{1}{2})} ^{A(k+\frac{1}{2})} \int_{A(l-\frac{1}{2})} ^{A(l + \frac{1}{2})} \prod_{j=1}^n \left[ \phi(x-Am_j- \alpha_j t - \beta_j s)  e^{2 \pi i A \#_j m_j (x-\alpha_j t - \beta_j s)} \right]\ K_1(s) K_2(t) ds dt\right]   \right|_{m_j = n_0 -\alpha_j  k} \\&=& Im\left[  \int_{-\frac{A}{2}} ^{\frac{A}{2}} \int_{-\frac{A}{2}} ^{\frac{A}{2}} \prod_{j=1}^n \left[ \phi(\theta_x-\alpha_j t - \beta_j s) \right] e^{2 \pi i  A\mathfrak{C} ks}K_1(s) K_2(t + Ak) ds dt\right] \\ &\gtrsim& \frac{D_{\phi, \vec{\alpha}, A, C_1, C_2 }}{Ak} \\&\gtrsim& \frac{D_{\phi, \vec{\alpha}, \infty, C_1, C_2}}{Ak}.
\end{eqnarray*}
Summing over $1 \leq k \leq \frac{N}{c_{\vec{\alpha}, \vec{\beta}}}$ yields a total contribution  $\sim \frac{\log(N)}{A}$. Therefore, setting 

\begin{eqnarray*}
\mathfrak{R}(n_0)= \left[ \bigcup_{ k \in \left[1, \frac{N}{3 \max_{1\leq j \leq n} \{ |\alpha_j| \} } \right]} \{k\} \times \{0\} \times \left\{\vec{m}(n_0, k,0) \right\} \right]^c \bigcap  \left( \mathbb{Z} \times \mathbb{Z} \times \mathbb{M} \right),
\end{eqnarray*} we are left with satisfactorily estimating 

\begin{eqnarray*}
Im \left[ e^{-2 \pi i A (\sum_{j=1}^n \#_j) n_0 x} \sum_{(k,l, \vec{m}) \in \mathfrak{R}(n_0)} T^{k,l}_{ \hat{K}_1(\vec{\alpha} \cdot \vec{\xi}) \hat{K}_2(\vec{\beta} \cdot \vec{\xi})[\vec{\alpha}, \vec{\beta}, \Phi, \Psi]}\left( \left\{ f_{m_j}^{N,A , \#_j} \right\}_{j=1}^n \right) (x) \right].
\end{eqnarray*}
 The rest of this section is dedicated to proving Lemma \ref{RdL2}.

\subsection{Small Perturbations}
\subsubsection{$1 \leq k \lesssim N$}
For each $n_0 \in [N/2,2 N/3]$, $k \in \left[ 1, \frac{N}{3 \max_{1 \leq j \leq n} \{ |\alpha_j| \} } \right]$, and $l \in \mathbb{Z}$ consider those $\vec{m} \in \mathbb{M}$ satisfying $m_j= m_j (n_0, k,l) + \Delta_j$ where $0< \sup_{1 \leq j \leq n} |\Delta_j| \leq max_{1 \leq j \leq n} \{ |\alpha_j|+ |\beta_j|\}$ along with $\sum_{j=1}^n \#_j m_j \alpha_j =0$ and $\sum_{j=1}^n \#_j m_j \beta_j <0$. For each $(k,l) \in \mathbb{Z} \times \mathbb{Z}$, denote this collection by $\mathbb{M}^{S}_{n_0, k,l}$ and define the collection of large perturbations by $\mathbb{M}^L_{n_0, k,l}$ by the relation

\begin{eqnarray*}
\mathbb{M} &=& \mathbb{M}^S_{n_0, k,l} \coprod \mathbb{M}^L_{n_0, k,l}\coprod \left(  \left\{ \vec{m}(n_0, k,0) \right\} \cap \mathbb{M} \right) 
\end{eqnarray*}
Hence, for every $n_0 \in \mathbb{Z}$

\begin{eqnarray*}
\mathbb{Z} \times \mathbb{Z} \times \mathbb{M}= \left[ \bigcup_{(k,l)  \in \mathbb{Z}^2}  \{ k\} \times \{l\} \times \mathbb{M}^S_{n_0, k,l} \right]  \coprod \left[ \bigcup_{(k,l) \in \mathbb{Z}^2} \{ k\} \times \{l\} \times \mathbb{M}_{n_0, k,l}^{L} \right] \coprod \left[ \bigcup_{k \in \mathbb{Z}} \{k\} \times \{0\}\times  \{ \vec{m}(n_0, k,0) \cap \mathbb{M}\}\right] .
\end{eqnarray*}
Now observe for every $\vec{m} \in \mathbb{M}_{n_0, k,l}^S$
\begin{eqnarray*}
\sum_{j=1}^n \#_j m_j (x-\alpha_j t - \beta_j s)&=& \sum_{j=1}^n \#_j m_j (x-\beta_j s) \\ &=&  \sum_{j=1}^n \#_j m_j (n_0, k,l)(x - \beta_j s) + \sum_{j=1}^n \#_j \Delta_j  (x- \beta_j s) \\ &=& \left(\sum_{j=1}^n \#_j \alpha_j \beta_j \right)  ks + \left( \sum_{j=1}^n \#_j (\Delta_j+n_0) \right) x  - \left( \sum_{j=1}^n \#_j \Delta_j \beta_j \right) s. 
\end{eqnarray*}
Therefore, setting $C(x, \vec{\Delta}) = e^{-2 \pi i A \left[ \sum_{j=1}^n \#_j (\Delta_j+n_0)\right] x}$, it follows that
\begin{eqnarray*}
&& C_{\vec{\Delta}}( x) \cdot T^{k,l}_{ \hat{K}_1(\vec{\alpha} \cdot \vec{\xi}) \hat{K}_2(\vec{\beta} \cdot \vec{\xi})[\vec{\alpha}, \vec{\beta}, \Phi, \Psi]}\left( \left\{ f_{j,m_j}^{N,A , \#_j} \right\}_{j=1}^n \right)  (x)
\\&=&\int_{A(k-\frac{1}{2})} ^{A(k+\frac{1}{2})} \int_{A(l-\frac{1}{2})} ^{A(l+\frac{1}{2})}\left[ \prod_{j=1}^n \left[\phi(x-Am_j - \alpha_j t - \beta_j s)  \right]  \right] e^{2 \pi i A\left[  \sum_{j=1}^n \#_j (\alpha_j k-\Delta_j)  \beta_j \right] s}  K_1(s) K_2(t) ds dt. 
\end{eqnarray*}
Fix $\epsilon>0$ and choose $c_{\vec{\alpha}, \vec{\beta}}$ sufficiently small to ensure that $\forall x \in \left[An_0 , An_0 + \frac{c_{\vec{\alpha}, \vec{\beta}}}{A} \right]$

\begin{eqnarray*}
Re\left[ e^{2 \pi i A \sum_{j=1}^n \#_j \Delta_j x} \right] = Re\left[  e^{2 \pi I A \left( \sum_{j=1}^n \#_j \Delta_j \right)\theta_x} \right]  > \gamma~for~some~\gamma: |\gamma - 1|<\epsilon << 1.
\end{eqnarray*}
Then $Im \left[ e^{2 \pi i A \left( \sum_{j=1}^n \#_j \Delta_j \right)x} \right] <\sqrt{1-\gamma^2} $. Fixing $n_0, k$ and summing the quantity
 \begin{eqnarray*}
 && Im \left[ e^{-2 \pi i A \sum_{j=1}^n \#_j n_0 x} T^{(k,l)}_{ \hat{K}_1(\vec{\alpha} \cdot \vec{\xi}) \hat{K}_2(\vec{\beta} \cdot \vec{\xi})[\vec{\alpha}, \vec{\beta}, \Phi, \Psi]}  (\vec{f}_{\vec{m}}^{N,A , \vec{\#}}) (x)   \right]\\ \geq && \gamma Im \left[  \int_{A(k-\frac{1}{2})} ^{A(k+\frac{1}{2})} \int_{A(l-\frac{1}{2})} ^{A(l+\frac{1}{2})}\left[ \prod_{j=1}^n \phi(x-Am_j-\alpha_j t-\beta_j s)   e^{2 \pi i A \#_j (\alpha_j k - \Delta_j)\beta_j  s}  \right] K_1(s) K_2(t) ds dt \right] \\-&& \sqrt{1-\gamma^2} \left| \int_{A(k-\frac{1}{2})} ^{A(k+\frac{1}{2})} \int_{A(l-\frac{1}{2})} ^{A(l+\frac{1}{2})} \left[\prod_{j=1}^n\phi(x-Am_j - \alpha_j t - \beta_j s)  \right] e^{2 \pi i A\left[ \sum_{j=1}^n \#_j (\alpha_jk - \Delta_j ) \beta_j \right] s} K_1(s) K_2(t) ds dt \right| \\ \geq &&-  \sqrt{1-\gamma^2} \left| \int_{A(k-\frac{1}{2})} ^{A(k+\frac{1}{2})} \int_{A(l-\frac{1}{2})} ^{A(l+\frac{1}{2})} \left[\prod_{j=1}^n\phi(x-Am_j - \alpha_j t - \beta_j s)  \right] e^{2 \pi i A\left[ \sum_{j=1}^n \#_j (\alpha_jk - \Delta_j ) \beta_j \right] s} K_1(s) K_2(t) ds dt \right| 
\end{eqnarray*}
over all $l \in \mathbb{Z}$ and $\vec{m} \in \mathbb{M}^S_{n_0, k,l}$ yields for all sufficiently large $k$ and sufficiently small $\epsilon(A)$ a lower bound of $-\frac{1}{A^2k}$. Indeed, this is true for the following $3$ reasons: Lemma \ref{ML7} guarantees that $\gamma Im\left[ \cdot \right] \geq 0$;  for fixed $(n_0,k,l)$, $ \# \left[ \mathbb{M}^S_{n_0, k, l} \right] = O(1)$;  for fixed $(n_0, k)$, the collection $\mathbb{M}^S_{n_0, k,l}$ is empty for all but $O(1)$ distinct $l \in \mathbb{Z}$. To observe this last claim, write down $\vec{m} = \vec{m} (n_0, k,l) + \vec{\Delta} \in \mathbb{M}$. Then $\sum_{j=1}^n \#_j m_j \alpha_j =0$ implies $-\mathfrak{C}l + \sum_{j=1}^n \#_j \alpha_j \Delta_j=0$.   Therefore, $|l| \lesssim_{\vec{\alpha}} 1$, and the small perturbations $\mathbb{M}^S_{n_0,k,l}$ produce an acceptable error when $1 \leq k \lesssim N$. 

\subsubsection{$k <<0$}
If $\sum_{j=1}^n \#_j m_j \beta_j <0$ and $\vec{m} \in \mathbb{M}^S_{n_0, k,l}$ for some $k <<0$, then

\begin{eqnarray*}
0 > \sum_{j=1}^n \#_j m_j \beta_j = - \mathfrak{C} k + \sum_{j=1}^n \#_j \Delta_j \beta_j = -\mathfrak{C} k + O(1).
\end{eqnarray*}
Therefore, $\mathbb{M}^S_{n_0, k,l} = \emptyset$ for all $k <<0$.  

\subsubsection{$|k| \lesssim 1$}
The total contribution is $O(1)$, which is acceptable error in light of the main contribution $O(\log(N))$. 
\subsubsection{$k \simeq N$}
The total contribution is $O(1)$, which is acceptable error in light of the main contribution. 
\subsubsection{$k >> N$}
If $\sum_{j=1}^n \#_j m_j \beta_j <0$ and $\vec{m} \in \mathbb{M}^S_{n_0, k,l}$ for some $k >>N$, then 

\begin{eqnarray*}
m_j = m_j(n_0, k,l) + \Delta_j = n_0 - \alpha_j k - \beta_j l+\Delta_j. 
\end{eqnarray*}
From previous considerations, $|l| =O(1).$ Therefore, $|m_j| \geq |\alpha_j k| - n_0 -|\beta_j l| >>N$, which contradicts $\vec{m} \in \mathbb{M}$. Hence, $\mathbb{M}^S_{n_0, k,l} = \emptyset$ for all $ k >>N.$

\subsection{Large Perturbations}

It now suffices to bound the contribution of $\mathbb{M}_{n_0, k,l}^L$. 
\subsubsection{$k < 0$}

For each $(k,l) \in \mathbb{Z} \times \mathbb{Z}, ~\exists j_* \in \{1, ..., n\}$ satisfying 

\begin{eqnarray*}
|m_{j_*} - n_0 - \alpha_{j_*} k- \beta_{j_*} l| \geq \frac{|l|}{5 n \cdot \sup_{1 \leq j \leq n} |\#_j \alpha_j|}.
\end{eqnarray*}
The proof is an easy contradiction argument. If $|m_j - n_0 - \alpha_{j} k - \beta_j l| < \frac{|l|}{5n \cdot \sup_{1 \leq j \leq n} |\#_j \alpha_j|}$ for each $j \in \{1, ..., n\}$, 

\begin{eqnarray*}
\sum_{j=1}^n \#_j m_j \alpha_j = \sum_{j=1}^n \#_j (m_j - n_0 - \alpha_j k - \beta_j l) \alpha_j + \sum_{j=1}^n \#_j (n_0 + \alpha_j k + \beta_j l) \alpha_j := I + II. 
\end{eqnarray*}
By assumption, $\left| ~I~\right| \leq n \left[ \sup_{1 \leq j \leq n} |\#_j \alpha_j| \right] \frac{|l|}{5n \cdot \sup_{1 \leq j \leq n} |\#_j \alpha_j|} = \frac{|l|}{5}.$ Moreover, $II \geq l$. Therefore, $\sum_{j=1} ^n \#_j m_j \alpha_j \geq \frac{4|l|}{5}$, which contradicts the frequency restrictions imposed by the multiplier. Therefore, one is free to extract the decay $\frac{1}{l^{\tilde{C}}}$ for  any $C>>1$. Similarly,  for each $(k,l) \in \mathbb{Z}^- \times \mathbb{Z}, ~\exists j_* \in \{1, ..., n\}$ such that
\begin{eqnarray*}
 |m_{j_*} - n_0 - \alpha_{j_*} k- \beta_{j_*}l| \geq \frac{|k|}{5n\cdot \sup_{1 \leq j \leq n} |\#_j \beta_j|}. 
\end{eqnarray*} Indeed, suppose $|m_j - n_0 - \alpha_j k| < \frac{|k|}{5n}$ for all $j \in \{1, ..., n\}$. Then $\sum_{j=1}^n m_j \#_j \beta_j \leq 0$ and yet 

\begin{eqnarray*}
\left( \sum_{j=1}^n \#_j m_j \beta_j\right) \geq -\frac{4k}{5} \left( \sum_{j=1}^n \#_j \alpha_j \beta_j \right)  >0
\end{eqnarray*}
This contradicts the restriction that $\sum_{j=1}^n \#_j m_j \beta_j < 0$. If $k=0$, then Hence, summing over all $\vec{m} \in \mathbb{M}^L_{n_0, k,l}$ along with $(k,l) \in \mathbb{Z}^- \times \mathbb{Z}$ yields an acceptable error term, i.e.

\begin{eqnarray*}
&& \sum_{k \leq 0} \sum_{l \in \mathbb{Z}} \sum_{\vec{m} \in \mathbb{M}^L_{n_0, k,l}} \left|  T^{k,l}_{ \hat{K}_1(\vec{\alpha} \cdot \vec{\xi}) \hat{K}_2(\vec{\beta} \cdot \vec{\xi})[\vec{\alpha}, \vec{\beta}, \Phi, \Psi]}\left(\vec{f}_{\vec{m}}^{N,A , \vec{\#}}\right) (x) \right|\\
 &\leq& \sum_{k \leq 0} \sum_{l \in \mathbb{Z}}\sum_{\vec{m} \in \mathbb{M}} \left| \int_{A(k-\frac{1}{2})} ^{A(k+\frac{1}{2})} \int_{A(l-\frac{1}{2})} ^{A(l+\frac{1}{2})}\prod_{j=1}^n \left[\phi(x-Am_j - \alpha_j t - \beta_j s)  e^{2 \pi i A  \#_j m_j (x-\alpha_j t - \beta_j s)} \right] K_1(s) K_2(t) ds dt \right| \\& \lesssim & \sum_{k,l \in \mathbb{Z}} \sum_{\vec{n} \in \mathbb{Z}^n} \frac{1}{(1+|l|^{\tilde{C}} )(1+|k|^{\tilde{C}})} \prod_{j=1}^n \frac{1}{1+|n_j|^{\tilde{C}}} \\&\lesssim& 1. 
 \end{eqnarray*}

\subsubsection{$1\leq k \lesssim N$}
This contribution is slightly more delicate and will feature the same $\log(N)$ growth as the main contribution. Note that $\sum_{j=1}^n \#_j \alpha_j m_j = 0$ requires the existence of some index $j _*\in \{ 1, ..., n\}$ for which 

\begin{eqnarray*}
\left|  m_{j_*} - n_0- \alpha_{j_*} k - \beta_{j_*} l \right| \geq  \frac{|~l~|}{5n \cdot \sup_{1 \leq j \leq n} |\#_j \alpha_j|}. 
\end{eqnarray*}
Hence, the total contribution will be $O\left( \frac{\log(N)}{A^M} \right)$, which is acceptable in light of the $O\left(\frac{\log(N)}{A}\right)$ contribution of the main terms by taking a large enough absolute constant $A$. 

\subsubsection{$k \simeq N$}
As before, we shall have $O\left( \frac{1}{1+l^{\tilde{C}}}\right)$ decay. The summation over $k \simeq N$ is harmless owing to $\sum_{k \simeq N}\frac{1}{k} \simeq 1$. 

\subsubsection{$k >>N$}
The decay is $O\left( \frac{1}{l^{\tilde{C}}} \right) \cdot O\left( \frac{1}{|k|^{\tilde{C}}}\right)$.  This concludes the estimates for 
\begin{eqnarray*}
Im \left[ e^{-2 \pi i A\left( \sum_{j=1}^ n\#_j \right)n_0 x } T^{k,l}_{ \hat{K}_1(\vec{\alpha} \cdot \vec{\xi}) \hat{K}_2(\vec{\beta} \cdot \vec{\xi})[\vec{\alpha}, \vec{\beta}, \Phi, \Psi]}\left(\vec{f}_{\vec{m}}^{N,A , \vec{\#}}\right) (x) \right]
\end{eqnarray*}
and hence the proofs of Lemma \ref{ML} and Theorem \ref{MT*} in the rational case. 
\subsection{PART 2: The Irrational Case}

Let $\vec{\alpha}, \vec{\beta} \in \mathbb{R}^n$ satisfy $\beta_j =  z_j \alpha_j$ for all $1 \leq j \leq n$ where $\vec{z} \in \mathbb{Z}^n$. Assume there exists  $\vec{\#} \in \mathbb{R}^n$ such that

\begin{eqnarray*}
\sum_{j=1}^n \#_j \alpha_j = \sum_{j=1}^n \#_j \alpha_j^2 = \sum_{j=1}^n \#_j \beta_j = \sum_{j=1}^n \#_j \beta_j^2 &=&0 \\
\sum_{j=1}^n \#_j \alpha_j \beta_j  & \not =& 0
\end{eqnarray*}
in addition to the condition $\#_j \alpha_j^2  \in \mathbb{Q}$. By dilating $\vec{\#}$, we may assume $\#_j \alpha_j^2 \in \mathbb{Z}$ for all $ 1\leq j \leq n$ and hence $\#_j \alpha_j \beta_j, \#_j \beta_j^2 \in \mathbb{Z}$ for all $ 1 \leq j \leq n$. 
For $A \in \mathbb{Z}^+$ and $j \in \{1, ..., d\}$ , construct the functions 

\begin{eqnarray*}
f^{N, A, \#, \alpha}(x) = \sum_{-N \leq m \leq N} \phi(x-A \alpha m) e^{2 \pi i A \# \alpha m x}.
\end{eqnarray*} 
Moreover, set $S(\vec{\alpha}) = \{ 1/\alpha_1,...,  1/ \alpha_n\}, \rho=1/A^2$, and $Bohr_{c(\vec{\alpha} )N}(S(\vec{\alpha}), \rho) = \{ N_1, ..., N_{|Bohr_{c(\vec{\alpha}) N}(S(\alpha), \rho)|} \}$. Just as before, $\left| Bohr_{c(\vec{\alpha})N}(S(\vec{\alpha}), \rho)  \right|\simeq_A N$. Fix $n_0 \in Bohr_{c(\vec{\alpha})N}(S(\vec{\alpha}), \rho)$ and set $\mathcal{N}^{n_0}_j$ equal to the closest integer to $\alpha_j^{-1} n_0$ for each $1 \leq j \leq n$. Set 

\begin{eqnarray*}
\Omega := \bigcup_{n_0 \in  Bohr_{c(\vec{\alpha})N}(S(\vec{\alpha}), \rho)} [ An_0, An_0 + c_{\vec{\alpha}, \vec{\beta}}/A].
\end{eqnarray*}
Then the theorem in the irrational case will follow from the pointwise estimate

\begin{eqnarray*}
\left| T_{\hat{K}_1(\vec{\alpha} \cdot) \hat{K}_2(\vec{\beta} \cdot)[\vec{\alpha}, \vec{\beta}, \Phi, \Psi]} \left( \left\{ f_{m_j}^{N, A, \#_j ,\alpha_j } \right\}_{j=1}^n \right)(x)\right| \gtrsim_A \log(N) 1_{\Omega}(x)~\forall~x \in \mathbb{R}.
\end{eqnarray*}
To justify the claim, let us first calculate the main contribution. To this end, assume $t = \tilde{t} +Ak$, where $|\tilde{t}| \leq A/2$ and observe for $1 \leq k \leq \frac{N}{C_{\vec{\alpha}, \vec{\beta}}}$

\begin{eqnarray*}
&& \sum_{j=1}^n \#_j\alpha_j (-k + \mathcal{N}^{n_0}_j)(x-\alpha_j t - \beta_j s) \\ &=& \sum_{j=1}^n \#_j (n_0 - \alpha_j k + \delta(n_0, j)) (x- \alpha_j \tilde{t} - A \alpha_j k - \beta_j s) \\ &=& \sum_{j=1}^n \#_j  n_0 x + \sum_{j=1}^n A \#_j \delta(n_0, j) n_0 + \sum_{j=1}^n \#_j \delta(n_0, j) (\theta_x - \alpha_j \tilde{t} - A\alpha_j k - \beta_j s)+ \mathfrak{C} ks \\&=& D(x) -\sum_{j=1}^n \#_j \delta(n_0, j) \alpha_j \tilde{t} + \left[ \mathfrak{C} k - \sum_{j=1}^n \#_j \delta(n_0, j) \beta_j \right]s - A \sum_{j=1}^n \#_j \delta(n_0, j) \alpha_j k,
\end{eqnarray*}
where $D(x) :=  \sum_{j=1}^n \#_j n_0 x + \sum_{j=1}^n A \#_j \delta(n_0, j) n_0 + \sum_{j=1}^n \#_j \delta(n_0, j)\theta_x$. Moreover, 

\begin{eqnarray*}
A \sum_{j=1}^n \#_j \delta(n_0, j) \alpha_j k = A \sum_{j=1}^n \#_j\alpha_j (n_0 - \alpha_j \mathcal{N}^{n_0}_j) \in \mathbb{Z}. 
\end{eqnarray*}
Therefore, letting $m_j(n_0, k) = -k + \mathcal{N}^{n_0}_j$ and $\sum_{j=1}^n \#_j \alpha_j \beta_j = \mathcal{C}$, 
\begin{eqnarray*}
e^{ 2 \pi i A \sum_{j=1}^n \#_j \alpha_j m_j(n_0, k) (x-\alpha_j t - \beta_j s)} = e^{2 \pi iA  D(n_0, x)} e^{- 2\pi i A \sum_{j=1}^n \#_j \delta(n_0, j) \alpha_j  \tilde{t}} e^{2 \pi i A\left[ \mathfrak{C} k - \sum_{j=1}^n \#_j \delta(n_0, j) \beta_j \right]s  }.
\end{eqnarray*}
Using $|\delta(n_0, j)| \leq_{\vec{\alpha}} \frac{1}{A^2}$, we may deduce for every $|x- n_0| = |\theta_x| \lesssim_{A, \{\vec{\alpha}\}, \{\vec{\beta}\}} 1$ the lower bound

\begin{eqnarray*}
Im \left[ e^{-2 \pi i A D(x)} T^{k,0}_{ \hat{K}_1(\vec{\alpha} \cdot \vec{\xi}) \hat{K}_2(\vec{\beta} \cdot \vec{\xi})[\vec{\alpha}, \vec{\beta}, \Phi, \Psi]}\left(\left\{ f_{m_j(n_0,k)}^{N,A , \#_j , \alpha_j} \right\}_{j=1}^n \right) (x) \right] \gtrsim \frac{1}{Ak}. 
\end{eqnarray*}

\subsection{Small Perturbations: $1 \leq k \lesssim N$}
For each $n_0 \in [c_AN, N] \cap Bohr_N(S, \rho(\vec{\alpha}))$ and $k \in \left[ 1, d_{\vec{\alpha}, A}N \right]$, Letting $m_j (n_0, k,l) =-k + z_j l + \mathcal{N}^{n_0}_j$ and using $\sum_{j=1}^n \#_j \alpha_j^2 m_j =0$ for all $\vec{m} \in \mathbb{M}$ yields for all $\vec{m} = \vec{m}_j (n_0, k,l) + \Delta_j \in \mathbb{M}$

\begin{eqnarray*}
&& \sum_{j=1}^n \#_j \alpha_j  (m_j(n_0, k,l) +\Delta_j) ( x- \alpha_j t - \beta_j s) \\ &=& \sum_{j=1}^n \#_j \alpha_j (-k - z_j l+ \mathcal{N}^{n_0}_j) (x  - \beta_j \tilde{s}- A \beta_j l ) + \sum_{j=1}^n \#_j \alpha_j \Delta_j (x- \beta_j \tilde{s}- A \beta_j l ) \\ &=& \sum_{j=1}^n \#_j \alpha_j \left[  \mathcal{N}^{n_0}_j - z_j l + \Delta_j \right] x + \left[ \mathcal{C}k  - \sum_{j=1}^n \#_j \alpha_j (\Delta_j+\mathcal{N}^{n_0}_j  )\beta_j \right] \tilde{s}  + Z \\&=& I + II + Z, 
\end{eqnarray*}
where $Z \in \mathbb{Z}$. For term $I$, we may note $\sum_{j=1}^n \#_j \alpha_j z_j =\sum_{j=1}^n \#_j \beta_j =0$ as well as 

\begin{eqnarray*}
\sum_{j=1}^n \#_j \alpha_j \Delta_j x =   \sum_{j=1}^n \#_j \alpha_j \Delta_j (A (\alpha_j \mathcal{N}^{n_0}_j + \delta(n_0, j))+ \theta_x) = Z_1 +  \sum_{j=1}^n \#_j \alpha_j \delta_j (A  \delta(n_0, j) + \theta_x), 
\end{eqnarray*}
where $Z_1 \in \mathbb{Z}$. This remainder is acceptable using $A |\delta(n_0,j)| , |\theta_x| \lesssim \frac{1}{A}$. To handle term $II$, rewrite

\begin{eqnarray*}
\sum_{j=1}^n \#_j \alpha_j \mathcal{N}^{n_0}_j \beta_j \tilde{s} = \sum_{j=1}^n \#_j (n_0+ \delta(n_0, j)) \beta_j \tilde{s} = \sum_{j=1}^n \#_j \delta(n_0, j) \beta_j \tilde{s}.
\end{eqnarray*}
Therefore, $II = [\mathfrak{C} k + O(1/A^2)] \tilde{s}$, 
and applying Lemma \ref{ML} gives a satisfactory lower bound of $\frac{C}{Ak}$ for
\begin{eqnarray*}
 Im\left[ e^{- 2 \pi i A \sum_{j=1}^n \#_j \alpha_j \mathcal{N}^{n_0}_j x} \int_{A(k-\frac{1}{2})} ^{A(k+\frac{1}{2})} \int_{A(l-\frac{1}{2})}^{A(l+\frac{1}{2})} \prod_{j=1}^n \left[ \phi(x - Am_j - \alpha_j t - \beta_j s) e^{2 \pi i A\#_j m_j (x-\alpha_j t - \beta_j s)} \right] K_1(s) K_2(t) ds dt \right].
\end{eqnarray*}

\subsection{Large Perturbations}
\subsubsection{$k <0, l \in \mathbb{Z}, m \in \mathbb{M}$}
For each $(k,l) \in \mathbb{Z} \times \mathbb{Z},~ \exists j_* \in \{1, ..., n\}$ satisfying 

\begin{eqnarray*}
|\alpha_{j_*}m_{j_*} - n_0 - \alpha_{j_*} k - \beta_{j_*}l| \geq \frac{|l|}{5n \cdot \sup_{1 \leq j \leq n} |\#_j \alpha_j|}.
\end{eqnarray*}
Indeed, if not, then $|\alpha_{j_*}m_{j_*} - n_0 - \alpha_{j_*} k - \beta_{j_*}l| \leq \frac{|l|}{5n \cdot \sup_{1 \leq j \leq n} |\#_j \alpha_j|}$ for all $1 \leq j \leq n$, and so 

\begin{eqnarray*}
\sum_{j=1}^n \#_j \alpha_j^2 m_j &=& \sum_{j=1}^n \#_j \alpha_j (\alpha_jm_j - n_0 - \alpha_j k - \beta_j l) + \sum_{j=1}^n \#_j \alpha_j (n_0 - \alpha_j k - \beta_j l) \\ &=& I + II. 
\end{eqnarray*}
By assumption, $|I| \leq \frac{|l|}{5}$. Moreover, $II = \sum_{j=1}^n \#_j \alpha_j (n_0 + \alpha_j k + \beta_jl) \geq l,$ which contradicts the restriction $\sum_{j=1}^n \#_j \alpha_j^2 m_j =0$ for all $\vec{m} \in \mathbb{M}$. Similarly, if $k<0$, $\exists j_* \in \{1, ..., n\}$ such that

\begin{eqnarray*}
|\alpha_{j_*} m_{j_*} -n_0 - \alpha_{j_*} k - \beta_{j_*} l| \geq \frac{|k|}{5n \cdot \sup_{1 \leq j \leq n} |\#_j \beta_j|}. 
\end{eqnarray*}
The interested reader may easily check that the collection of large perturbations yields an acceptable error to the main contribution.

\end{proof}

\subsection{Lemma}

 \begin{lemma}\label{ML*}
 Fix $\phi \in \mathcal{S}(\mathbb{R})$, $A>0$ and let $K_1, K_2$ satisfy the usual conditions. Then there exists $k_0$ such that for all $k \geq k_0(\phi, \alpha, \beta,A, K_1, K_2)$, 
 
 \begin{eqnarray*}
 Im \left[ \int_{-A}^A \int_{-A}^A \phi(\alpha t + \beta s) e^{2 \pi i A sk} K_1(t+Ak) K_2(s) ds dt \right] \gtrsim \frac{D_{ \phi, \alpha, A, C_1, C_2}}{Ak},
 \end{eqnarray*}
where $D_{ \phi, \alpha, A, C_1, C_2}:=C_1\cdot  C_2 \cdot  \int_{-A}^A \phi(\alpha t) dt$. 
 \end{lemma}
 
 \begin{proof}
 The proof is a straightforward application of elementary decay estimates and integration by parts. First, note that is suffices to assume $C_1 = K_1(Ak)$ and $C_2= Im \left[ \check{K}_2(Ak) \right]$ and then prove
 
\begin{eqnarray*}
\lim_{k \rightarrow \infty}\left| \left|  Im\left[ \int_{-A^2k}^{A^2k} \phi\left(\alpha t + \frac{\beta s}{Ak} \right) e^{2 \pi i s} \frac{1}{Ak}K_2\left( \frac{s}{Ak} \right) ds \right] - Im \left[ \check{K}_2 (A k) \right] \phi(\alpha t) \right| \right|_{L^\infty_t\left( [-A, A]\right)}=0.
\end{eqnarray*} 
Assuming the claim, choose $k_0(A)$ large enough to ensure the relevant difference is uniformly bounded by $O(1/A^2)$. Then \begin{eqnarray*}
&&\left| Im \left[ \int_{-A}^A \int_{-A^2k}^{A^2k} \phi\left(\alpha t + \frac{\beta s}{Ak} \right) e^{2 \pi i s} \frac{1}{Ak} K_2\left( \frac{s}{Ak} \right) ds~ K_1(t+Ak) dt - \frac{D_A}{Ak} \right] \right| \\&\leq& \left| Im \left[ \int_{-A}^A\left( \int_{-A^2k}^{A^2k} \phi\left(\alpha t + \frac{\beta s}{Ak} \right) e^{2 \pi i s}\frac{1}{Ak}K_2\left( \frac{s}{Ak} \right) ds- C_2  \phi(\alpha t)  \right)K_1(t+Ak) dt \right] \right|  \\ &+& \left| Im \left[ C_2 \left( \int_{-A}^A \phi(\alpha t) \left( K_1(t+Ak) - K_1(Ak) \right) dt \right) \right] \right|.
\end{eqnarray*}
We have easy estimates for both $|~I~| \lesssim \frac{1}{A^2 k}$ and 
\begin{eqnarray*}
|~II~| \leq c \int_{-A}^A |\phi(\alpha t)|  \frac{|t|}{(Ak)^2}  dt \lesssim \frac{1}{(Ak)^2}.
\end{eqnarray*}
Hence, the lemma follows once we show the claim. To this end,  break up the interior integral into three pieces: 

\begin{eqnarray*}
Im\left[ \int_{-A^2k}^{A^2k}  \phi\left( \alpha t + \frac{\beta  s}{Ak} \right) e^{2 \pi i s} \frac{1}{Ak} K_2\left( \frac{s}{Ak} \right)ds\right]  &=&  Im \left[ \left( \int_{-A^2k}^{-1} + \int_{-1}^1 + \int_1^{A^2k} \right) ... \right] \\ &:=& I_a^{k,t} + I_b^{k,t} + I_c^{k,t}. 
\end{eqnarray*}
It is easy to see that $I_b^{k,t} =Im \left[  \int_{-1}^1  \phi\left( \alpha t + \frac{\beta s}{Ak} \right) e^{2 \pi i s}\frac{1}{Ak} K_2\left( \frac{s}{Ak} \right) ds\right] \rightarrow  \phi(\alpha t) Im \left[ \int_{-1}^1 e^{2 \pi i s}\frac{1}{Ak} K_2 \left( \frac{s}{Ak} \right) ds\right]$ as $k \rightarrow \infty$ uniformly in $t \in [-A,A]$ as desired. Moreover, a simple integration by parts argument on terms $I_a^{k,t}$ and $I_c^{k,t}$ yields

\begin{eqnarray*}
I_c^{k,t} &=& Im \left[ \frac{1}{2 \pi i } \int_1^{A^2k} \phi \left( \alpha t + \frac{ \beta  s}{Ak} \right) \frac{d}{ds} \left( e^{2 \pi i s} \right) \frac{1}{Ak} K_2 \left( \frac{s}{Ak} \right)ds \right] \\ &=&Im \left[  \frac{1}{2 \pi i } \left[ \frac{ \phi( \alpha t + \beta A) K_2(A)}{Ak} - \phi \left( \alpha t + \frac{\beta }{Ak} \right)\frac{K_2\left( \frac{1}{Ak} \right)}{Ak} \right.\right. \\  &-& \left.\left. \frac{\beta}{Ak} \int_1^{A^2k} \phi^\prime \left( \alpha t + \frac{\beta s}{Ak} \right) e^{2 \pi i s} \frac{1}{Ak} K_2\left( \frac{s}{Ak} \right) ds- \int_1^{A^2k} \phi\left( \alpha t + \frac{\beta s}{Ak} \right) e^{2 \pi i s} \frac{1}{(Ak)^2} K_2^\prime \left( \frac{s}{Ak} \right) ds\right] \right].
\end{eqnarray*}
However, 

\begin{eqnarray*}
\lim_{k \rightarrow \infty} \left| \left| \int_{A^2 k} ^{\infty} \phi\left( \alpha t + \frac{\beta s}{Ak} \right) e^{2 \pi i s} \frac{1}{(Ak)^2} K_2^\prime \left( \frac{s}{Ak} \right) \right| \right|_{L^\infty_t([-A, A])} =0
\end{eqnarray*}
from the pointwise estimate $\left| \frac{1}{(Ak)^2} K_2^\prime \left( \frac{s}{Ak} \right)\right| \lesssim \frac{1}{s^2}$. Moreover, using $\left|\phi(\alpha t + \frac{\beta s}{Ak} )- \phi(\alpha t)\right| \lesssim \left(\frac{\beta s}{Ak} \right)^{\frac{1}{2}}$ yields

\begin{eqnarray*}
\lim_{k \rightarrow \infty} \left| \left|  \int_1^{\infty} \left[ \phi\left( \alpha t + \frac{\beta s}{Ak} \right) - \phi\left(\alpha t \right) \right] e^{2 \pi i s} \frac{1}{(Ak)^2} K_2^\prime \left( \frac{s}{Ak} \right) ds  \right| \right|_{L^\infty_t([-A, A])}=0.
\end{eqnarray*}
Again using the uniform smoothness of $\phi$ and oddness of $K_2$, 

\begin{eqnarray*}
\lim_{k \rightarrow \infty} \left| \left|  I^{k,t}_c + Im\left[ \frac{1}{2 \pi i } \left[ \phi(\alpha t) \frac{K_2(\frac{1}{Ak})}{Ak} +  \phi(\alpha t)\int_1^\infty e^{2 \pi i s} \frac{1}{(Ak)^2} K_2^\prime \left( \frac{s}{Ak} \right) ds \right] \right] \right| \right|_{L^\infty_t ([-A, A])}  &=&0 \\ \lim_{k \rightarrow \infty} \left| \left|  I^{k,t}_a + Im\left[\frac{1}{2 \pi i } \left[  \phi(\alpha t) \frac{K_2(\frac{1}{Ak})}{Ak} + \phi(\alpha t) \int_{-\infty}^{-1} e^{2 \pi i s} \frac{1}{(Ak)^2} K_2^\prime \left( \frac{s}{Ak} \right) ds \right]\right] \right| \right|_{L^\infty_t ([-A, A])}  &=&0
 \end{eqnarray*}
Therefore, it is enough to show

\begin{eqnarray*}
 Im \left[ \frac{1}{2 \pi i } \left[ -\frac{ 2 K_2\left( \frac{1}{Ak} \right)}{Ak}   - \int_{\mathbb{R}\cap [-1,1]^c} e^{2 \pi i s} \frac{1}{(Ak)^2} K_2^\prime \left( \frac{s}{Ak} \right) ds \right] +\int_{-1}^1 e^{2 \pi i s} \frac{1}{Ak} K_2 \left( \frac{s}{Ak} \right)ds- \check{K_2}(Ak)\right ] =0.
\end{eqnarray*}
This is immediate via integrating by parts $\int_{\mathbb{R}-[-1,1]} e^{2 \pi i s} \frac{1}{(Ak)^2} K_2^\prime \left( \frac{s}{Ak} \right) ds$ over its two disjoint regions to rewrite the LHS as $\int_\mathbb{R} e^{2 \pi i s} \frac{1}{Ak} K_2 \left( \frac{s}{Ak} \right) ds - \check{K_2}(Ak)$ and then perform the change of variable $s \mapsto Ak s$.

 \end{proof}
 
The following statement is an immediate corollary of Lemma \ref{ML*}:
\begin{lemma}\label{ML7}
Fix $\phi \in \mathcal{S}(\mathbb{R})$, $A>0$, and $K_1, K_2$ satisfying the usual conditions. Then there exists $k_0$ such that for all $k \geq k_0(\phi, \vec{\alpha}, \vec{\beta}, A, K_1, K_2)$

\begin{eqnarray*}
Im \left[ \int_{-A}^A \int_{-A}^A \prod_{j=1}^n \phi(\alpha _j t + \beta_j s) e^{2 \pi i A sk} K_1(s) K_2(t + Ak)  ds dt \right] \gtrsim \frac{D_{\phi, \vec{\alpha}, A, C_1, C_2}}{Ak}
\end{eqnarray*}
where $D_{\phi, \vec{\alpha}, A, C_1, C_2}:= C_1 \cdot C_2 \left[ \int_{-A}^A \prod_{j=1}^n \phi(\alpha_j t)  dt\right]$.
\end{lemma}

\begin{proof}
Same as before. 
\end{proof}

\section{Symbols Adapted to Subspaces $\Gamma \subset \mathbb{R}^n$ with $dim~ \Gamma  \geq  \frac{ n}{2}+\frac{3}{2}$}

\begin{definition}
Fix $\Phi, \Psi \in \mathcal{S}(\mathbb{R})$ such that $1_{[-1/2, 1/2]} \leq \Phi \leq 1_{[-1,1]}$ and $1_{[2, \infty)} \leq \Psi \leq 1_{[1, \infty)}$. Let $\Gamma\left(\{\vec{\alpha}^m\}_{m=1}^d\right)= \bigcap_{m=1}^d \left\{ \vec{\xi} \cdot \vec{\alpha}^m =0 \right\}$ and $\Gamma\left( \{\vec{\beta}^m\}_{m=1}^d\right)= \bigcap_{m=1}^d \left\{ \vec{\xi} \cdot \vec{\beta}^m =0 \right\}$. For every symbol $m : \mathbb{R}^n \rightarrow \mathbb{C}$, the $(\Phi, \Psi)-$localization of $m$ near $\Gamma\left( \{\vec{\alpha}^m\}_{m=1}^d\right)$ away from $\Gamma\left(\{\beta^m\}_{m=1}^d\right)$ is the symbol defined by 

\begin{eqnarray*}
m[\{ \vec{\alpha}^m\}, \{\vec{\beta^m}\}, \Phi, \Psi](\vec{\xi}) =\left[ \prod_{m}  \Phi(\vec{\alpha}^m \cdot \vec{\xi} )\right] \left[ \prod_m^\prime  \Psi(\vec{\beta}^m\cdot \vec{\xi}) \right]m (\vec{\xi})~~\forall\vec{\xi} \in \mathbb{R}^n,
\end{eqnarray*}
where the primed product means one multiplies only over those $m \in \{1, ..., d\}$ for which $\vec{\beta}^m \not = \vec{0}$. 

\end{definition}
Our main result in this section is 
\begin{theorem}\label{TechThm}
Let $n \geq 5, d \geq 1$. Let $\left\{ \alpha_j ^m\right\}_{1 \leq j \leq n; 1 \leq m \leq d} ,\left\{ \beta_j ^m\right\}_{1 \leq j \leq n; 1 \leq m \leq d} \in \mathbb{R}^{nd} $ be given. Suppose there exists a vector $\vec{a} \in \mathbb{R}^n$ such that $\alpha_j^m = a_jq_j^m$ and $\beta_j^1= a_j r_j$ where $q_j^m, r_j \in \mathbb{Q}$ for all $1 \leq j \leq n$ and $1 \leq m \leq d$.  Furthermore, assume there are $\vec{\#} \in \mathbb{R}^n$ and $\mathfrak{C} >0$ such that

\begin{eqnarray*}
\sum_{j=1}^n \#_j \alpha_j^n =\sum_{j=1}^n \#_j \beta_j^n  &=& 0~~~~~~~~~~~~~\forall~ n \in \{1, ..., d\} \\ 
\sum_{j=1}^n \#_j \alpha_j^n \alpha_j^m =\sum_{j=1}^n \#_j \beta_j^n \beta_j^m &=& 0~~~~~~~~~~~~~\forall~n,m \in \{1, ..., d\}\\ 
\sum_{j=1}^n \#_j \alpha_j^n \beta_j^m &=&\mathfrak{C}(\vec{\#}) \cdot \delta_{n,1} \delta _{m,1} ~~~~\forall~n,m \in \{1, ..., d\},
\end{eqnarray*}
where $\delta: \{ 1, ..., d\} \times \{1, ..., d\} \rightarrow \{0,1\}$ is the Kronecker delta function, $\mathfrak{C}>0$, and $\#_j a_j^2 \in \mathbb{Q}$ for all $1 \leq j \leq n$. 
Moreover, let $K_d(s)=\frac{s_1}{|\vec{s}|^{d+1}}$ be the first $d-$dimensional Riesz kernel. Then every $(\Phi, \Psi)$-localization of $\hat{K}_d(A\vec{\xi}) \hat{K}_d(B \vec{\xi}): \mathbb{R}^n \rightarrow \mathbb{C}$ gives rise to a multilinear multiplier 

\begin{eqnarray*}
T_{\hat{K}_d(A \vec{\xi}) \hat{K}_d(B \vec{\xi})[\{\vec{\alpha}^m\}, \{\vec{\beta}^m\}, \Phi, \Psi]}
\end{eqnarray*}
which satisfies no $L^p$ estimates.

\end{theorem}
Remark: The notation $q_j^m$ does not mean $(q_j)^m$. We should consider $q_j^m$ as an doubly-indexed quantity.  In the following lemma, $q_j^m = (q_j)^m$!



\begin{lemma}\label{IRL3}
Let $n \geq 5$, $d \geq 1$, and $2d+3 \leq n$. Suppose $\vec{a} \in \mathbb{R}^n$ satisfies $a_j ^{-1}= q_j + \alpha q_j^{2}$ and $\beta_j = a_j q_j^{-1}$ for some $\alpha \in \mathbb{R}$ and $\vec{q} \in \mathbb{Q}^n$ with distinct, non-zero entries such that $q_j \alpha \not =-1$ for all $1 \leq j \leq n$.  Let $\alpha_j^n=a_j q_j^{n}$ for $2 \leq n \leq d$ and $\alpha_j^1 = a_j$. Then there are $\vec{\#} \in \mathbb{R}^n$ and $\mathfrak{C}>0$ such that 

\begin{eqnarray*}
\sum_{j=1}^n \#_j \alpha_j^n=  \sum_{j=1}^n \#_j \alpha_j^n \alpha_j^m =
\sum_{j=1}^n \#_j \beta_j^1 = 
\sum_{j=1}^n \#_j \left[ \beta_j^1 \right]^2 &=& 0 \\ 
\sum_{j=1}^n  \#_j \alpha_j^m \beta_j^1 &=& \mathfrak{C} \delta_{1,m}
\end{eqnarray*}
with the additional property that $\#_j a_j^2 \in \mathbb{Q}$ for all $1 \leq j \leq d$. 
\end{lemma}
\begin{proof}
It suffices to prove there is a rational solution $\vec{\tilde{\#}} \in \mathbb{Q}^n$ to the system 
 
 \begin{eqnarray*}
\sum_{j=1}^n \#_j \alpha^1_j=\sum_{j=1}^n\tilde{\#}_j  a_j^{-2} a_j= \sum_{j=1}^n \tilde{\#}_j ( q_j + \alpha q_j^2)&=& 0 \\
 \sum_{j=1}^n \#_j \alpha_j^m = \sum_{j=1}^n\#_j a_j q_j^m =\sum_{j=1}^n  \tilde{\#}_j a_j^{-1} q^{m}_j=\sum_{j=1}^n \tilde{\#}_j (q_j + \alpha q_j^2) q_j^m&=&0 ~~~~ ( 2\leq m \leq d)\\ 
 \sum_{j=1}^n \#_j \beta^1_j  = \sum_{j=1}^n \tilde{\#}_j a_j^{-2} a_j q_j^{-1} = \sum_{j=1}^n \tilde{\#}_j ( q_j + \alpha q_j^2) q_j^{-1} &=& 0 \\ 
\sum_{j=1}^n \#_j \left[ \alpha_j^1\right]^2= \sum_{j=1}^n\#_j a_j^2 = \sum_{j=1}^n \tilde{\#}_j &=& 0 \\ 
\sum_{j=1}^n \#_j \alpha_j^m \alpha_j^n= \sum_{j=1}^n \#_j a_j^2 q_j^m q_j^n  = \sum_{j=1}^n \tilde{\#}_j q_j^{m+n}&=&0~~~~(2 \leq n,m \leq d)\\
 \sum_{j=1}^n \#_j \alpha_j^1 \alpha_j^m = \sum_{j=1}^n \#_j a_j^2 q_j^m = \sum_{j=1}^n \tilde{\#}_j q_j^m &=& 0 ~~~~~~(2 \leq m \leq d)\\ 
\sum_{j=1}^n \#_j \left[ \beta_j^1\right]^2= \sum_{j=1}^n \tilde{\#}_j q_j^{-2} = 0\\
\\
 \sum_{j=1}^n \#_j \beta_j^1 \alpha_j^m = \sum_{j=1}^n \#_j a_j^2 q_j^m q_j^{-1} = \sum_{j=1}^n \tilde{\#}_j q_j^{m-1} &=& 0 ~~~~~~(2 \leq m \leq d) \\ \sum_{j=1}^n \#_j \alpha_j^1 \beta_j^1=  \sum_{j=1}^n \tilde{\#}_j q_j^{-1} \not = 0.
 \end{eqnarray*}
Because $q_i \not = q_j$ whenever $i \not = j$ and $2d+3 \leq n$, one can choose non-trivial $\tilde{\vec{\#}} \in \mathbb{Q}^n$ to ensure $\sum_{j=1}^n \tilde{\#}_j q_j^m =0$ for all $m \in \{-2, 0, 1, 2, ..., 2d\}$ and $\sum_{j=1}^n \tilde{\#}_j q_j^{-1} \not = 0$. 
 \end{proof}

Using Theorem \ref{TechThm} and Lemma \ref{IRL3}, we obtain
\begin{theorem}\label{ExThm}
Let  $n, \mathfrak{d} \in \mathbb{N}$ satisfy $\frac{n+3}{2} \leq \mathfrak{d} <n$ and $n \geq 5$.  Furthermore, let $a_j^{-1} = q_j + \alpha q_j^2$ for some $\alpha \in \mathbb{R}$ and $\vec{q} \in \mathbb{Q}^n$ with distinct, non-zero entries such that $q_j \alpha \not = -1$ for all $1 \leq j \leq n$. Let $\alpha_j^n = a_j q_j^n$ for $2 \leq n \leq d$ and $\alpha_j^1 = a_j$. Moreover, let 

\begin{eqnarray*}
\Gamma = \bigcap _{m=1}^d \left\{ \vec{\xi} \cdot \vec{\alpha}^m =0 \right\} \subset \mathbb{R}^n.
\end{eqnarray*}
Then there exists a symbol $m_\Gamma$ adapted to $\Gamma$ in the Mikhlin-H\"{o}rmander sense and supported in $\left\{ \vec{\xi} : dist(\vec{\xi}, \Gamma) \lesssim 1 \right\}$ for which the associated multilinear multiplier $T_{m_\Gamma}$ is unbounded. 
\end{theorem}
\begin{prop}\label{prop1}
Let $n \geq d+1$. Fix $\vec{q} \in \mathbb{Q}^n$ and $\alpha \in \mathbb{R}$. Let $a_j^{-1} = q_j + \alpha q_j^2$ with distinct, non-zero entries such that $q_j \alpha \not = -1$ for all $1 \leq j \leq n$. Let $\alpha_j^n = a_j q_j^n$ for $2 \leq n \leq d$ and $\alpha_j^1 = a_j$. Let

\begin{eqnarray*}
\Gamma(\alpha, \vec{q}) := \bigcap _{m=1}^d \left\{\left. \vec{\xi} \in \mathbb{R}^n ~~\right|~ \vec{\xi} \cdot \vec{a}^m(\alpha, \vec{q} )=0 \right\}.
\end{eqnarray*}
Then the number of distinct subspaces in the collection $\left\{ \Gamma(\alpha, \vec{q}) \right\}_{|\alpha| \leq \epsilon}$ is uncountable for every $\epsilon >0$. 
\end{prop}

\begin{proof}
Because $\Gamma(0,\vec{q}) = \left[Span\left\{\vec{q}^{-1}, \vec{q}, \vec{q} \wedge \vec{q}, ..., \wedge^{d-1} \vec{q} \right\}\right]_{\perp}$, 

\begin{eqnarray*}
\Gamma(0,\vec{q}) _{\perp}= \left[ \left[Span\left\{\vec{q}^{-1}, \vec{q}, \vec{q} \wedge \vec{q}, ..., \wedge^{d-1} \vec{q} \right\} \right]_{\perp} \right]_{\perp} = Span\left\{\vec{q}^{-1}, \vec{q}, \vec{q} \wedge \vec{q}, ..., \wedge^{d-1} \vec{q} \right\} .
\end{eqnarray*}
As $n \geq d+1$ and $\vec{q} \in \mathbb{Q}^n$ has distinct non-zero entries, $\vec{1} \not \in Span\left\{\vec{q}^{-1}, \vec{q}, \vec{q} \wedge \vec{q}, ..., \wedge^{d-1} \vec{q} \right\}$ so that $d := dist(\vec{1}, \Gamma(0,\vec{q})_{\perp}) >0$.  Therefore, $\frac{d}{d \alpha} \vec{a}(\alpha)=-\frac{ q_j^2}{(q_j + \alpha q_j^2)^2}$. Therefore, $\left. \frac{d}{d \alpha}\vec{a}(\alpha) \right|_{\alpha = 0} \propto \vec{1}$. It follows that for small enough $\epsilon$ depending on $\vec{q}$, $\vec{a}(\epsilon) \not \in \Gamma(0,\vec{q})_{\perp}$. Therefore, letting

\begin{eqnarray*}
D_\epsilon(\alpha) :=dist(\vec{a}(\epsilon), \Gamma(\alpha, \vec{q})_{\perp}),
\end{eqnarray*}
 $D_\epsilon(0) \not = 0, D_\epsilon(\epsilon) =0$. Therefore, $Im \left[ D_\epsilon([0, \epsilon)]) \right]\supset [0, \delta)$ for some $\delta>0$. In particular, $D_\epsilon$ takes on uncountably many values, and so $\{\Gamma(\alpha, \vec{q})_{\perp} \}_{|\alpha| \leq  \epsilon}$ and $\left\{\Gamma(\alpha, \vec{q})\right\}_{|\alpha| \leq \epsilon}$ must be uncountable.

\end{proof}
\begin{prop} \label{prop2}
Let $n \geq d+1$. Fix $\vec{q} \in \mathbb{Q}^n$ and $\alpha \in \mathbb{R}$. Let $a_j^{-1} = q_j + \alpha q_j^2$ with distinct, non-zero entries such that $q_j \alpha \not = -1$ for all $1 \leq j \leq n$. Let $\alpha_j^n = a_j q_j^n$ for $2 \leq n \leq d$ and $\alpha_j^1 = a_j$. Let

\begin{eqnarray*}
\Gamma(\alpha, \vec{q}) := \bigcap _{m=1}^d \left\{\left. \vec{\xi} \in \mathbb{R}^n ~~\right|~ \vec{\xi} \cdot \vec{a}^m(\alpha, \vec{q} )=0 \right\}.
\end{eqnarray*}
Then, for each $\vec{q} \in \mathbb{Q}^n$ with distinct non-zero entries,  there exists $\epsilon>0$ such that $\Gamma(\alpha, \vec{q})$ is non-degenerate for all $|\alpha| \leq \epsilon$ in the sense of Muscalu, Tao, and Thiele \cite{MR1887641}, i.e. 
\begin{eqnarray*}
\tilde{\Gamma} := \left\{ (\xi_1, ,..., \xi_{n+1})\in \mathbb{R}^{n+1} : (\xi_1, ..., \xi_n) \in \Gamma, \sum_{j=1}^{n+1} \xi_{j}=0 \right\} 
\end{eqnarray*}
 is a graph over the variables $\xi_{i_1}, ..., \xi_{i_{n-d}}$ for every $1 \leq i_1 < i_2 < ... < i_{n-d} \leq n+1$. 

\end{prop}
\begin{proof}
First observe that for $\vec{v} \in \mathbb{Z}^n$, set  $\Gamma ^{\vec{v}} := \left\{ \vec{\xi} \in \mathbb{R}^n: \sum_{j=1}^n \xi_j v_j =0 \right\}$. Let $\vec{\alpha}^j = \wedge^{j} \vec{\gamma}$ for $j \in \{1, ..., d\}$ where $\vec{\gamma} \in \mathbb{Z}^n$ satisfies $\gamma_i \not = \gamma_j$ for all $i \not = j$. Then the subspace $\Gamma := \bigcap_{j=1}^d  \Gamma^{\vec{\alpha}^j} \subset \mathbb{R}^n $ is a non-degenerate subspace. Indeed, let $\mathcal{M} \in M_{d+1, n+1}(\mathbb{Z})$ be given by
\[\mathcal{M}=
  \begin{bmatrix}
    1 & 1& .... & 1 &1\\ 
    \gamma^1_1 & \gamma_2^1 & ... & \gamma_n^{1}&0 \\
     \gamma^2_1  & \gamma_2^2 & ... & \gamma_n^{2} &0 \\ 
    \vdots & \vdots  & \hdots & \vdots &\vdots\\ 
    \gamma^d_1& \gamma_2^d& ... & \gamma_n^{d}&0
  \end{bmatrix}.
\]
Then $\vec{\xi} \in \tilde{\Gamma} \subset \mathbb{R}^{n+1}$ iff $\mathcal{M} \vec{\xi} = \vec{0}$. Suppose $n-d$ distinct indices $\mathcal{I} \subset (i_1, i_2, ..., i_{n-d}) \subset (1, 2, ..., n+1)$ have been chosen. Form $\mathcal{M}_{\mathcal{I}^c}\in M_{d+1, d+1}$ by deleting those columns with indices in $\mathcal{I}$. Furthermore, let $\vec{\xi}_{\mathcal{I}^c} \in \mathbb{R}^{n-d}$ be given by deleting all indices in $\mathcal{I}$. Then for every $\vec{\xi} \in \mathbb{R}^{n+1}$, $\mathcal{M} \vec{\xi} = \mathcal{M}_{\mathcal{I}} \vec{\xi}_{\mathcal{I}} + \mathcal{M}_{\mathcal{I}^c} \vec{\xi}_{\mathcal{I}^c}$ and for every $\xi \in \tilde{\Gamma}$, 

\begin{eqnarray*}
\mathcal{M}_{\mathcal{I}} \vec{\xi}_{\mathcal{I}} = -  \mathcal{M}_{\mathcal{I}^c} \vec{\xi}_{\mathcal{I}^c}.
\end{eqnarray*}
However, $\mathcal{M}_{\mathcal{I}^c}$ must take one of the following forms: 

\[  \begin{bmatrix}
1&1 & ... &1 \\
    \gamma^1_{j_1} & \gamma^1_{j_2}& ... & \gamma^1_{j_{d+1}} \\
     \gamma^2_{j_1}  & \gamma^2_{j_2} & ... & \gamma_{j_{d+1}}^{2}  \\ 
    \vdots & \vdots  & \hdots & \vdots \\ 
    \gamma^d_{j_1}& \gamma_{j_2}^d& ... & \gamma_{j_{d+1}}^{d}
\end{bmatrix} ~~~or~~~~
  \begin{bmatrix}
    1 & 1& .... & 1 &1\\ 
    \gamma^1_{j_1} & \gamma^1_{j_2}& ... & \gamma^1_{j_{d}}&0 \\
     \gamma^2_{j_1}  & \gamma^2_{j_2} & ... & \gamma_{j_{d}}^{2} &0 \\ 
    \vdots & \vdots  & \hdots & \vdots &\vdots\\ 
    \gamma^d_{j_1}& \gamma_{j_2}^d& ... & \gamma_{j_{d}}^{d}&0
\end{bmatrix}
\]
for some $\vec{j} =(j_1, ..., j_d): 1 \leq j_1 \leq ... \leq j_d$. In either case, $\mathcal{M}_{\mathcal{I}^c}$ is invertible by the assumption $\gamma_i \not = \gamma_j$ whenever $i \not = j$. Hence, there is a well-defined mapping $\gamma : (\xi_{i_1},..., \xi_{i_{n-d}}) \rightarrow (\xi_1,..., \xi_{n+1})$ expressing $\tilde{\Gamma}$ as a graph over the variables $(\xi_{i_1}, ..., \xi_{i_{n-d}})$. In the limit as $\alpha \rightarrow 0$, $\vec{\alpha} \rightarrow q_j^{-1}$. Therefore, $\vec{\alpha}^1 \rightarrow \vec{q}^{-1}$ and $\vec{\alpha}^n \rightarrow \vec{q}^{n-1}$ for $2 \leq n \leq d$. The matrix 
 
 \[\mathcal{M}(\vec{q}) := 
  \begin{bmatrix}
    1 & 1& .... & 1 &1\\ 
q^{-1}_1 & q^{-1}_2& ... & q^{-1}_{d}&0 \\
    q_1  & q_2& ... & q_d &0 \\ 
    \vdots & \vdots  & \hdots & \vdots &\vdots\\ 
    q_1^{d-1}& q_2^{d-1}& ... &q_d^{d-1}&0
\end{bmatrix},
\]
has non-zero determinant for all its minors, which ensures $\Gamma(0, \vec{q})$ is non-degenerate. Moreover, as non-degeneracy holds at $\alpha=0$, it continuous to hold for all $|\alpha| \leq \epsilon(\vec{q})$.

\end{proof}

Combining Theorem \ref{ExThm} with Propositions \ref{prop1} and \ref{prop2} finally yields the takeaway result:
\begin{reptheorem}{MT}
Let  $n, \mathfrak{d} \in \mathbb{N}$ satisfy $\frac{n+3}{2} \leq \mathfrak{d} <n$ and $n \geq 5$.  Then there is an uncountable collection $\mathfrak{C}$ of $\mathfrak{d}-$ dimensional non-degenerate subspaces $\Gamma \subset \mathbb{R}^n$ such that for each $\Gamma \in \mathfrak{C}$ there is an associated symbol $m_{\Gamma}$ adapted to $\Gamma$ in the Mikhlin-H\"{o}rmander sense for which the associated multilinear multiplier $T_{m_{\Gamma}}$ is unbounded. 
\end{reptheorem}

\begin{proof} \textit{[Theorem \ref{TechThm}]}
\subsection{PART 1: The Rational Case}
Note that our assumption in the case $\alpha=0$ is equivalent to the superficially weaker assumption that $A= \left( \vec{\alpha}_1, ..., \vec{\alpha}_n \right)  \in (\mathbb{Q}^d)^n$ and $B=\left( \vec{\beta}_1, ..., \vec{\beta}_n\right) \in (\mathbb{Q}^d)^n$ satisfy

\begin{eqnarray*}
 (\vec{\alpha}^1 \wedge \vec{\beta}^1) \not \in Span \left\{ \vec{\alpha}^1, \vec{\beta}^1, \left. (\vec{\alpha}^n\wedge\vec{\beta}^m) \right|_{(n,m) :\delta_{n,1} \delta_{m,1} =0}, \left. (\vec{\alpha}^n \wedge \vec{\alpha}^m)\right|_{(n,m)}, \left. (\vec{\beta}^n\wedge \vec{\beta}^m)\right|_{(n,m)}\right\}. 
\end{eqnarray*}
Alternatively, $A=\left(\vec{\alpha}^1, ..., \vec{\alpha}^d \right), B=\left( \vec{\beta}^1, ..., \vec{\beta}^d \right) \in( \mathbb{Q}^n)^d$ under the obvious identification $\left( \mathbb{Q}^d \right)^n \simeq\left( \mathbb{Q}^n\right)^d \simeq \mathbb{Q}^{dn}$. By dilating if necessary, we shall assume $A, B \in \mathbb{Z}^{nd}$. 
Indeed, by the dimensionality constraints and spanning condition, we are assured by the Gram-Schmidt process of finding of vector $\vec{\#} \in \mathbb{R}^n$ for which the orthogonality constraints (*) are satisfied: 

\begin{eqnarray*}
\sum_{j=1}^n \#_j \alpha_j^n =\sum_{j=1}^n \#_j \beta_j^n  &=& 0~~~~~~~~~~~~~\forall~ n \in \{1, ..., d\} \\ 
\sum_{j=1}^n \#_j \alpha_j^n \alpha_j^m =\sum_{j=1}^n \#_j \beta_j^n \beta_j^m &=& 0~~~~~~~~~~~~~\forall~n,m \in \{1, ..., d\}\\ 
\sum_{j=1}^n \#_j \alpha_j^n \beta_j^m &=&\mathfrak{C}(\vec{\#}) \cdot \delta_{n,1} \delta _{m,1} ~~~~\forall~n,m \in \{1, ..., d\},
\end{eqnarray*}
where $\delta: \{ 1, ..., d\} \times \{1, ..., d\} \rightarrow \{0,1\}$ is the usual Kronecker delta function and $\mathfrak{C} \in \mathbb{R}^c\cap \{0\}^c$.  In fact, one can always restrict $\vec{\#} \in \mathbb{Z}^n$. Indeed, because $A , B \in \mathbb{Z}^{dn}$, we may perform the Gramm-Schmidt process to form an orthogonal basis $\left\{ \vec{\gamma}^1, ..., \vec{\gamma}^p\right\}$ for $\mathcal{S}:=Span \left\{\vec{1}, \vec{\alpha}^1, \vec{\beta}^1, \left. (\vec{\alpha}^n\wedge\vec{\beta}^m) \right|_{(n,m) :\delta_{n,1} \delta_{m,1} =0}, \left. (\vec{\alpha}^n \wedge \vec{\alpha}^m)\right|_{(n,m)}, \left. (\vec{\beta}^n\wedge \vec{\beta}^m)\right|_{(n,m)}\right\}$ such that $\vec{\gamma}^j \in \mathbb{Q}^n$ for every $j \in \{1, ..., p\}$.  Moreover, $\vec{\alpha}^1 \wedge \vec{\beta}^1$ is not in the span, so we can find an element in $\mathbb{Q}^n$ orthogonal to $\mathcal{S}$ by setting 

\begin{eqnarray*}
\vec{\#} = \vec{\alpha}^1 \wedge \vec{\beta}^1 - \sum_{j=1}^p \frac{ \vec{\gamma}^j \langle \vec{\gamma}^j , \vec{\alpha}^1 \wedge \vec{\beta}^1 \rangle}{\langle \vec{\gamma}^j, \vec{\gamma}^j \rangle}
\end{eqnarray*}
for which the constraints $(*)$ are satisfied. By a signed dilation, we may take $\vec{\#} \in \mathbb{Z}^n$ so that $\mathcal{C} \in \mathbb{Z}^+$ without loss of generality.

Recall $f^{N, A, \#}(x):= \sum_{-N \leq m \leq N} \phi(x-Am) e^{2 \pi i A m \#  x}= \sum_{-N \leq m  \leq N} f_m^{N, A \#} (x)$, where $\phi$ satisfies the same properties as the $\phi$ appearing in Theorem \ref{MT*}.  Fix $\vec{m}$ and assume $m_j = n_0 - \vec{\alpha_j }\cdot \vec{k} - \vec{\beta_j} \cdot \vec{l}$ for some $\vec{k}$ and $\vec{l}$ for all $j \in \{1, ..., n\}$. Then for appropriate choices for $\phi$ and $A$, the frequency support assumptions ensure

\begin{eqnarray*}
\sum_{j=1}^n \#_j \alpha^k_j m_j =0~~\forall 1 \leq k \leq d; ~~\sum_{j=1}^n\#_j  \beta^1_j m_j >0
\end{eqnarray*}
provided 

\begin{eqnarray*}
T_{\hat{K}_d(A\cdot) \hat{K}_d(B\cdot)[\{ \vec{\alpha}^m\}, \{\vec{\beta}^m\}, \Phi, \Psi]}\left(\left\{  f_{m_j}^{N, A \#_j} \right\}_{j=1}^n\right) \not \equiv 0.
\end{eqnarray*}
Indeed, by construction, 

\begin{eqnarray*}
 \hat{K}_d(A\cdot  ) \hat{K}_d(B\cdot  )[\{ \vec{\alpha}^m\}, \{\vec{\beta}^m\}, \Phi, \Psi](\vec{\xi}) = \hat{K}_d(A  \vec{\xi}) \hat{K}_d(B \vec{\xi}) m_1\left(\vec{\xi}\right) m_2\left(\vec{\xi}\right),
 \end{eqnarray*}
 where the symbols $m_1$ and $m_2$ are Mikhlin-H\"{o}rmander adapted to $\bigcap_{j=1}^d \{\vec{\xi} : dist(\vec{\xi}, \Gamma^{\vec{\alpha}^{j}}) \lesssim 1\}$ and $\{\vec{\xi}: \xi \cdot \vec{\beta}^1\gtrsim 1\}$ and are also identically equal to $1$ on $\bigcap_{j=1}^d \{\vec{\xi} : dist(\vec{\xi}, \Gamma^{\vec{\alpha}^{j}}) \lesssim 1\}$ and $\left\{\vec{\xi}: \vec{\xi} \cdot \vec{\beta}^1 \gtrsim 1\right\}$ respectively.  Therefore, setting
\begin{eqnarray*}
\mathbb{M}  := \left\{ \vec{m} \in \mathbb{Z}^n\cap [-N,N]^n \left|  \sum_{j=1}^n \#_j m_j \beta^1_j >0~;~\sum_{j=1}^n \#_j m_j \alpha^k_j =0~\forall~1 \leq k \leq d\right\} \right. , 
\end{eqnarray*}

\begin{eqnarray*}
 T_{\hat{K}_d(A\cdot) \hat{K}_d(B\cdot)[\{ \vec{\alpha}^m\}, \{\vec{\beta}^m\}, \Phi, \Psi]}\left(\left\{  f^{N, A, \#_j} \right\}_{j=1}^n\right)=  \sum_{\vec{m} \in \mathbb{M}} T_{\hat{K}_d(A\cdot) \hat{K}_d(B\cdot)[\{ \vec{\alpha}^m\}, \{\vec{\beta}^m\}, \Phi, \Psi]}\left(\left\{  f_{m_j}^{N, A, \#_j} \right\}_{j=1}^n\right). 
\end{eqnarray*}
Furthermore, note 

\begin{eqnarray*}
&&T_{\hat{K}_d(A\cdot) \hat{K}_d(B\cdot)[\{ \vec{\alpha}^m\}, \{\vec{\beta}^m\}, \Phi, \Psi]}\left( \left\{ f^{N, A, \#_j} \right\}_{j=1}^n \right)(x)\\ &=& \sum_{\vec{m}\in  \mathbb{M}} \int_{\mathbb{R}^{2d}} \left[ \prod_{j=1}^n \phi(x-Am_j -\vec{\alpha}_j \cdot \vec{t} - \vec{\beta}_j \cdot \vec{s})e^{2 \pi i A \#_j m_j (x-\vec{\alpha}_j \cdot \vec{t} - \vec{\beta}_j \cdot \vec{s})} \right]\frac{s_1}{|\vec{s}|^{d+1}} \frac{t_1}{|\vec{t}|^{d+1}} d \vec{t} d \vec{s} \\ &=& \sum_{\vec{m} \in  \mathbb{M}} \sum_{\vec{l}, \vec{k} \in \mathbb{Z}^d}\int_{A(l_1-\frac{1}{2})}^{A(l_1+\frac{1}{2})}  ... \int_{A(k_d -\frac{1}{2})}^{A(k_d+\frac{1}{2})} \left[ \prod_{j=1}^n \phi(x-Am_j -\vec{\alpha}_j \cdot \vec{t} - \vec{\beta}_j \cdot \vec{s})e^{2 \pi i A\#_j m_j (x-\vec{\alpha}_j \cdot \vec{t} - \vec{\beta}_j \cdot \vec{s})} \right]\frac{s_1}{|\vec{s}|^{d+1}} \frac{t_1}{|\vec{t}|^{d+1}} d \vec{t} d \vec{s} .
\end{eqnarray*}
Large terms arise whenever $x \in \left[An_0 , An_0 +\frac{ c_{\vec{\alpha }, \vec{\beta}}}{A}\right]$ and there exists $\vec{m} \in [-N, N]^n \cap \mathbb{Z}^n$ together with $\vec{k}\in \mathbb{Z}^d$ satisfying the constraints $k_1 \gtrsim1$, $|\vec{k}| \lesssim N$  and

\begin{eqnarray*}
n_0 - m_j(n_0, \vec{k}, \vec{0}) - \vec{\alpha}_j \cdot \vec{k} =0~\forall~j \in \{1, ..., n\}. 
\end{eqnarray*}
In fact, we shall prove that it is enough to produce a lower bound for these large terms because the remainder may be subsumed as error:

\begin{lemma}\label{RdL3}
To prove theorem \ref{MT}, it suffices to show $\exists  c_{\left\{\vec{\alpha}^j, \vec{\beta}^j \right\}_{j \in \{1, ..., d\}}}$ with the property that  for every $x \in \left[ An_0, An_0 + \frac{c_{\left\{\vec{\alpha}^j , \vec{\beta}^j \right\}_{j \in \{1, ..., d\}}}}{A} \right]$, $k \in \left[ k_0, \frac{N}{3d \max_{(j_1, j_2) \in [1, ..., n]\times [1,...., d]} \left\{ |\alpha_{j_1}^{j_2}| \right\}} \right]$, and $n_0 \in [-N/3, N/3]$

\begin{eqnarray*}\label{Rd3}
Im \left[ e^{-2 \pi i A \left(\sum_{j=1}^n \#_j\right) n_0 x} \sum_{\vec{z} \in \mathbb{Z}^d : \vec{m}(n_0, \vec{z}, \vec{0}) \in \mathbb{M}, z_1 = k_1} T^{\vec{k}, \vec{0}}_{\hat{K}_d(A\cdot) \hat{K}_d(B\cdot)[\{ \vec{\alpha}^m\}, \{\vec{\beta}^m\}, \Phi, \Psi]} \left(\left\{ f_{m_j(n_0, \vec{k}, \vec{0})}^{N,A, \#_j} \right\}_{j=1}^n \right)(x) \right]\gtrsim \frac{1}{A^d k_1}.
\end{eqnarray*}

\end{lemma}

\subsection{Main Contribution: $\vec{k}:k_1 \gtrsim 1, \vec{l}=\vec{0}$}
Before proving the lemma, we verify the lower bound in the above display.  Note that if $n_0 \in [-N/3, N/3]$ and $\vec{k} \in \mathbb{Z}^d \cap \left[ k_0, \frac{N}{3d \max_{(j_1, j_2) \in [1, ..., d]\times [1,...., n]} \left\{ |\alpha_{j_1}^{j_2}| \right\}} \right]^d$, then $m_j(\vec{k}, \vec{0})  :=n_0 - \vec{\alpha}_j \cdot k ~\forall~j \in \{1, ..., n\} \in \mathbb{Z}^n \cap [-N, N]^n$. For $x \in \left[ An_0, An_0 + \frac{c_{\left\{\vec{\alpha}^j, \vec{\beta}^j \right\}_{j \in \{1, ..., d\}}}}{A}  \right]$, let $\theta_x= \lfloor x \rfloor.$ Evaluating the diagonal contribution yields 
\begin{eqnarray*}
& & e^{-2 \pi i A\left( \sum_{j=1}^n \#_j \right) n_0 x} T_{\hat{K}_d(A\cdot) \hat{K}_d(B\cdot)[\{ \vec{\alpha}^m\}, \{\vec{\beta}^m\}, \Phi, \Psi]} ^{\vec{k}, \vec{0}} \left(\left\{  f^{N, A, \#_j}_{m_j(n_0, k)}\right\}_{j=1}^n   \right)(x)\\ &= & \int_{[-\frac{A}{2}, \frac{A}{2}]^d}  \int_{A(k_1 - \frac{1}{2})}^{A(k_1+\frac{1}{2})} ... \int_{A(k_d-\frac{1}{2})} ^{A(k_d+\frac{1}{2})} e^{2 \pi i  A\left( \sum_{j=1}^d \#_j \alpha^1_j \beta^1_j \right) s_1 k_1} \frac{s_1 t_1 }{|\vec{s}|^{d+1} |\vec{t}|^{d+1}} d\vec{t} d\vec{s}\\
&=&\int_{-\frac{A^2\mathfrak{C} k_1}{2} }^{\frac{ A^2\mathfrak{C}k_1}{2}} ... \int_{-\frac{ A^2\mathfrak{C}k_1}{2}}^{\frac{ A^2\mathfrak{C}k_1}{2} } \int_{-\frac{A}{2}}^{\frac{A}{2}} ... \int_{-\frac{A}{2}} ^{\frac{A}{2}} \prod_{j=1}^n \phi\left( \theta_x+ \vec{\alpha}_j \cdot \vec{t} + \vec{\beta}_j \cdot \frac{ \vec{s}}{A \mathfrak{C} k_1}\right )e^{2 \pi i s_1} \frac{s_1 (t_1+Ak_1)  }{|\vec{s}|^{d+1} |\vec{t}+ A\vec{k}|^{d+1}} d\vec{t} d\vec{s} .
\end{eqnarray*}
To show the desired inequality, it suffices to prove 
\begin{eqnarray*}
&& \lim_{k_1 \rightarrow \infty} \int_{-A^2\mathfrak{C} k_1/2}^{ A^2\mathfrak{C} k_1 /2}... \int_{-A^2\mathfrak{C} k_1/2}^{A^2 \mathfrak{C}k_1/2} \prod_{j=1}^n \phi\left( \theta_x + \vec{\alpha}_j \cdot \vec{t} + \vec{\beta}_j \cdot \frac{\vec{s}}{A \mathfrak{C}k_1}\right) e^{2 \pi i s_1} \frac{s_1}{|\vec{s}|^{d+1}} d \vec{s} \\&=&  \left[ \prod_{j=1}^n \phi\left( \theta_x + \vec{\alpha}_j \cdot \vec{t} \right)\right] \int_\mathbb{R} ... \int_\mathbb{R} e^{2 \pi i s_1} \frac{s_1}{|\vec{s}|^{d+1}} d\vec{s}
\end{eqnarray*}
uniformly in $t \in [-\frac{A}{2},\frac{A}{2}]^d$.  Indeed, the RHS of the uniform limit is integrated with respect to $\vec{t}$ and summed over all appropriate vectors $\vec{k} $ to give
\begin{eqnarray*}
\left[ \int_{\mathbb{R}^d} e^{2 \pi i  s_1}\frac{ s_1 d\vec{s}}{|\vec{s}|^{d+1}} \right] \times \left[ \sum_{\vec{z} \in \mathbb{Z}^d: z_1 =k_1, |z_2|, ..,|z_d| \lesssim k_1} \frac{1}{|Az_1|^{d}}  \right] \simeq \frac{1}{A^d k_1}.
\end{eqnarray*}
Moreover, 
\begin{eqnarray*}
\left[ \frac{Ak_1}{A^{d+1}} \right]  \left[ \sum_{\vec{z} \in \mathbb{Z}^d: z_1 =k_1, |z_2|, ..,|z_d| \lesssim k_1} \frac{1}{|Az_1|^{d}}  \right]  \simeq \left[ \frac{Ak_1}{A^{d+1}} \right] \sum_{\vec{\nu} \in \mathbb{Z}^{d-1}:  |\tilde{\vec{\nu}}| \gtrsim k_1} \frac{1}{\left|\vec{\nu}\right|^{d+1}} \simeq\frac{1}{A^dk_1}.
 \end{eqnarray*}
 To prove the uniform limit,  it suffices to control
 
 \begin{eqnarray*}
&& \left| \int_{-A^2 k_1/2}^{ A^2 k_1 /2}... \int_{-A^2 k_1/2}^{A^2 k_1/2} \left[\prod_{j=1}^n  \phi\left( \theta_x + \vec{\alpha}_j \cdot \vec{t} + \vec{\beta}_j \cdot \frac{\vec{s}}{A \mathfrak{C} k_1}\right) - \prod_{j=1}^n  \phi\left( \theta_x + \vec{\alpha}_j \cdot \vec{t} \right) \right]  e^{2 \pi i s_1} \frac{s_1}{|\vec{s}|^{d+1}} d \vec{s} \right|\\& \leq& \left|  \int_{-A^2 k_1 /2}^{-1} \int_{-A^2 k_1 /2} ^{A^2 k_1/2} ... \int_{-A^2 k_1/2} ^{A^2 k_1/2} \left[  \prod_{j=1}^n \phi\left( \theta_x + \vec{\alpha}_j \cdot \vec{t} + \vec{\beta}_j \cdot \frac{\vec{s}}{A \mathfrak{C}k_1}\right) - \prod_{j=1}^n  \phi\left( \theta_x + \vec{\alpha}_j \cdot \vec{t} \right) \right] e^{2 \pi i s_1} \frac{s_1}{|\vec{s}|^{d+1}} d \vec{s}\right| \\ &+& \left|  \int_{-1}^1  \int_{-A^2 k_1 /2} ^{A^2 k_1/2} ... \int_{-A^2 k_1/2} ^{A^2 k_1/2} \left[  \prod_{j=1}^n \phi\left( \theta_x + \vec{\alpha}_j \cdot \vec{t} + \vec{\beta}_j \cdot \frac{\vec{s}}{A \mathfrak{C}k_1}\right) - \prod_{j=1}^n  \phi\left( \theta_x + \vec{\alpha}_j \cdot \vec{t} \right) \right]e^{2 \pi i s_1} \frac{s_1}{|\vec{s}|^{d+1}} d \vec{s}  \right| \\ &+& \left| \int_1^{A^2 k_1/2} \int_{-A^2 k_1 /2} ^{A^2 k_1/2} ... \int_{-A^2 k_1/2} ^{A^2 k_1/2}  \left[ \prod_{j=1}^n \phi\left( \theta_x + \vec{\alpha}_j \cdot \vec{t} + \vec{\beta}_j \cdot \frac{\vec{s}}{A \mathfrak{C}k_1}\right) - \prod_{j=1}^n  \phi\left( \theta_x + \vec{\alpha}_j \cdot \vec{t} \right) \right] e^{2 \pi i s_1} \frac{s_1}{|\vec{s}|^{d+1}} d \vec{s}  \right|\\ &:=& I + II +III. 
 \end{eqnarray*}
  \subsubsection{Bounding Term $II$}
Because $\phi \in \mathcal{S}(\mathbb{R})$, 

\begin{eqnarray*}
 \left| \prod_{j=1}^n \phi\left( \theta_x + \vec{\alpha}_j \cdot \vec{t} + \vec{\beta}_j \cdot \frac{\vec{s}}{A \mathfrak{C} k_1}\right) - \prod_{j=1}^n  \phi\left( \theta_x + \vec{\alpha}_j \cdot \vec{t} \right) \right| \lesssim \frac{|\vec{s}|}{k_1}
\end{eqnarray*}
uniformly in $t \in \mathbb{R}$. 
Moreover, splitting the integrals in term $II$ by 

\begin{eqnarray*}
\int_{-A^2 k/2}^{A^2 k/2} ... \int_{-A^2 k/2} ^{A^2 k/2} =\left (\int_{-A^2k /2}^{-1} + \int_{-1}^1 + \int_1^{A^2 k_\frac{1}{2}}\right)...\left(\int_{-A^2k /2}^{-1} + \int_{-1}^1 + \int_1^{A^2 k_\frac{1}{2}}\right) 
\end{eqnarray*}
and bringing the modulus inside the integrals reduces the estimate to controlling only two types of terms: 

\begin{eqnarray*}
\frac{1}{k_1} \int_{[-1,1]^d}  \frac{d\vec{s}}{|\vec{s}|^{d-1}} &\lesssim& \frac{1}{k_1}\\ 
\frac{1}{k_1} \int_{[-A^2k_1/2, A^2k_1/2]^{d-1}} \frac{d\vec{s}}{1+|\vec{s}|^{d-1}} &\lesssim &\frac{\log(k_1)}{k_1}.
\end{eqnarray*}
These estimate are trivial and together show the uniform limit of term $II$.

\subsubsection{Bounding Terms $I, III$}
Integrating terms $I$ and $III$ by parts in $s_1$ and then bringing the absolute values crudely inside the resulting integrals yields a quantity $O_A(1/|k_1|)$.  Therefore, the limit

\begin{eqnarray*}
&&\lim_{k_1 \rightarrow \infty} \int_{-A^2\mathfrak{C} k_1/2}^{ A^2\mathfrak{C} k_1 /2}... \int_{-A^2\mathfrak{C} k_1/2}^{A^2 \mathfrak{C}k_1/2} \prod_{j=1}^n \phi\left( \theta_x + \vec{\alpha}_j \cdot \vec{t} + \vec{\beta}_j \cdot \frac{\vec{s}}{A \mathfrak{C}k_1}\right) e^{2 \pi i s_1} \frac{s_1}{|\vec{s}|^{d+1}} d \vec{s} \\&=&  \left[ \prod_{j=1}^n \phi\left( \theta_x + \vec{\alpha}_j \cdot \vec{t} \right)\right] \int_\mathbb{R} ... \int_\mathbb{R} e^{2 \pi i s_1} \frac{s_1}{|\vec{s}|^{d+1}} d\vec{s}
\end{eqnarray*}
is uniform for $t \in [-A/2, A/2]^d$ and the reduction to the lemma follows. To prove the theorem, it therefore suffices to show Lemma \ref{RdL3}. To this end, we consider several cases. 
\subsection{Small Perturbations}

\subsubsection{$\vec{k} : 1 \lesssim k_1 \lesssim N, \vec{l}=\vec{0}$}
By the proceeding calculation, it suffices to bound the integrals for the cases where $\vec{l}= \vec{0}$.
Moreover, without loss of generality, $m_j =  n_0 - \vec{\alpha}_j \cdot \vec{k} + \Delta_j~\forall~j \in \{1, ..., n\}$ where $\max_{1 \leq j \leq n} |\Delta_j| \leq 2d \max_{1 \leq j \leq n } \max_{1\leq k \leq d}\left\{  |\alpha^k_j|, |\beta^k_j| \right\}:= C_{A, B}$.  (The restriction $|\Delta_j| \lesssim_{A, B} 1$ arises because summing over those $\vec{m} \in \mathbb{M}$ which are not contained in this collection yields an acceptable error term.) For each $\vec{k}: 1 \lesssim k_1 \lesssim N$, we may restrict our attention to those $\vec{m}$ in

\begin{eqnarray*}
\mathbb{M}^S_{\vec{k}} &:=&  \left\{ \vec{m} \in \mathbb{M} \left| m_j= n_0 - \vec{\alpha}_j \cdot \vec{k} + \Delta_j  ~~\forall~j \in \{1, ..., n\}~~; ~~\sum_{j=1}^n \#_j \Delta_j \vec{\alpha}_j = \vec{0}~~~;~~ \sup_{1 \leq j \leq n} |\Delta_j| \leq C_{A,B}  \right\}\right. .
\end{eqnarray*}
Indeed, in this case, 
 
 \begin{eqnarray*}
 \sum_{j=1}^n \#_j m_j(x- \vec{\alpha}_j \cdot \vec{t} -\vec{\beta}_j \cdot \vec{s}) &=& \sum_{j=1}^n \#_j (n_0 - \vec{\alpha}_j \cdot \vec{k}+\Delta_j) (x - \vec{\alpha}_j \cdot \vec{t} - \vec{\beta}_j \cdot \vec{s}) \\ &=& \left( \sum_{j=1}^n \#_j \right) n_0 x + s_1 k_1- \left( \sum_{j=1}^n \#_j \Delta_j \vec{\alpha}_j \right) \cdot \vec{t} - \left( \sum_{j=1}^n \#_j \Delta_j \vec{\beta}_j \right) \cdot \vec{s} \\ &=&  \sum_{j=1}^n \#_j n_0x + \sum_{j=1}^n \#_j \Delta_j \theta_x+ s_1 k_1- \left( \sum_{j=1}^n \#_j \Delta_j \vec{\beta}_j \right) \cdot \vec{s}  + Z.
 \end{eqnarray*}
 Hence, setting $C_{\vec{\Delta}}(x) = e^{2 \pi i A\left[ \sum_{j=1}^n \#_j  (n_0 x+\Delta_j \theta_x)\right]}$, 
  
\begin{eqnarray*}
& \int_{-\frac{A}{2}}^{\frac{A}{2}}... \int_{-\frac{A}{2}}^{\frac{A}{2}} \int_{A(k_1-\frac{1}{2})}^{A(k_1+\frac{1}{2})}  ... \int_{A(k_d -\frac{1}{2})}^{A(k_d+\frac{1}{2})} \left[ \prod_{j=1}^n \phi(x-Am_j -\vec{\alpha}_j \cdot \vec{t} - \vec{\beta}_j \cdot \vec{s})\cdot e^{2 \pi i A\#_j m_j (x-\vec{\alpha}_j \cdot \vec{t} - \vec{\beta}_j \cdot \vec{s})} \right]\frac{s_1}{|\vec{s}|^{d+1}} \frac{t_1}{|\vec{t}|^{d+1}} d \vec{t} d \vec{s}   \\ &=C_{\vec{\Delta}}(x) \int_{-\frac{A}{2}}^{\frac{A}{2}}... \int_{-\frac{A}{2}}^{\frac{A}{2}} \left[ \prod_{j=1}^n \phi(\theta_x -\vec{\alpha}_j \cdot \vec{t} - \vec{\beta}_j \cdot \vec{s})\right]  e^{2 \pi i A\mathfrak{C} s_1 k_1} e^{-A \left( \sum_{j=1}^d \#_j \Delta_j \vec{\beta}_j \right) \cdot \vec{s}}  \frac{s_1}{|\vec{s}|^{d+1}} \frac{t_1+Ak_1}{|\vec{t}+A \vec{k}|^{d+1}} d \vec{t} d \vec{s} .
\end{eqnarray*}
As in the proof of Theorem \ref{MT*}, we can control the argument of the above display by the requirement $|\theta_x| \lesssim_{\vec{\alpha}, \vec{\beta}, A} 1$.

 \subsubsection{$\vec{k} : 1 \lesssim k_1 \lesssim N, \vec{l} \not = \vec{0}$}
These terms will be subsumed as error by integration by parts (twice) with respect to $s_1$. Let $\vec{m}(n_0, \vec{k},\vec{l})$ be given component-wise by $m_j(n_0, \vec{k}, \vec{l})=n_0 -\vec{\alpha}_j \cdot \vec{k} - \vec{\beta}_j\cdot \vec{l}+\Delta_j$. Furthermore, assume $\max_{1 \leq j \leq n} |\Delta_j| \lesssim 1$. Then we need to estimate

\begin{eqnarray*}&   \left| \left[ \int_{[-\frac{A^2\mathfrak{C}k_1}{2},\frac{A^2\mathfrak{C}k_1}{2}]^d}  \int_{[-\frac{A}{2}, \frac{A}{2}]^d} e^{2 \pi i  s_1} \frac{(s_1+A^2 \mathfrak{C} k_1 l_1)  \left[ \prod_{j=1}^n \phi\left( \theta_x- A \Delta _j+ \vec{\alpha}_j \cdot \vec{t}+ \vec{\beta} _j \cdot \frac{\vec{s}}{A\mathfrak{C}k_1} \right )\right] }{|\vec{s}+A^2\mathfrak{C} k_1\vec{l}|^{d+1}} \cdot \theta(\vec{t}) \cdot \frac{\cdot t_1 + Ak_1}{|\vec{t}+ A\vec{k}|^{d+1}} d\vec{t} d\vec{s}\right]\right|  \\&\lesssim I + III + III + IV,
\end{eqnarray*} where 
\begin{eqnarray*}
&I =  \left| \left[ \int_{[-\frac{A^2\mathfrak{C}k_1}{2}, \frac{ A^2\mathfrak{C} k_1}{2}]^{d-1}} \int_{[-\frac{A}{2}, \frac{A}{2}]^d}\frac{ A^2\mathfrak{C} k_1l_1 \left[ \prod_{j=1}^n \phi\left( \theta_x-A\Delta_j+ \vec{\alpha}_j \cdot \vec{t}+ \vec{\beta} _j\cdot \left(\frac{A}{2}, \frac{\vec{s}}{A\mathfrak{C} k_1} \right)   \right )\right]}{| (A^2\mathfrak{C}  k_\frac{1}{2}, \vec{s}+A^2\mathfrak{C} k_1 \vec{l})|^{d+1}}   \cdot \theta(\vec{t}) \cdot   \frac{t_1 + Ak_1}{|\vec{t}+ A\vec{k}|^{d+1}} d\vec{t} d\vec{s}\right] \right| \\ &II = \left|  \left[ \int_{[-\frac{ A^2\mathfrak{C} k_1}{2}, \frac{A^2\mathfrak{C}k_1}{2}]^d} \int_{[-\frac{A}{2}, \frac{A}{2}]^d}e^{2 \pi i  s_1} \frac{ \left[ \prod_{j=1}^n \phi\left( \theta_x-A\Delta_j+ \vec{\alpha}_j \cdot \vec{t}+ \vec{\beta} _j \cdot \frac{\vec{s}}{A\mathfrak{C}k_1} \right )\right]}{|\vec{s}+A^2\mathfrak{C}k_1\vec{l}|^{d+1}}  \cdot \theta(\vec{t}) \cdot  \frac{t_1 + Ak_1}{|\vec{t}+ A\vec{k}|^{d+1}} d\vec{t} d\vec{s}\right] \right|  \\ & III =   \left| \left[ \int_{[-\frac{ A^2\mathfrak{C} k_1}{2},\frac{ A^2\mathfrak{C}k_1}{2}]^d} \int_{[-\frac{A}{2},\frac{A}{2}]^d}e^{2 \pi i  s_1} \frac{(s_1+A^2 \mathfrak{C} k_1 l_1)^2 \left[ \prod_{j=1}^n \phi\left( \theta_x-A\Delta_j+ \vec{\alpha}_j \cdot \vec{t}+ \vec{\beta} _j \cdot \frac{\vec{s}}{A\mathfrak{C}k_1} \right )\right] }{|\vec{s}+A^2\mathfrak{C}k_1\vec{l}|^{d+3}}   \cdot \theta(\vec{t}) \cdot \frac{t_1 + Ak_1}{|\vec{t}+ A\vec{k}|^{d+1}} d\vec{t} d\vec{s}\right]\right| \\ &IV= \left|  \left[ \int_{[-\frac{ A^2\mathfrak{C} k_1}{2}, \frac{A^2\mathfrak{C}k_1}{2}]^d} \int_{[-\frac{A}{2}, \frac{A}{2}]^d}e^{2 \pi i  s_1}(s_1 + A^2 \mathfrak{C} k_1 l_1) \frac{ \frac{d}{ds_1}\left[ \prod_{j=1}^n \phi\left( \theta_x-A\Delta_j+ \vec{\alpha}_j \cdot \vec{t}+ \vec{\beta} _j \cdot \frac{\vec{s}}{A\mathfrak{C}k_1} \right )\right]}{|\vec{s}+A^2\mathfrak{C}k_1\vec{l}|^{d+1}} \cdot \theta(\vec{t}) \cdot   \frac{t_1 + Ak_1}{|\vec{t}+ A\vec{k}|^{d+1}} d\vec{t} d\vec{s}\right] \right| .
\end{eqnarray*}
Summing over all vectors $\vec{k} \in \mathbb{Z}^d: k_1$ is fixed and $\vec{m} \in \mathbb{M}$ yields a total contribution at most

\begin{eqnarray*}
C \cdot \sum_{\vec{l} \in \mathbb{Z}^d: |l_1| \lesssim 1} ~~\sum_{\vec{\Delta} \in \mathbb{Z}^d: \left| \vec{\Delta}\right| \lesssim 1}\frac{1}{k^2_1} \frac{1}{1+|\vec{\Delta}|^2} \frac{|l_1|}{1+|\vec{l}|^{d+1}} \lesssim \frac{1}{k_1^2}.
\end{eqnarray*}
The quadratic decay in $k_1$ clearly yields an acceptable error to the main contribution.

\subsection{Large Perturbations}

\subsubsection{$\vec{k} : 1 \leq k_1 \lesssim N, \vec{l} = \vec{0}$}
If $\max_{1 \leq j \leq n} |\Delta_j| \gtrsim 1$, then an arbitrary decay of $\frac{1}{A^{\tilde{N}}}$ may be attached to $\frac{1}{|k_1|}$, which is acceptable whenever $1 \leq k_1 \lesssim N$ upon taking $A$ sufficiently large. Further details are left to the reader. 

\subsubsection{$\vec{k}: k_1 <<0: \vec{l}=\vec{0}$}
Let $\vec{m} \in \mathbb{M}$, $\vec{t} \in \prod_{j=1}^d [ A(k_j-\frac{1}{2}), A(k_j+\frac{1}{2})]$, and $\vec{s} \in \prod_{j=1}^d [A(l_j -\frac{1}{2}), A(l_j+\frac{1}{2})]$. Then there exists $j \in \{1, ..., n\}$ satisfying

\begin{eqnarray*}
|An_0 -A m_j - \vec{\alpha}_j \cdot \vec{t} - \vec{\beta}_j \cdot \vec{s} | \geq \frac{A k_1}{5n \cdot \sup_{1 \leq j \leq n} |\#_j \beta^1_j|} .
\end{eqnarray*}
Indeed, suppose not. Then 

\begin{eqnarray*}
 A \sum_{j=1}^n  \#_j m_j \beta_j^1 &=& \sum_{j=1}^n\#_j (-An_0 +A m_j + \vec{\alpha}_j \cdot \vec{t} + \vec{\beta}_j \cdot \vec{s})\beta_j^1 + \sum_{j=1}^n \#_j ( An_0 - \vec{\alpha}_j \cdot \vec{t} - \vec{\beta}_j \cdot \vec{s}) \beta_j^1\\&\leq &   -A\mathfrak{C} k_1 + \frac{Ak_1}{5}+O(A) <0,
\end{eqnarray*}
which would violate the assumption $\vec{m} \in \mathbb{M}$. 
If one iterates the proceeding uniform argument, it is a simple matter to use $|\theta_x| \gtrsim Ak_1$ to extract $O\left(|k_1|^{-\tilde{N}}\right)$decay, which is summable .

\subsubsection{$\vec{k}: k_1 >> N, \vec{l} = \vec{0}$}
The same argument as in the proceeding section yields $O\left( |k_1 |^{-\tilde{N}} \right)$ decay. Summability then ensures this contribution can be subsumed as error. Further details are left to the reader. 

\subsubsection{$|l_1| >> 1$}
It suffices to observe that for every $\vec{m} \in \mathbb{M}$, $\vec{t} \in \prod_{j=1}^d [ A(l_j-\frac{1}{2}), A(l_j+\frac{1}{2})]$, and $\vec{s} \in \prod_{j=1}^d [A(k_j -\frac{1}{2}), A(k_j+\frac{1}{2})]$, there exists $j \in \{1, ...,n\}$ satisfying 

\begin{eqnarray*}
|An_0 -A m_j - \vec{\alpha}_j \cdot \vec{t} - \vec{\beta}_j \cdot \vec{s} | \geq \frac{A l_1}{5n \cdot \sup_{1 \leq j \leq n} |\#_j \alpha^1_j|} .
\end{eqnarray*}
Indeed, suppose not. Then 

\begin{eqnarray*}
A \left| \sum_{j=1}^n \#_j m_j \alpha_j^1 \right| &=& \left| \sum_{j=1}^n \#_j (-An_0 + A m_j + \vec{\alpha}_j \cdot \vec{t} + \vec{\beta}_j \cdot \vec{s}) \alpha_j^1 +\sum_{j=1}^n \#_j ( An_0 - \vec{\alpha}_j \cdot \vec{t} - \vec{\beta}_j \cdot \vec{s}) \alpha_j^1\right|  \\ &\geq  & A\mathfrak{C} l_1 - A l_1/5 +O(A)>0,
\end{eqnarray*}
which would again violate the assumption $\vec{m} \in \mathbb{M}.$. This extra $O(|l|^{-\tilde{N}})$ decay enables us to use the same integration by parts argument as before to deduce the following estimate $\forall x \in \mathbb{R}$:
\begin{eqnarray*}
\sum_{\vec{z} \in \mathbb{M}: z_1 =k_1}  \sum_{\vec{l} \not = \vec{0}} \sum_{\vec{m} \in \mathbb{M}_S (n_0, \vec{k}, \vec{l})} \left| T^{\vec{z}, \vec{l}} _{} \left( \left\{ \vec{f}_{j, m_j}^{N, A,\#_j} \right\}_{j=1}^n\right)(x) \right| \lesssim \sum_{\vec{l} \in \mathbb{Z}^d: \vec{l} \not = 0} \sum_{\vec{\Delta} \in \mathbb{Z}^d} \frac{1}{ k_1^2} \sum_{1+|\vec{\Delta}|^{\tilde{N}}} \frac{|l_1|}{1+|\vec{l}|^{d+1}} \frac{1}{1+|l_1|^{\tilde{N}}}\lesssim \frac{1}{ k_1^2}
\end{eqnarray*}

\subsection{PART 2: The Irrational Case}
Let $\left\{ \alpha_j^m \right\}_{1 \leq j \leq n ; 1 \leq m \leq d} \in \mathbb{R}^{nd}$ be given by $\alpha_j^m = a_j q_j^m$ and $\beta_j^1 = a_j r_j$ for all components $1 \leq j \leq n, 1 \leq m \leq d$. Furthermore, assume there are $\vec{\#}\in \mathbb{R}^n$ and $\mathfrak{C} >0$ such that 

\begin{eqnarray*}
\sum_{j=1}^n \#_j \alpha_j^m=  \sum_{j=1}^n \#_j \alpha_j^m \alpha_j^l =
\sum_{j=1}^n \#_j \beta_j^1 = 
\sum_{j=1}^n \#_j \left[ \beta_j^1 \right]^2 &=& 0 ~~~(1 \leq m,l \leq d)\\ 
\sum_{j=1}^n  \#_j \alpha_j^m \beta_j^1 &=&\mathfrak{C}  \delta_{1,m}~~~~(1 \leq m \leq d)
\end{eqnarray*}
 with the additional property that $\#_j a_j^2 \in \mathbb{Q}$ for all $1 \leq j \leq n$. By dilating, we shall assume $q_j^m, r_j , \#_j a_j^2 \in \mathbb{Z}$ .  For $A \in \mathbb{Z}^+$ and $j \in \{1, ..., d\},$ construct the functions 

\begin{eqnarray*}
f_j^{N, A, \vec{\#}} (x):= \sum_{-N \leq m \leq N} \phi(x-Aa_j m) e^{2 \pi i A \#_j a_j m x} = \sum_{-N \leq m \leq N} f_{m}^{N, A, \#} (x).
\end{eqnarray*}
As in the rational case, let 

\begin{eqnarray*}
\mathbb{M} :=   \left\{\vec{m} \in \mathbb{Z}^n \cap [-N,N]^n \left|  \sum_{j=1}^n \#_j a_j m_j \beta_j^1 >0; \sum_{j=1}^n \#_j a_jm_j \alpha_j^k =0 ~\forall ~1 \leq k \leq d \right.\right\}
\end{eqnarray*}
and note 
\begin{eqnarray*}
\vec{m} \not \in \mathbb{M} \implies T_{\hat{K}_d(A\cdot) \hat{K}_d(B\cdot)[\{ \vec{\alpha}^m\}, \{\vec{\beta}^m\}, \Phi, \Psi]}\left(\left\{  f_{m_j}^{N, A \#_j} \right\}_{j=1}^n\right) \equiv0
\end{eqnarray*}
for sufficiently large constant A (only depending on the matrices $\left\{ \vec{\alpha}_j\right\}, \left\{ \vec{\beta}_j\right\}$.  As usual, set $S(\vec{a})= \left\{ a_1^{-1}, ..., a_n^{-1} \right\}$, choose $\rho = 1/A^2$, and let $Bohr_{c(\vec{a})N}(S(\vec{a}), \rho) = \left\{ N_1, ..., N_{|Bohr_{c(\vec{a})N}(S(\vec{\alpha}), \rho)} \right\}$. Then $|Bohr_{c(\vec{a})N}(S(\vec{\alpha}), \rho)| \simeq_A N$. Construct
\begin{eqnarray*}
\Omega = \bigcup_{n_0 \in Bohr_{c(\vec{a})N}(S(\vec{a}), \rho)} \left[ An_0- \frac{c_{\left\{ \vec{\alpha}_j \right\}, \left\{ \vec{\beta}_j \right\} }}{A}, An_0 + \frac{c_{\left\{ \vec{\alpha}_j \right\} ,\left\{ \vec{\beta}_j \right\}}}{A}  \right] .
\end{eqnarray*}
Then $|\Omega| \gtrsim_A N$ for small enough implicit constant in the definition of $\tilde{N}$. Consequently, it suffices to show the pointwise estimate

\begin{eqnarray*}
\left| T_{\hat{K}_d(A\cdot) \hat{K}_d(B\cdot)[\{ \vec{\alpha}^m\}, \{\vec{\beta}^m\}, \Phi, \Psi]}\left(\left\{  f_{m_j}^{N, A \#_j} \right\}_{j=1}^n\right)(x)\right| \gtrsim_A \log(N) 1_{\Omega}(x)~\forall x \in \mathbb{R}. 
\end{eqnarray*}

\subsection{Main Contribution: $\vec{k} : k_1 \gtrsim 1; \vec{l}=\vec{0}$}
First note that we are able to approximately solve 

\begin{eqnarray*}
n_0 - a_j m_j(n_0,  \vec{k},\vec{0}
) - \vec{\alpha}_j \cdot \vec{k} \simeq 0~\forall~1 \leq j \leq n.
\end{eqnarray*}
Indeed, for any $n_0 \in Bohr_{c(\vec{a})N}(S(\vec{a}), \rho)$, define $m_j(n_0, \vec{k}, \vec{0}) = \mathcal{N}_j^{n_0} -\vec{q}_j \cdot \vec{k} \in \mathbb{Z}$, where $\mathcal{N}_j^{n_0}$ is the closest integer to $n_0/a_j$. By construction,

\begin{eqnarray*}
|n_0 -a_j m_j(n_0, \vec{k}, \vec{0}) - \vec{\alpha}_j \cdot \vec{k}|= \delta(n_0,j) \lesssim 1/A^2
\end{eqnarray*}
and for sufficiently small $c(\vec{a})$ we have $\mathcal{N}_j^{n_0} - \vec{q}_j \cdot \vec{k} \in [-N, N]$ for all $n_0 \in Bohr_{c(\vec{\alpha})N} (S(\vec{a}), \rho)$ and $|\vec{k}| \lesssim N$. 
Moreover, it is simple to observe for $\vec{m} \in \mathbb{M}$

\begin{eqnarray*}
\sum_{j=1}^n \#_j a_j (\mathcal{N}_j^{n_0} - \vec{q}_j \cdot \vec{k})(x-\vec{\alpha}_j \cdot \vec{t} - \vec{\beta}_j \cdot \vec{s}) &=&  \sum_{j=1}^n  \#_j a_j \mathcal{N}_j^n (x- \vec{\beta}_j \cdot \vec{s}) + \mathfrak{C} k_1 s_1 \\ &=& \sum_{j=1}^n \#_j (n_0 + \delta(n_0,j))(x - \vec{\beta}_j \cdot \vec{s}) + \mathfrak{C} k_1 s_1 \\ &=& \sum_{j=1}^n \#_j n_0 x + \left[ \mathfrak{C} k_1 s_1 - \sum_{j=1}^n \#_j \delta(n_0,j) \vec{\beta}_j \cdot \vec{s} \right].
\end{eqnarray*}
We may run the same $t-$uniform argument as before to deduce an acceptable main contribution. Details are left to the reader. 

\subsubsection{Small Perturbations}
A quick glance at the rational case shows that $\vec{l}=0$ is the only potential difficulty. However, as in the proof of the irrational case of Theorem \ref{MT*}, we may rewrite
\begin{eqnarray*}
\sum_{j=1}^n \#_j a_j \left( \mathcal{N}_j^{n_0} - \vec{q}_j \cdot \vec{k} +\Delta_j \right)x &=& \sum_{j=1}^n \#_j a_j \mathcal{N}_j^{n_0} x+  \sum_{j=1}^n   \#_j a_j \Delta_j \left( n_0 + \theta_x \right) \\&=&  \sum_{j=1}^n \#_j a_j \mathcal{N}_j^{n_0} x+  \sum_{j=1}^n   \#_j a_j \Delta_j \left( a_j \mathcal{N}_j^{n_0} + \delta(n_0, j) + \theta_x \right) \\ &=& \sum_{j=1}^n \#_j a_j \mathcal{N}_j^{n_0} x + \sum_{j=1}^n\#_j a_j \Delta_j \left( a_j \mathcal{N}_j^{n_0} + \delta(n_0, j) + \theta_x \right) + Z, 
\end{eqnarray*}
where $Z \in \mathbb{Z}$. Because $|\delta(n_o, j)| \leq 1/A^2$ and $| \theta_x | \lesssim 1/A$, the argument of the above display is under good control. 
It is worth pointing out that our arguments do not require us to solve or even approximately solve up to some admissible error the relations 

\begin{eqnarray*}
m_j(n_0, \vec{k} , \vec{l}) = n_0 - \vec{\alpha}_j \cdot \vec{k} - \vec{\beta}_j \cdot \vec{l}~~\forall ~1 \leq j \leq n.
\end{eqnarray*}
Indeed, we have seen that when $\vec{l} \not = 0$, the corresponding contribution can be handled by brutally placing mods inside its various constituent pieces. Specifically, we have the estimate 
\begin{eqnarray*}
\sum_{\vec{k} \in \mathbb{Z}^d} \sum_{\vec{l} \not = \vec{0}} \sum_{\vec{m} \in \mathbb{M}_S(n_0, \vec{k}, \vec{l})}\left| T^{\vec{k}, \vec{\l}}_{\hat{K}_d(A\cdot) \hat{K}_d(B\cdot)[\{ \vec{\alpha}^m\}, \{\vec{\beta}^m\}, \Phi, \Psi]}\left(\left\{  f_{m_j}^{N, A \#_j} \right\}_{j=1}^n\right) (x) \right| \lesssim 1~\forall~x \in \mathbb{R}.
\end{eqnarray*}

\subsubsection{Large Perturbations}
This case is handled using the same argument as before, so the details are omitted. 

\end{proof}

\small
\nocite{*}
\bibliographystyle{plain}
\bibliography{NegRes}

 \end{document}